\documentclass{imsart}

\RequirePackage{amsthm,amsmath,amsfonts,amssymb}
\RequirePackage[numbers]{natbib}
\RequirePackage[colorlinks,citecolor=blue,urlcolor=blue]{hyperref}%% uncomment this for coloring bibliography citations and linked URLs
\RequirePackage{graphicx}%% uncomment this for including figures

\usepackage{mathrsfs}
\usepackage{mathptmx,amssymb}
\usepackage{mathtools}
\usepackage{amsmath}
\usepackage{color}
\definecolor{BLUE}{RGB}{41,86,143}
\definecolor{RED}{RGB}{178,31,53}
\usepackage{graphicx}

\usepackage{amssymb}             % AMS Math
\usepackage{accents}
\DeclareSymbolFontAlphabet{\amsmathbb}{AMSb}

\allowdisplaybreaks

\newcommand{\refappendix}[1]{\hyperref[#1]{Appendix}}

\startlocaldefs

\newcommand{\bb}[1]{\mathbb{#1}}
\newcommand{\E}{\ensuremath{ \bb{E} } }
\newcommand{\Q}{\ensuremath{ \bb{Q} } }
\newcommand{\bo}[1]{\ensuremath{{\bf #1 } }}

\newcommand{\re}{\ensuremath{\mathbb{R}}}
\newcommand{\paren}[1]{\ensuremath{\left( #1\right) }}
\newcommand{\F}{\ensuremath{\mathscr{F}}}

\newcommand{\p}{\mathbb{P}}

\newcommand{\z}{\ensuremath{\mathbb{Z}}}
\newcommand{\na}{\ensuremath{\mathbb{N}}}

\newcommand{\mc}[1]{\ensuremath{\mathscr{#1}}}
\newcommand{\G}{\ensuremath{\mc{G}}}
\newcommand{\floor}[1]{\ensuremath{\lfloor #1\rfloor}}

\theoremstyle{plain}
\newtheorem{teo}{Theorem}
\newtheorem{lemma}{Lemma}
\newtheorem{coro}{Corollary}
\newtheorem{propo}{Proposition}

\theoremstyle{definition}

\newcounter{subsubsubsection}[subsubsection]

\usepackage{bm}

\endlocaldefs

\begin{document}

\begin{frontmatter}
\title{Sampling schemes  of  multitype continuous-time Bienaym\'e-Galton-Watson trees and limiting critical genealogies}
%\title{A sample article title with some additional note\thanksref{T1}}
\runtitle{Sampling schemes  of  MBGW trees and limiting genealogies}
%\thankstext{T1}{A sample of additional note to the title.}

\begin{aug}
%%%%%%%%%%%%%%%%%%%%%%%%%%%%%%%%%%%%%%%%%%%%%%%
%% Only one address is permitted per author. %%
%% Only division, organization and e-mail is %%
%% included in the address.                  %%
%% Additional information such as            %%
%% identifying the corresponding author must %%
%% be included in in the Acknowledgments     %%
%% section if necessary.                     %%
\author[A]{\fnms{Osvaldo}~\snm{Angtuncio Hern\'andez}\ead[label=e1]{osvaldo.angtuncio@cimat.mx}\orcid{0009-0008-9599-623X}},
\author[B]{\fnms{Simon C.}~\snm{Harris}\ead[label=e2]{simon.harris@auckland.ac.nz}}
\and
\author[C]{\fnms{Juan Carlos}~\snm{Pardo}\ead[label=e3]{jcpardo@cimat.mx}}

%%%%%%%%%%%%%%%%%%%%%%%%%%%%%%%%%%%%%%%%%%%%%%
%% Addresses                                %%
%%%%%%%%%%%%%%%%%%%%%%%%%%%%%%%%%%%%%%%%%%%%%%
\address[A]{Departamento de Probabilidad y Estad\'istica, Centro de Investigaci\'on en Matem\'aticas\printead[presep={,\ }]{e1,e3}}

\address[B]{University of Auckland, New Zealand\printead[presep={,\ }]{e2}}

\address[C]{Departamento de Probabilidad y Estad\'istica, Centro de Investigaci\'on en Matem\'aticas}

\end{aug}

\begin{abstract}

We study the genealogies of samples of $k$ distinguished particles drawn from the population alive at some fixed time in a  continuous-time multitype Bienaym\'e-Galton-Watson (MBGW) process under two different type dependent sampling schemes:
uniform sampling without replacement within types given a fixed type configuration, and sampling according to type-dependent weights. These schemes complement the uniform sampling at fixed time $T$  considered in Angtuncio, Pardo, C. Harris (2026a) which did not distinguish between sampled types. 
Under each scheme for a fixed sampling time $T$, we characterise the associated times of most recent common ancestors, ancestral offspring distributions, and type-dependent ancestral structure of the sample genealogy.

In addition, under the assumption that the MBGW process is critical with finite second moments, we show that, conditional on survival of the population, a large time limiting sample genealogy emerges which is robust to the sampling scheme used. We identify this universal genealogy to have the same tree structure as the single-type case in C. Harris, Johnston, Roberts (2020), and we describe its ancestral type behaviour over scaled-times - this essentially being decoupled from the tree structure {except} at the times of ancestral splitting events.

% end changes

\end{abstract}



\end{frontmatter}
%%%%%%%%%%%%%%%%%%%%%%%%%%%%%%%%%%%%%%%%%%%%%%
%% Please use \tableofcontents for articles %%
%% with 50 pages and more                   %%
%%%%%%%%%%%%%%%%%%%%%%%%%%%%%%%%%%%%%%%%%%%%%%

\tableofcontents

\section{Introduction and main results}
Continuous-time multitype Bienaym\'e--Galton--Watson (MBGW) processes generalise the classical branching model by allowing individuals to belong to different types, each type determining its reproduction mechanism. 
The multitype framework was first developed in discrete time and is now standard; see, for instance, \cite{MR2047480,MR0408019}. 
When the mean matrix has finite entries and is irreducible, Perron--Frobenius theory describes the long-time behaviour through a dominant eigenvalue, leading to a classification parallel to the single-type case. 
If irreducibility is absent, substantially richer phenomena may occur, see \cite{MR2226887}.

Let $\mathbb{N}=\{1,2,\dots\}$ and $\mathbb{Z}_+=\{0\}\cup\mathbb{N}$. 
Fix $d\in\mathbb{N}$ and rates $\alpha_1,\dots,\alpha_d>0$. 
We consider a $d$-type continuous-time MBGW process initiated by a single ancestor of type $m\in[d]:=\{1,\dots,d\}$. 
Individuals (or particles) evolve independently as follows: a type-$m$ individual lives an exponential time with parameter $\alpha_m$ and, at death, produces offspring according to an independent copy of a random vector 
$\mathbf{L}_m=(L_m^{(1)},\dots,L_m^{(d)})\in\mathbb{Z}_+^d$ with distribution
\[
p_m(\bm\ell)=\mathbb{P}(\mathbf{L}_m=\bm\ell)=\mathbb{P}(\mathbf{L}_m=(\ell_1, \ldots, \ell_d)),
\qquad 
\sum_{\bm\ell\in\mathbb{Z}_+^d}p_m(\bm\ell)=1.
\]
Hence $p_m(\bm\ell)$ represents the probability of producing $\ell_j$ children of type $j$, for all $j\in[d]$. 
All descendants live and reproduce independently following the same rules according to their type.

For $t\ge0$, write $\mathbf{Z}_t=(Z_t^{(1)},\dots,Z_t^{(d)})$, where $Z_t^{(m)}$ is the number of type-$m$ individuals alive at time $t$, and set $N_t=\sum_{m=1}^d Z_t^{(m)}$. The process $\mathbf{Z}=(\mathbf{Z}_t)_{t\ge 0}$ 
is referred to as the MBGW process.
We note that $\mathbf{Z}$ satisfies the \emph{branching property}, that is, the law of the process $\mathbf{Z}$  starting from $\mathbf{x+y}\in\mathbb{Z}_+^d$ coincides with that of the sum of two independent copies of the process $\mathbf{Z}$ started from $\mathbf{x}$ and $\mathbf{y},$ respectively. 
Throughout, we assume that the process $\mathbf{Z}$ is \emph{conservative}, that is, it does not explode almost surely with $\mathbb{P}(N_t<\infty,\,  \forall t>0)=1$. 
A sufficient condition for conservativeness of $\mathbf{Z}$ is  given by Savits \cite{MR282426}, namely
\begin{equation}\label{savits}
\int^1\frac{{\rm d}s}{s-\overline{F}(s)}<\infty, \qquad \textrm{with}\qquad \overline{F}(s)=\max_{m\in \{1,\ldots, d\}}\sum_{n\ge 2} \mathbb{P}\left( \sum_{i=1}^d L^{(i)}_m=n\right) s^{n}.
\end{equation}
See for instance Proposition 2.12 and Remark 2.13 in the aforementioned reference.

However, as we wish to consider the genealogies of individuals in the population, we will require a richer process than $\mathbf{Z}$ which also includes information about how individuals are related to one another. 
We will do this by assigning a unique label to each individual according to the standard Ulam-Harris convention.  
The Ulam-Harris labelling encodes the genealogical structure of a single-type family tree as follows. Let $\mathbb{U}=\bigcup_{n\ge0}\mathbb{N}^n$ with $\mathbb{N}^0=\{\emptyset\}$. 
The initial ancestor is labelled $\emptyset$. If an individual $u\in\mathbb{N}^n$ has $\ell$ offspring, they are labelled $u1,\dots,u\ell\in\mathbb{N}^{n+1}$. 
For example, individual labelled $u=(3,2,4)$ is the fourth child of the second child of the third child of the initial ancestor.
 For $u\in \mathbb{U}$, let  $|u|$ denote the generation of $u$. We write $u\prec v$ if $u$ is a strict ancestor of $v$, and $u\preceq v$ if $u$ is an ancestor of $v$ or $u=v$.

Let $\mathcal{N}_t$ denote the set of labels alive at time $t$. For each individual $u\in\mathcal{N}_t$,  let $C_t^u$ denote its type at time $t$. 
The population size at time $t$ is therefore $N_t={\rm card}\{\mathcal{N}_t\}$  
and the number of individuals of type $m$ at time $t$ is $Z_t^{(m)}=\sum_{u\in\mathcal{N}_{t}} \mathbf{1}_{\{C^u_t=m\}}$. The collection 
$\{(u,C^u_s) : u\in \mathcal{N}_s\}_{s\le t}$
records both the genealogical structure of the population and the types of all individuals up to time $t$. Since this information is sufficient for our purposes, we let
 $(\mathcal{F}_t)_{t\ge 0}$ denote the 
right-continuous filtration generated by the enriched process $\{(u,C^u_t) : u\in \mathcal{N}_t\}_{t\geq 0}$,
that is, $\mathcal{F}_t:=\sigma( (u,C^u_s)_{u\in \mathcal{N}_s} : s\leq t)$.
We assume that the filtration $(\mathcal{F}_t)_{t\ge 0}$ satisfies all the usual hypotheses, and write $\mathcal{F}:=\sigma(\cup_{t\ge 0}\mathcal{F}_t )$.

For each $m\in\{1,\dots,d\}$,  let $\mathbb{P}_{m}$ denote the law of the process $\mathbf{Z}$, including its genealogical information, when the population is initiated by a single ancestor of type $m$  labeled $\emptyset$. 
These probability measures are defined on the filtered probability space $(\Omega, \mathcal{F}, (\mathcal{F}_t)_{t\ge 0})$. Later, we will also consider initial populations consisting of more than one individual. In this case, we specify an initial number 
 $z_m$ of individuals of type $m$. Unless explicitly stated otherwise, we will not keep track of the precise labelling of the initial individuals. Instead, when the initial population size $n_0:=\sum_{m=1}^d z_m$ satisfies $n_0>1$, we assign the labels $1,\dots,n_0$ at random; when $n_0=1$, we use the label $\emptyset$.  For $\mathbf{z}=(z_1,\dots,z_d)\in \mathbb{Z}_+^d$, we  denote by $\mathbb{P}_{\mathbf{z}}$  the law of  $\mathbf{Z}$, including its genealogical information, when the process starts with $z_m$ individuals of type $m$. Finally, letting $(\mathbf{e}_m)_{m\in[d]}$ denote the canonical basis in $\mathbb{R}^d$, where $\mathbf{e}_m$ is a vector in $\mathbb{R}^d$ with value 1 in its $m$-th coordinate and $0$ elsewhere, 
we observe that  $\mathbb{P}_{\mathbf{e}_m}=\mathbb{P}_{m}$, in agreement with the single-ancestor case.

For  vectors ${\bm r},{\bm \ell}\in \mathbb{Z}_+^d$, we use the multi-index notation  $\bo{r^{\bm \ell}}:=r_1^{\ell_1}\cdots r_d^{\ell_d}.$
For $\bo{r}\in [0,1]^d$, let 
\[
\bo{f} (\bo r):=\paren{f_1(\bo{r}),\ldots, f_d(\bo{r})}\in [0,1]^d
\] 
denote the probability generating function associated with the offspring  distribution ${\bo p}$, where for every component $i\in \{1,\ldots,d\}$,
\[
f_i(\bo{r}):=\E_i\left[\prod_{m=1}^dr_m^{L^{(m)}}\right]=\sum_{{\bm \ell}\in \mathbb{Z}^d_+}p_i(\bm \ell )\bo{r^{\bm \ell}}=\sum_{{\bm \ell}\in \mathbb{Z}^d_+}p_i\paren{\ell_1,\ldots, \ell_d}r_1^{\ell_1}\cdots r_d^{\ell_d}.
\]
For simplicity of exposition, we write $[d]:=\{1,\ldots,d\}$. 

Recall that a MBGW branching process is called \emph{simple} if its generating function $\bo{f}$ is such that for all $i\in [d]$, $f_i$ is linear
in each coordinate with no constant term, i.e.
\[
f_i(\bo{r})=p_i(\bo e_1)r_1+\cdots+  p_i(\bo e_d)r_d, \quad \textrm{for}\quad \bo{r}\in [0,1]^d.
\]
In other words, each individual has exactly one offspring possibly of different type and thus the process has a constant number of individuals. 
We wish to exclude such simple cases.

For our purposes, we further require  that each type has a positive probability of eventually producing offspring of every other type.  We express this assumption in terms of the so-called {\it mean matrix} of $\bo Z$. More precisely, we define the mean matrix  $\bo{M}:=(m_{ij})_{i,j\in [d]}$, where $m_{ij}$ denotes the expected number of type $j$ offspring produced by an individual of type $i$, namely
\begin{equation*}
m_{ij}:=\E_{i}\left[L^{(j)}\right]=\sum_{\bm \ell\in \mathbb{Z}^d_+} \ell_jp_i(\bm{\ell}).
\end{equation*}
The process is called irreducible if,
for every $i,j\in [d]$, $m_{ij}^{(n)}>0$ for some $n$,
where $m^{(n)}_{ij}$ is the $(i,j)$-th entry of the matrix ${\bo M}^n$.

From now on, we assume that 
\begin{equation}\label{hyp1}
 \tag{\bf{H}} \textrm{$\bo Z$ is non-simple, conservative and   irreducible.}
\end{equation}

\medskip

%Our aim is twofold. First, we analyse the genealogy of a sample of 
%$k>1$ individuals observed at fixed times under different sampling schemes introduced below. Second, we study the corresponding limiting genealogies as time tends to infinity, under the assumption that the continuous-time MBGW process is critical with finite offspring variance.

%
% Extras from longer abstract moved here...
%

Our aim is twofold. First, we analyse the genealogy of a sample of 
$k\geq2$ individuals observed at fixed times under different sampling schemes introduced below: uniform sampling without replacement; uniform sampling without replacement given a fixed type configuration, and sampling according to type-dependent weights.
For the genealogies obtained by sampling at fixed times $T$ with each scheme, we will characterise the associated times of most recent common ancestors, ancestral offspring distributions, and type-dependent ancestral structure. 
Second, we study the corresponding limiting genealogies conditional on survival as time tends to infinity, under the assumption that the continuous-time MBGW process is critical with finite offspring variance.
We show that, under uniform sampling without replacement  and ignoring types, the sample genealogy converges in distribution, as $T$ tends to infinity, to a universal limiting structure with the same tree topology as Kingman's coalescent but quite different coalescent time dependence - exactly as in the single-type case in \cite{MR4133376}. In particular, the $k-1$ pairwise coalescence times form a mixture of i.i.d. random variables.
Looking in greater detail when types are tracked, we identify the limiting genealogy together with the types of the ancestors involved in each splitting event and their corresponding offspring distributions, as well as the ancestral type behaviour along branches within the tree. In the limit, this type evolution is essentially decoupled from the tree structure, except at the times of splitting events. Importantly, we show that this limiting genealogical type behaviour is robust with respect to the sampling scheme used: the same genealogy also arises under uniform sampling conditional on a fixed type configuration and under type-dependent weighted sampling. 

For uniform sampling without replacement, in our proofs we will make use $k$ distinguished \emph{spine} particles together with a suitable change of measure, as developed by the authors in \cite{AHP-p1}.  In the cases of uniform sampling conditional on fixed types, or weighted sampling according to types, we will develop some alternative changes of measure with appropriate properties tailored according to the particular sampling scheme.

Before introducing our main objects of interest, the associated ancestral processes, we first describe the three sampling schemes that we will consider. We note that in \cite{AHP-p1},  the genealogy of a sample of 
$k>1$ individuals observed at fixed times was studied under uniform sampling.

 In the present work we consider the following sampling schemes for selecting $k$ particles from the population alive at time 
$T$ (see Figure \ref{figDifferentSamlpings}):

	\begin{figure}
\includegraphics[width=.6\textwidth]{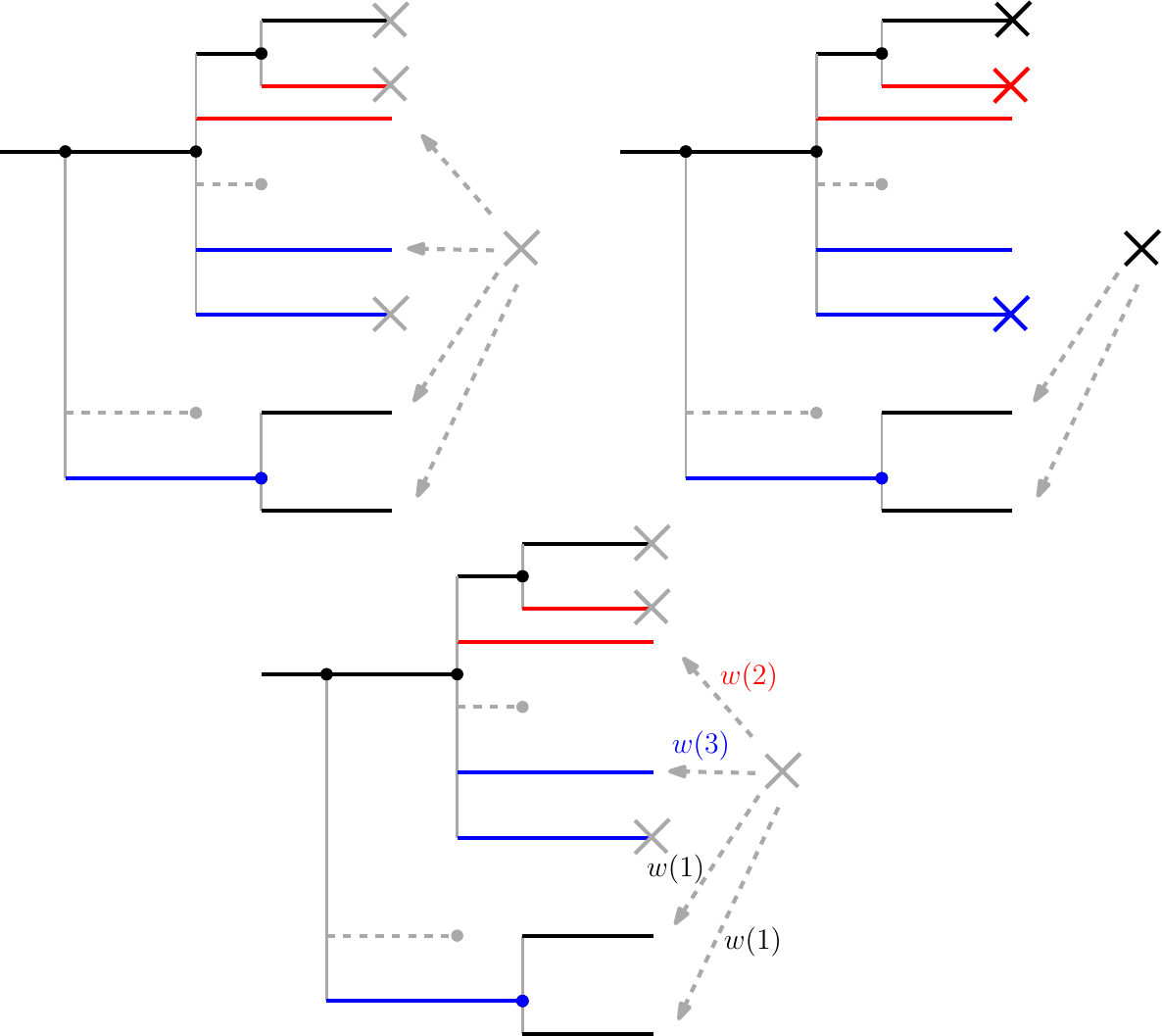}
\caption{
{\it Examples of different sampling schemes.} In all panels, crosses on leaves indicate sampled individuals, and arrows in front of leaves indicate possible choices for the next individual to be sampled. 
 {\it (i) Uniform sampling (top-left):} the next cross selects uniformly among the four remaining leaves.
{\it (ii) Uniform sampling given a fixed type configuration (top-right):} the next Black cross selects uniformly among the two leaves of the same color (Black). {\it (iii) Sampling according to type dependent weights (bottom):} the next cross selects one of the four remaining leaves according to given type weights eg. the Red is sampled next with probability $w(2)/(2w(1)+w(2)+w(3))$.
%In all panels, crosses on leaves indicate sampled individuals, and the arrows in front of leaves indicate possible choices for the next to be sampled. 
}\label{figDifferentSamlpings}	
		\end{figure}
%\end{center}

\begin{itemize}
\item[(i)] \textbf{Sampling uniformly at random without replacement at time $T$.} 
Conditionally on the event  that there are at least $k$ particles alive at time $T > 0$, we select a given sample of $k$ distinct individuals with probability
\[
\frac{1}{N_T(N_T-1)\cdots(N_T-k+1)},
\]
where, recall,  $N_T$ is the number of individuals alive at time $T$.

\item[(ii)] \textbf{Sampling uniformly at random without replacement at time $T$ given a fixed type configuration.} 
Let $\mathbf{c}=(c_1, c_2, \ldots, c_k)$ be a prescribed  vector of types (or colours) for the sample. For each $m\in [d]$, let
\[
 D_m:=\sum_{h=1}^k\bo 1_{\{c_h=m\}}
 \]
  denote the number of sampled individuals of type $m$.
Conditionally on the event that $Z^{(m)}_T\geq D_m$, for all $m\in [d]$, a given sample of $k$ individuals whose $i$-th individual has type $c_i$ is selected, with probability
\[
\prod_{m=1}^d \frac{1}{Z_T^{(m)}(Z_T^{(m)}-1)\cdots(Z_T^{(m)}-D_{m}+1)}.
\]
\item[(iii)]  
\textbf{Sampling according to type dependent weights at time $T$.} 

Let $(w_1, w_2, \ldots, w_d)$ be a vector of nonnegative weights associated with the types. 
Conditional on the population's evolution up to time $T$, we wish to sample each ordered selection $\mathbf{v}=(v_1,\dots, v_k)$ of $k$  distinct individuals alive at time $T$ with a probability proportional to the product of their individual type weights.
That is, given there at least $k$ particles alive, the sample $\mathbf{v}=(v_1,\dots, v_k)\in\mathcal{N}^{(k)}_T$ of individuals with corresponding types $\mathbf{c}=(c_1,\dots,c_k)$ is selected with probability
\begin{equation}\label{multiweightprobability}
\frac{\prod_{i=1}^k w_{c_i}}{\sum_{(u_1,\dots, u_k)\in\mathcal{N}^{(k)}_T}\prod_{h=1}^k w_{c(u_h)} }
=
\frac{\prod_{i=1}^k w_{c_i}}{k! \sum_{\substack{k_1,\dots,k_d \in [k]\\\textrm{s.t. }\sum_{i=1}^d k_i=k}} \prod_{m=1}^d \binom{Z_T^{(m)}}{k_i}w_m^{k_m}}
\end{equation}
where $c(v_h)$ denotes the type (or color) of particle $v_h$, and noting that
summing over all possible choices $\mathbf{u}\in \mathcal{N}^{(k)}_T$ can be decomposed by first deciding how many to pick of each type (i.e. $k_1,\dots,k_d$ where these sum to $k$), then deciding which $k_i$ to choose from the $Z_t^{(i)}$ of type $i$  with these choices all having weight $\prod_{i=1}^d  w_i^{k_i}$, and then taking account of the $k!$ possible orderings the $k$ chosen individuals.

Equivalently, the sample $\mathbf{v}\in \mathcal{N}^{(k)}_T$ is selected with probability 
\begin{equation}
\frac{\prod_{m=1}^d w_m^{D^{(m)}_{\bo v}}}{\sum\binom{k}{k_1,\dots,k_d}\prod_{m=1}^d \left(Z_T^{(m)}(Z_T^{(m)}-1)\cdots(Z_T^{(m)}-k_m+1)\right)w_m^{k_m}}, 
\end{equation}
where $\bo D_{\bo v}=(D_{\bo v}^{(1)}, \cdots, D_{\bo v}^{(d)})$ is the type-degree of $\bo v$ with
\[
D_{\bo v}^{(i)}:=\sum_{h=1}^k\bo 1_{\{c(v_h)=i\}}, \qquad i\in[d],
\]
$\binom{k}{k_1,\dots,k_d} = k!/(k_1!\dots k_d!)$ is the multinomial coefficient, and
the sum in the denominator runs over all type-degree configurations $(k_1,\dots,k_d)$ such that $\sum_{i=1}^d k_i=k$, where $k_i$ represents the number of type $i$ individuals in the selection.

%\item[\correccion{(old version)}]  
%\textbf{Multinomial sampling over types at time $T$.} 
%Let $(w_1, w_2, \ldots, w_d)$ be a vector of  nonnegative weights  associated with the types. Conditionally on the event  that there are at least $k$ particles alive at time $T > 0$, 
%we select a sample of $k$ individuals, denoted by $\bo v=(v_1,\ldots, v_k)$,   such that   $D_{\bo v}^{(i)}\le Z^{(i)}_T$ for all $i \in [d]$, where 
%\[
%D_{\bo v}^{(i)}:=\sum_{h=1}^k\bo 1_{\{c(v_h)=i\}}, \qquad i\in[d],
%\]
%and $c(v_h)$ denotes the type (or color) of particle $v_h$.
%
%The sample $\mathbf{v}$ is selected with probability 
%\begin{equation}
%\frac{\prod_{m=1}^d w_m^{D^{(m)}_{\bo v}}}{\sum\binom{k}{\widetilde{\mathbf{D}}}\prod_{m=1}^d \left(Z_T^{(m)}(Z_T^{(m)}-1)\cdots(Z_T^{(m)}-\widetilde{D}_{m}+1)\right)w_m^{\widetilde{D}_m}}, 
%\end{equation}
%
%%\comentario{Here I erased the extra combinatorial term}
%where  $\bo D_{\bo v}=(D_{\bo v}^{(1)}, \cdots, D_{\bo v}^{(d)})$ is the type-degree vector of $\bo v$, and $\binom{k}{\widetilde{\mathbf{D}}}$ is the multinomial coefficient. The sum in the denominator runs over all admissible type configurations $\widetilde{\mathbf{D}}$ corresponding to $k$-tuples of individuals alive at time $T$.
%
\end{itemize}

The genealogical structure of a continuous-time MBGW process is canonical in the sense that every individual alive at time 
$t$ has a unique ancestor at each earlier time $s<t$. This naturally raises questions about the common ancestry of sampled individuals and the evolution of their types.

Before introducing the ancestral dynamics of a sample, we recall some standard terminology. A {\it block} is a subset $B\subseteq \mathbb{N}$, and we denote  $[k]=\{1,\dots,k\}$. 
 A {\it partition} of a block $B\subseteq \mathbb{N}$ is a countable collection 
 $A=\{A_i, i\in \mathbb{N}\}$ 
 of pairwise disjoint blocks  such that $\cup_{i\in \mathbb{N}}A_i=B$.
 In this setting, it is natural to consider partitions whose blocks carry type information. To this end, we adopt the terminology of \cite{AHP-p1}.  A {\it coloured partition} is  a partition whose blocks are endowed with types (or colours).  Formally, a coloured partition of $[k]$ is a collection $\bo{P}=\{P_1,\ldots, P_d\}$,  where for each type $m\in [d]$,
   $P_m=\{A_{m,1}, A_{m,2},\ldots,A_{m,g_m}\}$
   is a family of $g_m\ge 0$ disjoint blocks of colour 
$m$. Writing $a_{m,q}=|A_{m,q}|$ for the size of block $A_{m,q}$, the blocks across all colours form a partition of $[k]$, and thus
\[
\sum_{m\in [d]} \sum_{q\in[g_m]} a_{m,q}=k.
\]
With a block $A_{m,q}$ in  $P_m$, any two elements $h_1,h_2\in A_{m,q}$ satisfy the equivalence relation
\[
\textrm{ $h_1\sim h_2$ if the marks $h_1$ and $h_2$ follow the same individual of type $m$. }
 \]
The total number of elements of $[k]$ associated with  type $m$ is therefore
\begin{equation}\label{prom}
\overline{a}_m:=a_{m,1}+\cdots +a_{m,g_m}=\sum_{q\in[g_m]} a_{m,q}.
\end{equation}
 Figure \ref{figNotationFirstSplittingTimePartition} illustrates an  explicit example of a coloured partition embedded in a MBGW tree with marks.

\begin{figure}
\includegraphics[width=.6\textwidth]{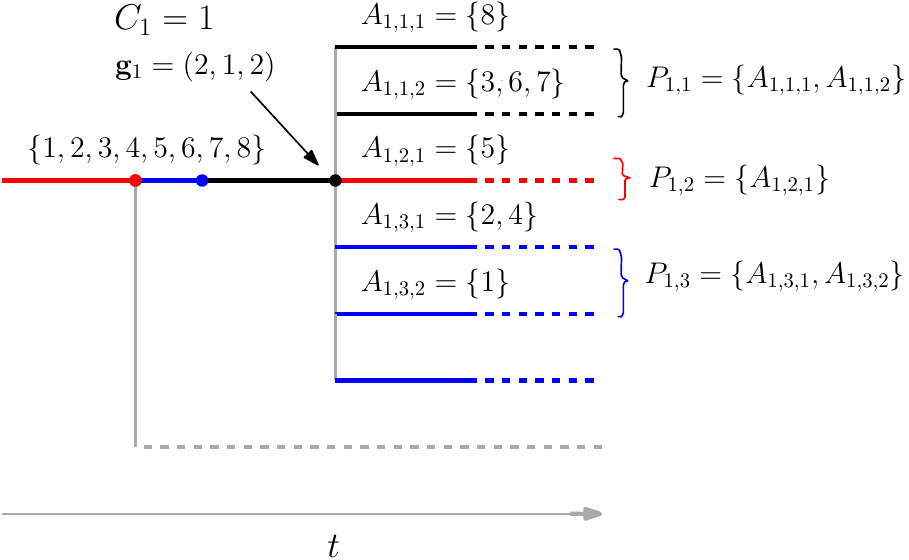}
\caption{Example of a coloured partition after a splitting event with $k=8$ marks, and three types $1,2,3$ represented by colours Black, Red, Blue, respectively. 
The vertex  carrying all marks  at time $t-$, has type 1 (Black) and offspring types $\bm \ell=(2,1,3)$ (that is, 2 Black, 1 Red, and 3 Blue). 
%The colour just before the birth is Black, so $C_1=1$. 
After the splitting event, the coloured partition formed is $\bo{P}=(P_1,P_2,P_3)$. 
In this case there are $g_1=2$ blocks of type 1, namely  $A_{1,1}=\{8\}$ and $A_{1,2}=\{3,6,7\}$, 
with sizes $a_{1,1}=1$ and $a_{1,2}=3$, respectively. Moreover, there are $g_2=1$ blocks of type 2, namely  $A_{2,1}=\{5\}$, of size $a_{2,1}=1$.
There are $g_3=2$ blocks of type 3, which are $A_{3,1}=\{2,4\}$ and $A_{3,2}=\{1\}$ of respective sizes $a_{3,1}=2$ and $a_{3,2}=1$. Finally, the number of marks following type 1, 2 or 3 individuals are $\overline a_1=4,\ \overline a_2=1$ and $\overline a_3=3$. 
}\label{figNotationFirstSplittingTimePartition}
\end{figure}

We label the $k$ sampled particles by the integers $1$ through $k$.  Depending on the  sampling scheme, we associate with the sample an {\it ancestral coloured process} taking values in the space of coloured partitions of $[k]$: (i) $\pi^{(d,k,  T, u)} := (\pi^{(d,k,T, u)}_t)_{t \in [0,T]}$, (ii) $\pi^{(d,k,T, {\bo c})} := (\pi^{(d,k, T, {\bo c})}_t)_{t \in [0,T]}$ or (iii) $\pi^{(d,k,T, {\bo w})} := (\pi^{(d,k, T, {\bo w})}_t)_{t \in [0,T]}$.
  These processes are  defined  by the  rule:
  \begin{center}
$i$ and $j$ belong to the same block of $\pi^{(d,k, T, \cdot)}_t$ with colour $m$ if and only if they descend from the same  ancestor of type $m$ at time $t$.
\end{center}
Equivalently, the particles labelled $i$ and $j$ at time $T$ share a unique common ancestor of type  $m$ at time $t$. 
For brevity, we write $\pi^{(d,k,T,\cdot)}$ to refer generically to any of these processes.  

For all sampling schemes, the initial state consists of a single ancestral block. If the root of the MBGW tree has type $r$, then
\[
\pi^{(d,k,T, \cdot)}_0 =\overline{[k]}_{r}:=\{\varnothing,\ldots,\varnothing, [k],\varnothing, \ldots, \varnothing\}
\]
that is, the block $[k]$ appears in the coordinate corresponding to type $r$, while all other coordinates are empty. This configuration represents a single ancestral block endowed with the unique initial colour (type).

At time $T$, the process reaches the discrete coloured partition into singletons, 
$\pi^{(d,k,T, \cdot)}_T =\{S_m\}_{m\in [d]}$, where $S_{m}$ denotes the collection of singleton blocks corresponding to particles of type $m$.
 As time evolves over $[0,T]$, the process moves through the space of coloured partitions of $[k]$. Blocks may split into several sub-blocks, possibly of different types, and may also change type over time.

Let $M$ denote the number of splitting events required to decompose  the initial block $[k]$ into singletons and write, with a slight abuse of notation,  
$$
0=\tau_0<\tau_1 < \cdots < \tau_M,
$$
for the corresponding split times. These are precisely the times at which the process   $\pi^{(d,k,T,\cdot)}$ experiences  a discontinuity due to a genuine block splitting.

Related constructions in the single-type setting can be found  in Bertoin and Le Gall \cite{MR1771663}, Harris et al. \cite{MR4133376}, Harris et al. \cite{MR4718398}  and Johnston \cite{MR4003147}. In contrast to the single-type setting, the coloured framework introduces an additional source of discontinuity:  changes of colours within blocks. Even when the underlying partition structure remains unchanged, a change in the type of one or more elements induces a jump in the  process. Consequently  $\pi^{(d,k, T, \cdot)}$ is  almost surely right-continuous, with jumps arising from either from block splittings or from colour changes.

\medskip

To state the first main results of this paper, we introduce additional notation.   The {\it coloured topology}, or {\it ancestral coloured sequence, }$\mathcal{T}^{(\cdot)}$  of $\pi^{(d,k, T, \cdot)}$ is defined as the sequence 
\[
\mathcal{T}^{(\cdot)}:=(\mathcal{T}^{(\cdot)}_0,\cdots, \mathcal{T}^{(\cdot)}_{M}) \qquad \textrm{with}\qquad \mathcal{T}^{(\cdot)}_{h}=\pi^{(d,k, T, \cdot)}_{\tau_h}.
\] 
Intuitively, one may view this construction as encoding a multitype tree with edge lengths and $k$ marked leaves. Each branch undergoes  colour changes, recorded through the colours of the blocks of  
 $\pi^{(d,k, T, \cdot)}$.  For each $h = 0, \ldots, M$, the partition $\mathcal{T}^{(\cdot)}_h$ captures both the block structure and the type of each individual alive at time $\tau_h$. For each $h\in [M]$ and $m\in [d]$, let  $G_{h,m}$ denote the number of new blocks of  type $m$ created at time $\tau_h$, and write  $\bo G_{h}=(G_{h,1},\ldots, G_{h,d})$.  Thus, the  sequence of coloured partitions $\{\mathcal{T}_h^{(\cdot)}\}_{h = 0}^{M}$ satisfies the following properties:
\begin{itemize}
\item $\mathcal{T}^{(\cdot)}_0$ is the trivial coloured partition consisting of a single block,
\item $\mathcal{T}^{(\cdot)}_{M}$ is the discrete coloured partition consisting of singletons, and
\item for each $h = 0, \ldots, M - 1$, the coloured partition $\mathcal{T}^{(\cdot)}_{h+1}$ 
consists of 
\begin{itemize}
\item the coloured subpartitions of $\mathcal{T}^{(\cdot)}_h$ that did not split (updated with their current colours at time $\tau_{h+1}$), and
\item the newly created coloured subpartitions, determined by the vector $\bo G_{h+1}$, at time $\tau_{h+1}$.
\end{itemize}
\end{itemize}

 From $(\mathcal{T}^{(\cdot)}_h)_{h = 0}^{M}$,  we extract the {\it ancestral coloured subsequence}
 \[
 \mathcal{P}^{(\cdot)}:=({\mathcal{P}}^{(\cdot)}_0, \ldots, \mathcal{P}^{(\cdot)}_M),
 \]
 which contains only those blocks whose size changes at each splitting event. We set $\mathcal{P}^{(\cdot)}_0=\mathcal{T}^{(\cdot)}_0$ and for $h\ge 1$, $\mathcal{P}^{(\cdot)}_h$ consists precisely of the newly created coloured subpartition at time $\tau_h$.  See Figure \ref{figNotationColouredPartitionSplittingPartitionProcessV2} for an illustration.
	\begin{figure}
\includegraphics[width=.6\textwidth]{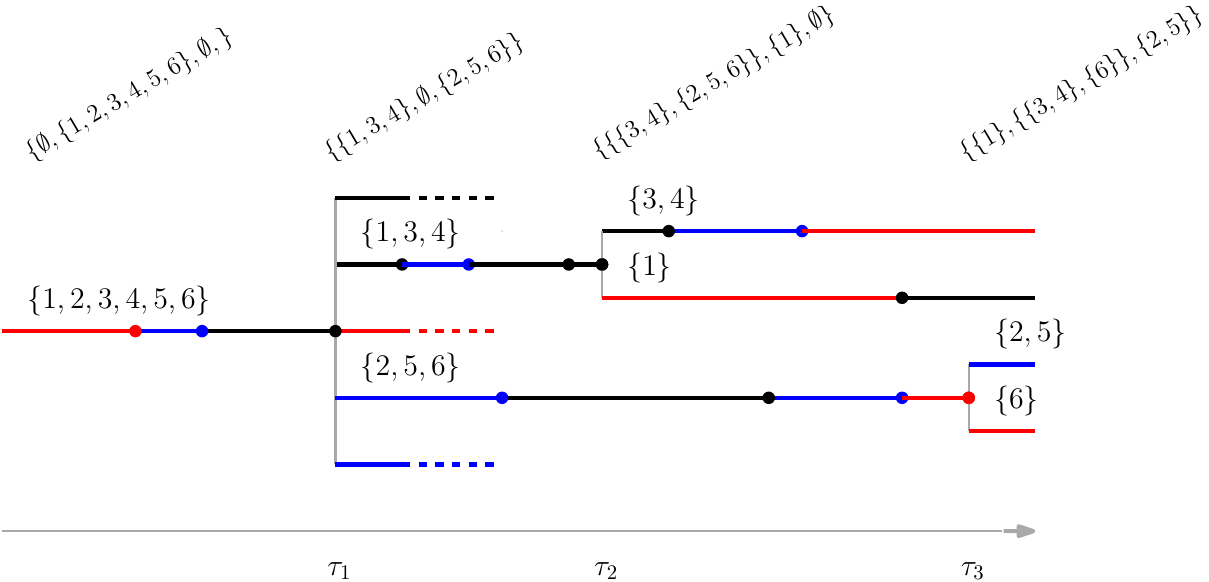}\caption{3-type MBGW tree with 6 spines starting with one individual of type 2 and  its ancestral coloured sequence. Individuals of type 1 are depicted with color Black, type 2 with color Red, and type 3 with color Blue. 
At time zero, the ancestral coloured sequence takes the value $\overline{[6]}_2$, since the unique partition $[6]$ follows an individual type 2. 
We depict this in the three lines above at time 0. 
Just after time $\tau_1$, the ancestral coloured sequence takes the value $\mathcal{T}^{(\cdot)}_{1}=\{\{1,3,4\}\},\emptyset,\{2,5,6\}\}$, since the partition $\{1,3,4\}$ follows a type one individual, and the partition $\{2,5,6\}$ follows a type 3.
Finally, just after time $\tau_2$, the ancestral coloured sequence takes the value $\mathcal{T}^{(\cdot)}_{2}=\{P_{2,1},P_{2,2},\emptyset\}$ where $P_{2,1}=\{\{1\} \}$ and $P_{2,2}=\{\{2,5,6\},\{3\},\{4\} \}$. 
Note that in this case, the ancestral coloured subsequence is $\mathcal{P}^{(\cdot)}_{0}=\{\emptyset,[6],\emptyset\}$, $\mathcal{P}^{(\cdot)}_{1}=(\{1,3,4\},\emptyset,\{2,5,6\})$, $\mathcal{P}^{(\cdot)}_{2}=\{\{3,4\},\{1\},\emptyset\}$, and $\mathcal{P}^{(\cdot)}_{3}=\{\emptyset,\{6\},\{2,5\}\}$. 
 }\label{figNotationColouredPartitionSplittingPartitionProcessV2}	
		\end{figure}
%\end{center}

The tree topology of an ancestral coloured sequence is obtained by discarding the colour information. Specifically, from  $(\mathcal{T}^{(\cdot)}_0,\cdots, \mathcal{T}^{(\cdot)}_{M})$ we derive the corresponding sequence of uncoloured partitions $(\Xi^{(\cdot)}_0, \ldots, \Xi^{(\cdot)}_{M})$, denoted as {\it ancestral sequence} or {\it topology}, by  which encodes the hierarchical splitting structure independently of the types.

We now turn to the main objective of this work, namely the description of the joint distribution of all spine splitting events occurring up to a fixed time horizon $T$. Our goal is to characterise, in a unified framework, the full collection of random objects generated by these events: the splitting times, the types of the individuals involved, the offspring configurations produced at each splitting, and the evolution of the induced ancestral coloured subsequence of the label set $[k]$. In particular, we keep track of how the labels of a uniform sample of size  $k$  are redistributed among descendants through successive spine splittings.

To make this precise, we fix $n\leq k-1$ and  assume that exactly $n$ spine splitting events occur before time $T$, that is,   $M=n$ and $0<\tau_1<\cdots <\tau_n<T$. At each splitting time $\tau_h$, $h\in [n]$, a single individual on the spine, denoted by the label $v(h)$,  gives birth to $\bo L_{v(h)}$ new offspring. We denote by  $C_{\tau_h}$  the type (or colour) of the spine individual $v(h)$ involved in the $h$-th splitting event. 
This reproduction event induces a redistribution of a subset of the $k$ sampled marks among the offspring, which we encode by $\mathcal P_{\tau_h}$. 
 
 The sequence of spine splitting events thus generates an ancestral coloured subsequence $\mathcal P=(\mathcal P_{\tau_h})_{h\in [n]}$ of $[k]$, which records the genealogical evolution of the sampled lineages along the spine. For notational convenience, we write $ C_h:=C_{\tau_h}$ and $\mathcal P_h:=\mathcal P_{\tau_h}$. Each element $\mathcal P_h$ takes values of the form  $ \bo P_h=(P_{h,1},\ldots, P_{h,d}),$
  where   $P_{h,m}$ corresponds to offspring of type $m\in[d]$. More precisely $P_{h,m}=\{A_{h,m,q}\}_{q\in [G_{h,m}] }$, 
is a family of $G_{h,m}$ disjoint blocks, each block representing a group of marks that follow the same descendant of type $m$.  This coloured partition therefore simultaneously encodes the offspring structure at the splitting time and the induced redistribution of the sampled lineages.

 We now formalise the event of interest.  Consider the times $0<t_1<t_2<\cdots <t_n<1$, and for each $h\in [n]$ and offspring configuration
 \[
 {\bo g}_{h}:=(g_{h,1},\ldots, g_{h,d})\leq \bm \ell_{h}:=(\ell_{h,1},\ldots, \ell_{h,d})\in \mathbb{Z}^d_+,
 \]
 together with the partition  $\bo P_h$ with ${\bo g}_{h}$ disjoint blocks and a type $i_h\in [d]$. We  define the event
\begin{equation}\label{deltatn}
\Delta_T(n):=\bigcap_{h\in [n]}\left\{\tau_h\in {\rm d}t_hT,\mathcal{P}_{h}=\bo{P}_h,\bo{L}_{v(h)}=\bm{\ell}_h , C_{h}=i_h, M=n \right\}. 
\end{equation}

Moreover, if $A_{h,m,q}$ is a block of $\mathcal{P}_h$, we denote by  $k_{v(h, m,q)}:=\textrm{card}\{A_{h,m,q}\}$ the number of marks associated with the descendant of $v(h)$ corresponding to the block $A_{h,m,q}$ and having type $m$. We denote this descendant  by $v(h,m,q)$.
We refer to Figure \ref{figNotationJointSplittingTimes} for an illustration of this notation.

\begin{center}
	\begin{figure}
		\includegraphics[width=.4\textwidth]{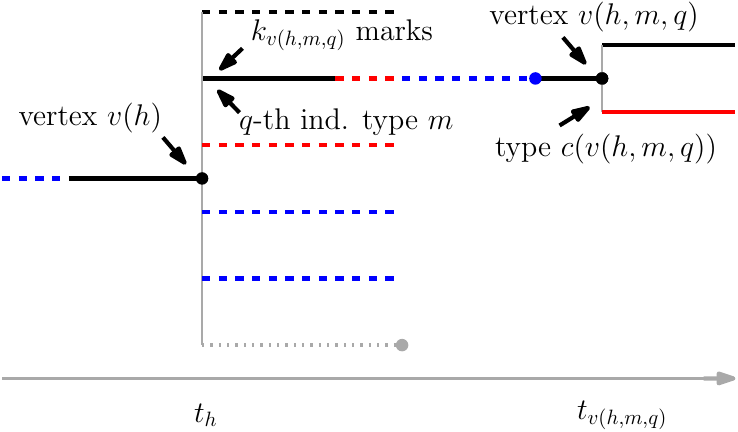}\caption{We show a small window of a multitype tree around a vertex $v(h)$ involved in a spine splitting event at time $t_h$. 
Such a vertex has $L_{v(h)}$ offspring. We represent in gray vertices being born that do not carry marks after time $t_h$. We follow the subtree generated by the $q$-th vertex type $m$ carrying $k_{v(h,m,q)}$ marks (in the picture $q=2$ and $m=1$). Such  vertex, at time $t_{v(h,m,q)}$ is denoted by $v(h,m,q)$, has type $c(v(h,m,q))$ and undergoes a spine splitting event.
}\label{figNotationJointSplittingTimes}
	\end{figure}
\end{center} 
For convenience, we introduce the following notation. For any  $m,j\in \na$, define 
\begin{equation}\label{decfac}
m^{\floor{j}}:=m(m-1)\cdots (m-j+1),
\end{equation}
 the decreasing factorial with the conventions  $m^{\floor{j}}=0$ if $m<j$ and $m^{\floor{0}}=1$.  In particular \[
 N_{t}^{\floor{k}}=N_{t}(N_{t}-1)\cdots (N_{t}-k+1), 
  \]
 with the understanding that  $N_{t}^{\floor{k}}=0$, whenever $N_{t}< k$.

Recall from \cite{AHP-p1} that  sampling $k$ individuals uniformly without replacement at time $T$, conditionally on the event $\{N_T\ge k\}$, induces a probability measure denoted by $ \mathbb{P}^{(k)}_{unif,T,r}$. This measure acts on measurable functionals of the genealogies of $k$-tuples of particles. More precisely, let $f$ be a functional depending on  the ancestral lines of $k$ sampled particles, including their birth and
death times, types, and the number of offspring they have upon death.
Let $\bm{\varsigma}_T=(\varsigma^{(1)}_T, \ldots,\varsigma^{(k)}_T )$ denote a uniform sample without replacement at time $T$, drawn from a MBGW process started from a single  individual of type $r$, and conditioned  on $\{N_T\ge k\}$. Then, accordingly to \cite{AHP-p1},   the probability measure $ \mathbb{P}^{(k)}_{unif,t,r}$ on $\{N_T\ge k\}$ is defined by
\begin{equation}\label{probuniIntro}
 \mathbb{E}^{(k)}_{unif,T,r}\left[f( \bm{\varsigma}_T)\right]=\mathbb{E}_r\left[\frac{1}{N_T^{\floor{k}}}\sum_{\bo{v}\in \mathcal{N}^{(k)}_T} f(\bo{v})\Bigg| N_T\ge k\right]
\end{equation}
where $\mathcal{N}_{t}^{(k)}$ denotes the set of all $k$-tuples of distinct particles alive at time $T$. The prefactor on the right-hand side corresponds to the probability of selecting any given ordered $k$-tuple.

We refer to \cite{AHP-p1}, Section 2.1.1.1, for a formal  construction of $ \mathbb{P}^{(k)}_{unif,T,r}$, and to Section \ref{subsecuniformsampling}  for a concise overview of its properties. For convenience, we recall Theorem 1 from \cite{AHP-p1}.
\begin{teo}\label{propofsplitting1} Suppose that \eqref{hyp1} holds.
For any $k\geq 1$, $T\in \re_+$ and $r\in [d]$, we have
\begin{equation}\label{eqnfsplitting2}
	\begin{split}
\mathbb{P}^{(k)}_{unif,T,r}\paren{\Delta_T(n)}&= \frac{1}{(k-1)!}\frac{1}{\mathbb{P}_{r}\paren{\ N_T\geq k}}\int_0^\infty(e^\phi-1)^{k-1}\Q^{(k),\phi\vec{1}}_{T,r}(\Delta_T(n))\E_{r}\left[N^{\floor{k}}_Te^{-\phi\vec{1}\cdot \bo Z_T}\right]{\rm d}\phi,
\end{split}
\end{equation}
where $\mathbb{Q}^{(k),\phi\mathbf{1}}_{T,r}(\Delta_T(n))$ is given by
\begin{equation}\label{eqnJointSplittingEventWithPartitionChildAndTypeV2}
\begin{split}
\Q^{(k),\phi\vec{1}}_{T,r}(\Delta_T(n))&= \prod_{h=1}^n\prod_{m=1}^d\E_{m}\left[e^{-\phi\vec{1} \cdot \bo Z_{T-t_h}}\right]^{\ell_{h,m} -g_{h,m}}p_{i_h}(\bm \ell_{h})\bm \ell_{h}^{\floor{\bo g_h}}\\
		&\hspace{-1.5cm}  \times 
		\prod_{h=0}^{n-1}\prod_{m=1}^d\prod_{\substack{ q\leq g_{h,m}:\\ k_{v(h,m,q)}\geq 2}}\E_{m}\left[Z^{(c(v(h,m,q)))}_{t_{v(h,m,q)}-t_h} \prod_{j=1}^d\E_{j}\left[e^{-\phi\vec{1}\cdot \bo Z_{T-t_{v(h,m,q)}}}\right]^{ Z^{(j)}_{t_{v(h,m,q)}-t_h}-\delta_{c(v(h,m,q)),j}}\right]\\
	& \hspace{1.5cm} \times
	\frac{\prod_{h=1}^{n}\prod_{m=1}^{d}\E_{m}\left[N_{T-t_{h}}e^{-\phi\vec{1} \cdot \bo Z_{T-t_{h}}}\right]
		^{\#\{ q\leq g_{h,m}:k_{v(h,m,q)}=1\}}
	}{\E_{r}\left[N_T^{\floor{k}}e^{-\phi\vec{1}\cdot \bo Z_{T}}\right]} \prod_{h=1}^n\alpha_{i_{h}} {\rm d}t_h,
	\end{split}
\end{equation}
where  $\delta_{i, m}=1$ if $m=i$, and 0 otherwise. 
\end{teo}
The measure $\Q^{(k),\bm\theta}_{T,r}$, for $\bm \theta\in \mathbb{R}^d_+$, is a probability measure under which the spines are biased so that, at time $T$, they form a uniform sample of $k$ distinct individuals form the population. We again refer to \cite{AHP-p1}, Section 2.1.1.1, or Section \ref{subsecuniformsampling}  for a proper definition and details.

Our first main result extends Theorem \ref{propofsplitting1} to the setting to  uniform sampling without replacement given a fixed type configuration.
Let $\mathbf{c}=(c_1, c_2, \ldots, c_k)$ be fixed. We denote by $\bm{\varsigma}_T=(\varsigma^{(1)}_T, \ldots, \varsigma^{(k)}_T)$ a uniform sample without replacement at time $T$, where the $h$-th sampled individual is of type $c_h$. The sample is drawn from a MBGW process initiated from a single individual of type $r$, and conditioned on $\{Z_T^{(m)} \ge D_m; m \in [d]\}$.
We define the probability measure $\mathbb{P}^{(k)}_{\mathbf{c},T,r}$ on the event $\{Z_T^{(m)} \ge D_m; m \in [d]\}$ by
\begin{equation}\label{probCIntro}
\begin{split}
 \mathbb{E}^{(k)}_{\mathbf{c},T,r}\left[f( \bm{\varsigma}_T)\right]=\mathbb{E}_r\left[ \frac{1}{\prod_{m=1}^d (Z_T^{(m)})^{\floor{D_{m}}}}\sum_{\bo{v}\in \mathcal{N}_{T}^{(\bo c)}} f(\bo{v})\Bigg| Z^{(m)}_T\ge D_m,\, m\in [d]\right], 
 \end{split}
\end{equation}
where $f$ is  as in \eqref{probuniIntro}.

Before stating the result, we recall the notation from Theorem \ref{propofsplitting1} and  introduce additional notation. For each $h\in [n]$, let $\widetilde t_h$ be the first time at which mark $h$ separates from all the other marks. By construction  $\widetilde t_h=t_{v(h',m,q)}$ for some $h'\in [n]$, $m\in [d]$ and $q\in [g_{h',m}]$. We denote by 
 $\widetilde m_h$  the type (or color) of the individual  born at time $\widetilde t_h$, that  carries mark $h$.
This information is equivalently encoded in the sequence of coloured partitions $(\bo P_h;h\in [n])$.
For a configuration $\bo c$,  we define the corresponding    \emph{sample degree}  vector, which records the number of times each type appears in the sample.  Specifically, for  each type $m\in [d]$, the  sample degree of $\bo c$ is  $D_{\bo c,m}:={\rm card}\{h\in [k]: c_h=m  \}$ and the sample degree vector is denoted by $\bo D_{\bo c}:=(D_{\bo c,1},\ldots, D_{\bo c,d})$. Note that  $\sum_{m=1}^dD_{\bo c,m}=k$.
Define $S_{\mathbf{c}} = \{m \in [d] : D_{\mathbf{c},m} \ge 1\}$.

\begin{teo}\label{propofsplitting2Intro} Suppose that \eqref{hyp1} holds.
For any $k\geq 1$, $T\in \re_+$, $r\in [d]$ and $\bo c:=(c_1,\ldots, c_k)\in [d]^k$, we have
\begin{equation}\label{eqpunifcIntro}
\begin{split}
\p^{(k)}_{\mathbf{c},T,r}(\Delta_T(n))&= \frac{1}{\mathbb{P}_{r}\paren{\ \bo Z_T\geq \bo D_{\bo c}}}\left(\prod_{m \in S_{\mathbf{c}}}\frac{1}{(D_{\bo c,m}-1)!}\right)\\
&\hspace{-1cm}\times\int_{\mathbb{R}^{|S_{\mathbf{c}}|}_+} \left(\prod_{m \in S_{\mathbf{c}}} (e^{\phi_m}-1)^{D_{\bo c,m}-1}\right) \Q^{(k),\bm \phi}_{\bo c,T,r}(\Delta_T(n))\E_{r}\left[\prod_{m \in S_{\mathbf{c}}}( Z^{(m)}_T)^{\floor{D_{\bo c,m}}}e^{-\bm \phi\cdot \bo Z_T}\right]{\rm d}\bm \phi
\end{split}
\end{equation}where
%\begin{teo}\label{teoJointLawUnderQSSameLawAsUnderQk}
%We have
\begin{align*}
&\Q^{(k),\bm \phi}_{\bo c,T,r}\Big(\Delta_T(n)\Big)=\Q^{(k),\bm \phi }_{T,r}\Big(\Delta_T(n)\Big)\frac{\E_{r}\left[N^{\floor{k}}_{T}e^{-\bm \phi\cdot \bo Z_{T}}\right]}{\E_{r}[\prod_{m \in S_{\mathbf{c}}}( Z^{(m)}_T)^{\floor{D_{\bo c,m}}}e^{-\bm \phi \cdot \bo Z_{T}}]}\prod_{h=1}^{k}
\frac{\E_{\widetilde m_h}[Z^{(c_h)}_{T-\widetilde t_h}e^{-\bm \phi \cdot \bo Z_{T -\widetilde t_h}}]}{\E_{\widetilde m_h}\left[N_{T -\widetilde t_h}e^{-\bm \phi\cdot \bo Z_{T -\widetilde t_h}}\right]},
\end{align*}and where $\bm \phi$ only has non-zero $m$-th entry  for $m \in S_{\mathbf{c}}$. 
%\end{teo}
\end{teo}

To state our second main result, we introduce the measure induced by  sampling according to type dependent weights. Let $\bm{\varsigma}_T=(\varsigma^{(1)}_T, \ldots,\varsigma^{(k)}_T )$ be a sample at time $T$ drawn according to  \eqref{multiweightprobability}, from a MBGW process started from a single  individual of type $r$, and  conditioned on  $\{N_T\ge k\}$. 
For $\bo v=(v^{(1)},\ldots, v^{(k)})\in \mathcal{N}_T^{(k)}$, with type-degree vector $\bo D_{\bo v}=(D^{(1)}_{\bo v},\ldots, D^{(d)}_{\bo v})$, 
  the probability measure $ \mathbb{P}^{(k)}_{\bo w,T,r}$ on $\{N_T\ge k\}$ is defined as follows
\begin{equation}\label{probWIntro}
 \mathbb{E}^{(k)}_{\bo w,T,r}\left[f( \bm{\varsigma}_T)\right]=\mathbb{E}_r\left[\frac{1}{\sum_{\tilde{\bo v}\in \mathcal{N}^{(k)}_T}\bo w^{\bo D_{\tilde{\bo v}}}}\sum_{\bo{v}\in \mathcal{N}^{(k)}_T} \bo w^{\bo D_{\bo v}}f(\bo{v})\Bigg| N_T\ge k\right],
\end{equation}
%\comentario{Here I erased the extra combinatorial term}
where $f$ is  as in \eqref{probuniIntro}.

\begin{teo}\label{propofsplitting3Intro} Suppose that \eqref{hyp1} holds.
For any $k\geq 1$, $T\in \re_+$, $r\in [d]$ and $\bo w:=(w(1),\ldots,w(d))\in \re^d_+$, we have
\begin{equation}\label{eqnfsplitting3Intro}
\begin{split}
&\mathbb{P}^{(k)}_{\bo w,T,r}\paren{\Delta_T(n)}
= \frac{1}{(k-1)!}\frac{1}{\mathbb{P}_{r}\paren{N_T\geq k}}\\
& \hspace{2cm}\times\int_0^\infty\phi^{k-1}\Q^{(k),\phi\bo w}_{\bo w, T,r}\left[\frac{\big(\bo Z_T\cdot \bo w\big)^k\mathbf{1}_{\Delta_T(n)}}{\sum_{\bo v\in \mathcal{N}^{(k)}_T}\bo w^{\bo D_{\bo v}}}\right]\E_{r}\left[e^{-\phi\bo w\cdot \bo Z_T}\sum_{\bo v\in \mathcal{N}^{(k)}_T}\bo w^{\bo D_{\bo v}}\right]{\rm d}\phi,
\end{split}
\end{equation}
where $\Q^{(k),\phi\bo w}_{\bo w, T,r}$ is a probability measure defined in \eqref{defMeasureQkTrGivenAllInfoS3} which in particular satisfies
\[
\Q^{(k),\phi\bo w}_{\bo w, T,r}\Big(\Delta_T(n)\Big)=\sum_{ \bo c\in \bo [d]^k} \Q^{(k), \phi  \bo w }_{\bo c, T,r}\Big(\Delta_T(n)\Big)\frac{\Q^{(k), \phi\bo w }_{T,r}\paren{c(\bm \varsigma_t)=\bo c}\prod_{h=1}^{k}w(c^{(h)})}{\Q^{(k), \phi\bo w }_{T,r}\left[\bo w^{\bo D_{\bm \varsigma_t}}\right]}.
\]
%\comentario{Here I erased the extra combinatorial term}
\end{teo}

Note that Theorem \ref{propofsplitting2Intro} involves a $d$-dimensional integral, reflecting the fact that types are treated separately. In contrast, in Theorem \ref{propofsplitting3Intro}, the sampling depends explicitly on the type composition of the sample, which complicates the structure of the expression. In particular, the term $\mathbb{Q}^{(k),\boldsymbol{\phi}\mathbf{w}}_{\mathbf{w}, T,r}(\Delta_T(n))$ does not appear directly in \eqref{eqnfsplitting3Intro}.  Nevertheless, we will show below that $\Q^{(k),\bm \phi \bo w }_{\mathbf{w}, T,r}(\Delta_T(n))$ and the corresponding term in \eqref{eqnfsplitting3Intro} have the same limit. 

\medskip

Next, we  extend the limiting results of \cite{MR4133376} to the multitype setting  under the three  sampling  schemes introduced above, building on the preceding results.  
In Theorem \ref{mainresult2} below, we  show that when sampling 
$k>1$ individuals from a population observed at large times in a critical MBGW process with finite second moments, the limiting genealogy of the sample is essentially the same across all three sampling procedures. This limiting genealogy exhibits a universal structure: its tree topology and split times, namely the branching structure and the relationships between nodes, independently of types, do not depend on the offspring distribution. In fact, they coincide with those obtained in the single-type case studied in \cite{MR4133376}.

When types are taken into account, however, new phenomena emerge. The type composition at each branching event is strongly influenced by the offspring distribution. Nevertheless, we prove that, in the limit, the type configuration becomes independent of both the tree topology and the split times, thereby providing a clearer description of the asymptotic behaviour of the process.

Throughout the remainder of this section, we assume that 
\[
m_{ij}=\frac{\partial f_i}{\partial r_j}\paren{\vec{1}}<\infty,
\]
where $\vec{1}:=(1, \ldots, 1)\in \mathbb{Z}_+^d$. This condition, in particular, ensures that the process $\mathbf{Z}$ is conservative. When the mean matrix $\mathbf{M}=(m_{ij})_{i,j\in[d]}$ has finite entries and is irreducible, the matrix $\mathbf{C}=\mathrm{diag}(\bm{\alpha})(\mathbf{M}-\mathbf{I})$ is well defined and irreducible as well. By the Perron-Frobenius theorem, $\mathbf{C}$ admits a dominant real eigenvalue $\rho$ with associated positive left and right eigenvectors $\bm{\eta}$ and $\bm{\xi}$, which we normalise as
\begin{equation}\label{eigenvectors}
\mathbf{1}\cdot \bm{\xi}=1,
\qquad
\bm{\eta}\cdot \bm{\xi}=1.
\end{equation}
Accordingly, the process is said to be subcritical, critical, or supercritical depending on whether $\rho<0$, $\rho=0$, or $\rho>0$, respectively.

We now introduce the extinction probability of $\bo{Z}$ as
\[
\bo{q}:=\lim_{t\to \infty} \bo{q}(t)=\lim_{t\to \infty} \paren{q_1(t),\ldots, q_d(t)},
\]
 where for each $i\in [d]$ and $t>0$, 
\[
q_i(t):=\p_{i}\paren{\bo{Z}_t=\bo{0}}.
\]
It is well known that both $\mathbf{q}$ and $\vec{1}$ are fixed points of the equation $\mathbf{f}(\mathbf{r})=\mathbf{r}$, and that under the assumption that the mean matrix has finite entries and is irreducible, the process becomes extinct almost surely if and only if $\rho\le 0$. 

The asymptotic behaviour of the extinction probability is of particular interest in the critical regime (i.e., when $\rho=0$). To describe this behaviour, we impose the following second-moment condition:\begin{equation}\label{secondmoment}
\zeta:=\sum_{i,j,\ell=1}^d\alpha_i\E_i\left[L^{(j)} \big(L^{(\ell)}-\mathbf{1}_{\{j=\ell\}}\big)\right]\eta_i\xi_j\xi_\ell=\sum_{i,j,\ell=1}^d\alpha_i\frac{\partial^2f_i}{\partial r_j\partial r_\ell}\paren{\vec{1}}\eta_i\xi_j\xi_\ell<\infty.
\end{equation}
This assumption ensures that all second-order moments of the offspring distribution are finite, and in particular implies that the entries of the mean matrix are finite.

Throughout the remainder of this section, we work under the following standing assumption,
\begin{equation}\label{hyprifa}\tag{\textbf{H1}}
\textrm{Assumptions \eqref{hyp1},  \eqref{eigenvectors} and \eqref{secondmoment} hold; and $\bo Z$ is critical.}
\end{equation}

Under these conditions, the survival probability admits the following asymptotic behaviour, and a Yaglom-type limit holds.

\begin{propo}\label{propSophie}
Assume that $\rho=0$ and $\zeta<\infty$. Then,  for all $i\in [d]$,
	\[
	1-q_i(t)\sim \frac{2\xi_i}{\zeta }\frac{1}{t}, \qquad \textrm{as} \quad t\to \infty.
	\]
Moreover, 
\[
\left(\frac{\bo{Z}_t}{t}\ \Big| \bo{Z}_t\ne \bo{0}\right) \xrightarrow[t\to\infty]{(d)}\frac{\zeta}{2}\gamma\,\bm \eta, 
\]
where $\gamma$ is a standard exponential r.v. of mean one, and $\stackrel{(d)}{\to}$ denotes convergence in distribution.
\end{propo}
We refer the reader to Sewastjanow \cite{MR0408019} for this and other classical limit results for continuous time MBGW processes.

%The following theorem establishes that the ancestral process $\Xi^{(d,k,T,\cdot)}$ associated with the sampling schemes described above converges to the universal limiting process $\nu^{( k,2)}$ identified in Harris, et al. \cite{MR4133376}.

%begin{teo}\label{mainresult1} Consider a continuous-time MBGW tree  rooted at a vertex   of type $r$, associated with a MBGW process $\bo Z$ under $\mathbb{P}_r$,  satisfying \eqref{hyprifa}.  Let  $\Xi^{(d,k, T, \cdot)}$ denote the ancestral process  associated with $k$ particles sampled at time $T$ according to any of the following  schemes:  (i) uniformly without replacement, conditional on the event $\{N_T\ge k\}$; (ii) uniformly without replacement given a fixed type configuration, conditional on  $\{ Z^{(m)}_T\geq D_m\},$ for each $m\in [d]\}$; or (iii) multinomial sampling over types, conditional on the event $\{N_T\ge k\}$. Then as $T\to\infty$, 
%\[
%\left(\Xi_{sT}^{(d, k,T, \cdot)}\right)_{s\in [0,1]} {\Longrightarrow} \left(\nu^{( k,2)}_s \right)_{s\in [0, 1]}\qquad \mathrm{as}\quad  T\to \infty,
%\]
%where $ \nu^{(k,2)}:=(\nu^{(k,2)}_s )_{s\in [0, 1]}$ is a  partition-valued process  starting  from  $[k]$.\end{teo}

Our next main result shows that the limits of the probabilities in \eqref{eqnfsplitting2}, \eqref{eqpunifcIntro}, and \eqref{eqnfsplitting3Intro} coincide. This fact is crucial for understanding the limiting genealogy of the ancestral coloured processes $\pi^{(d,k,T,\cdot)}$, and in particular implies that it is identical across all sampling schemes.

 We do not establish full convergence in distribution of these ancestral coloured processes, as our analysis (see Theorems \ref{propofsplitting1}, \ref{propofsplitting2Intro}, and \ref{propofsplitting3Intro}) focuses only on the blocks involved in splitting events, namely the ancestral coloured subsequence. Nevertheless, as we show below, this information is sufficient to characterise the limiting object.

 For  a random variable $\bo L\in \mathbb{Z}^d_+$, we define
\begin{equation}\label{zetayw}
\zeta_i=\alpha_i\eta_i \E_i\left[w({\bo L})\right]
\qquad \mbox{where}\qquad 
%w(\bm \ell):=\sum_{m\ne n}\ell_m\ell_n\xi_{m}\xi_{n}+\sum_{m\in [d]}\ell_m(\ell_{m}-1)\xi_{m}^2
w(\bm \ell):=\sum_{\substack{m< n \\ m,n\in [d]}} 2\ell_m\ell_n\xi_{m}\xi_{n}+\sum_{m\in [d]}\ell_m(\ell_{m}-1)\xi_{m}^2.
%\sum_{(g_1,\ldots,g_d): \sum_m g_m=2} (1+\mathbf{1}_{\{g_i= g_j=1, i\ne j\}})\bm \ell^{\floor{\bo g}}\bm \xi^{\bo g},
\end{equation}
 We observe that $\zeta=\sum_{i=1}^d \zeta_i$.

\begin{teo}\label{mainresult2} Consider a continuous-time MBGW tree  rooted at a vertex   of type $r$, associated with a MBGW process $\bo Z$ under $\mathbb{P}_r$,  satisfying \eqref{hyprifa}.  Then 
\begin{equation}\label{onlybin}
\lim_{T\to \infty}\mathbb{P}^{(k)}_{unif,T,r}\left(M=k-1, \sum_{m=1}^d G_{h, m}=2\quad \textrm{for all}\quad h\in [k-1]\right)=1,
\end{equation}
in other words, in the limit there are only binary splittings.  Moreover, let $\mathcal P=(\beta_0, \ldots, \beta_{k-1})$ be any binary ancestral coloured subsequence,  in other words if $\bo g_h=(g_{h,1},\ldots, g_{h,d})$ with $\sum_{m=1}^dg_{h,m}=2$, and where $g_{h,m}$ is the value taken by the r.v. $G_{h,m}$,  then 
\begin{equation}\label{limitsigualdad}
\lim_{T\to \infty}\mathbb{P}^{(k)}_{unif,T,r}\paren{\Delta_T(k-1)}=\lim_{T\to \infty}\mathbb{P}^{(k)}_{\bo c,T,r}\paren{\Delta_T(k-1)}=\lim_{T\to \infty}\mathbb{P}^{(k)}_{\bo w,T,r}\paren{\Delta_T(k-1)},
\end{equation}
where
\begin{equation}\label{dk2withcolors}
\begin{split}
\lim_{T\to \infty}\mathbb{P}^{(k)}_{unif,T,r}\paren{\Delta_T(k-1)}&=\left(\prod_{h=1}^{k-1}\frac{\zeta_{i_h}}{\zeta}\frac{p_{i_h}(\bm \ell_{h})w(\bm \ell_{h})}{\mathbb{E}_{i_h}\left[w(\bo L)\right]}\frac{\bm \ell_{h}^{\floor{\bo g_h}}\bm \xi^{\bo g_h}}{w(\bm \ell_{h})}\right)\\
&\hspace{1cm}\times\frac{2^{k-1}}{(k-1)!}\int_0^\infty  \frac{y^{k-1}}{(1+y)^2} \prod_{h=1}^{k-1}\frac{1}{(1+(1-t_h)y)^2} {\rm d} t_h{\rm d}y,
\end{split}
\end{equation}
with 
${\bm \ell}^{\floor{\bo g_h}}=\prod_{m=1}^d (\ell^{(m)})^{\floor{g_{h,m}}}$ and  ${\bm \xi}^{\bo g_h}=\prod_{m=1}^d \xi_m^{g_{h,m}}.$
\end{teo}
Observe first that the splitting probabilities in the limiting genealogy are independent of the type of the root. Moreover, the limiting genealogy is binary, as established in \eqref{onlybin}, so that each splitting event produces exactly two offspring blocks, which may carry identical or distinct colours.

The integral term on the right-hand side of \eqref{dk2withcolors} fully determines the genealogical structure, namely the tree topology together with the splitting times. Remarkably, this term coincides with  the joint law of the $k-1$ splitting times in the universal binary genealogy  arising in the single-type BGW case; see, for instance, \cite{MR4133376}, Theorem 3, \cite{MR4718398}, Theorem 1.2, or \cite{MR4003147}, equation (3.15).
 We denote this limiting process by $\nu^{(1,k,2)}$. Its splitting times $0<\tau_1<\cdots<\tau_{k-1}<1$ have joint density given by
\begin{align} \label{eq:binarysplits}
f_k(t_1,\ldots,t_{k-1})
= k! \int_0^\infty \left( \prod_{i=1}^{k-1}\frac{\varphi}{\bigl(1+\varphi(1-t_i)\bigr)^2} \right)\frac{1}{(1+\varphi)^2}\,{\rm d}\varphi{\rm d}t_1\cdots{\rm d}t_{k-1}.
\end{align}
The prefactor $2^{k-1}/(k-1)!$ in  \eqref{dk2withcolors}  accounts for the number of admissible ranked binary tree topologies, recalling that there are $k!(k-1)!/2^{k-1}$ such trees with $k$ labelled leaves; see, e.g., \cite{MR727923}. This reflects the intrinsic combinatorial complexity of the genealogy.

The ancestral coloured subsequence $(\beta_0,\ldots,\beta_{k-1})$ then enriches this genealogical structure by specifying, at each splitting event, which block splits, how it splits, and how colours are assigned to the two offspring blocks. In this way, it encodes the type evolution along the genealogy.

Since only ancestral coloured subsequences are observed, one cannot deduce convergence in distribution of the full genealogies of a MBGW tree to a limiting genealogy, in contrast with the single-type setting. Nevertheless, the preceding result identifies the limiting object. We expect that full convergence can still be established, as discussed below Theorem \ref{propoColorOfTheSpineIsMCPUnifLimit}, although additional arguments are required.

To make the picture more explicit, recall that in the single-type case the process $\nu^{(1,k,2)}$ has only binary splittings in which each jump splits a single block into two. If the current state consists of $i$ blocks with sizes $a_1,\dots,a_i$, then the next split occurs in block $j$ with probability $(a_j-1)/(k-i)$. Conditionally on splitting a block of size $a$, the offspring sizes are $(U,a-U)$, where $U$ is uniform on ${1,\dots,a-1}$, independently of the splitting times. Viewed backward in time, $\nu^{(1,k,2)}$ induces the same random tree topology as Kingman's coalescent \cite{MR671034}, in the sense that each pair of blocks merges with equal probability at each step. The same limiting genealogy also arises for Galton-Watson processes in varying environments; see \cite{boenkost2024genealogynearlycriticalbranching,MR4840489}.

Our result shows that, in the multitype setting, the limiting object can be interpreted as a coloured version of this universal binary genealogy. The genealogical structure (topology and splitting times) is governed entirely by $\nu^{(1,k,2)}$, while the types evolve independently along the tree via an additional sampling mechanism.

More precisely, at each splitting event, a block of size $j$ splits into two blocks of sizes $(\mathcal U, j-\mathcal U)$, where $\mathcal U$ is uniform on ${1,\dots,j-1}$. The type of the ancestral lineage immediately before the split is $i$ with probability $\zeta_i/\zeta$. Conditional on this type, the offspring configuration $\bm{\ell}$ is sampled according to the size-biased law
	\[
	\frac{p_i(\bm \ell)w(\bm \ell)}{\E_i\left[w({\bo L})\right]}.
	\]
 Given $\bm{\ell}$, the two offspring that carry the descendant lineages are then selected without replacement with probabilities proportional to the weights $\bm{\xi}$. In particular, the probability that these two offspring have ordered types $(m,n)$ is
\[
\begin{cases}
\frac{\ell_m\ell_n\xi_m\xi_n}{w(\bm \ell)}&\mbox{if $m\neq n$} \\
\frac{\ell_m(\ell_m-1)\xi_m^2}{w(\bm \ell)} &\mbox{if $m= n$}. \end{cases}
\]
Thus, conditionally on the offspring configuration, the two descendant lineages are sampled without replacement from the $\bm{\ell}$ children according to the weights $\bm{\xi}$. This provides a transparent interpretation: the genealogy is universal and binary, while the type evolution arises from a weighted sampling procedure along its branches.

Finally, note that the preceding result describes the splitting times, the tree topology, and the colours immediately before and after each splitting event, but not the evolution of colours between successive splitting times. This is addressed in the next result, where we analyse the colour dynamics prior to the first splitting event. By the Markov branching property, analogous descriptions hold between successive splitting times, allowing one to reconstruct the full colour evolution along the genealogical lines. 

For completeness, we also consider the process $(C_{h})_{h \in [k-1]} $, which highlights the distinction between the colour configuration immediately before a split and the evolution between splitting events. To this end, we introduce $c(\varsigma_{s}^{(1)})$  to denote the colour (or type) of the vertex in the sample $\varsigma^{(1)}$ at time $s$.

\begin{teo}\label{propoColorOfTheSpineIsMCPUnifLimit} Consider a continuous-time MBGW tree  rooted at a vertex   of type $r$, associated with a MBGW process $\bo Z$ under $\mathbb{P}_r$,  satisfying \eqref{hyprifa}. 
For any $n\in \na$ and $r\in [d]$, consider $0<t_1<\cdots <t_n<1$ and $(i_h)_{h\in [n]}\in [d]^{n}$. 
Let $\mathcal{D}_m:=\#\{h\in [n]:i_h=m \}$ for any $m\in [d]$, and define $\overline{\mathcal{D}}:=(\mathcal{D}_1,\ldots, \mathcal{D}_d)$. 
Then, we have
\begin{equation}\label{inhoMarkovsplit}
\lim_{T\to \infty}\mathbb{P}^{(k)}_{unif,T,r}\Big(c(\varsigma_{t_1 T}^{(1)})=i_1,\ldots, c(\varsigma_{t_n T}^{(1)})=i_n,\tau_1>t_n T\Big)
=\bm \eta^{\overline{\mathcal{D}}}\bm \xi^{\overline{\mathcal{D}}}\mathbb{E}\left[ \left(\frac{1-t_n}{1-t_n W}\right)^{k-1}\right] ,
\end{equation}
where $W$ denotes a Beta r.v. with parameters $(k,1)$. In particular, for $t\in(0,1)$,
\begin{equation}\label{firstsplitinglimit1}
\lim_{T\to \infty}\mathbb{P}^{(k)}_{unif,T,r}\Big(\tau_1> tT \Big) 
= \E\left[ \left(\frac{1-t}{1-t W}\right)^{k-1}\right]
\end{equation}
and
\[
\lim_{T\to \infty}\mathbb{P}^{(k)}_{unif,T,r} \left(\left.c(\varsigma_{t_1 T}^{(1)})=i_1,\ldots, c(\varsigma_{t_n T}^{(1)})=i_n\,\right|\,\tau_1> t_nT\right)
=\bm \eta^{\overline{\mathcal{D}}}\bm \xi^{\overline{\mathcal{D}}}.
\]
Moreover, for any $(i_h)_{h\in [k-1]}\in [d]^{k-1}$ 
\[
\lim_{T\to \infty}\mathbb{P}^{(k)}_{unif,T,r} \Big( C_{1}=i_1, \cdots, C_{k-1}=i_{k-1}\Big)
=\prod_{h=1}^{k-1}\frac{\zeta_{i_h}}{\zeta}.
\]
\end{teo}
We emphasise that the laws described in the previous proposition do not depend on the specific values of $(t_1,\ldots, t_{n-1})$. Moreover, the distribution of the most recent common ancestor of the sample, $\tau_1$, coincides with that given in Theorem~3.2 of \cite{MR3570095}. 

\medskip

With all this information at hand, we can now provide a complete description of the limiting genealogy of an MBGW tree satisfying \eqref{hyprifa}. Let us denote this limiting genealogical process by $ \nu^{(d,k,2)}$, and write
$\mathbb{P}^{(d,k,2)}_r$ for  its  law  started from $\overline{[k]}_r$. The process $\nu^{(d,k,2)}$ encodes the ancestral coloured structure of a binary branching tree and evolves as follows:
\begin{enumerate}
\item[i)] The process starts from $\overline{[k]}_r$ under $\mathbb{P}^{(d,k,2)}_r$. 
\item[ii)] A particle carrying $h\in [k]$ spines evolves its subtree forward in time independently of the rest of the process (branching Markov property).
\item[iii)] A particle carrying $h$ spines and alive at time $s\in (0,1)$, has colour $i$ at time $t\in (s,1)$, conditional on no branching occurring in the interval $(s,t]$, with probability $\eta_i\xi_i. $
\item[iv)] A particle carrying $h$ spines and alive at time $s\in (0,1)$, undergoes its first branching event at a time $t\in (s,1)$  according to  the law described in  \eqref{firstsplitinglimit1} (with $k$  replaced by $h$).
\item[v)] A particle carrying $h$ spines has type $i$ at the splitting time $\tau_j-$, $j\in [k-1]$, with probability $\zeta_i/\zeta.$ 
\item[vi)] Given that a particle type $i\in [d]$ carrying $h\in [k]$ spines at time $t\in (0,T)$ branches,  the $h$ spines split into $\bo g=(g_1, \ldots, g_d)$ groups such that $\sum_{m=1}^dg_{m}=2$ with probability
\[
\frac{\mathbb{E}_{i}\left[\bo L^{\floor{\bo g}}\right]{\bm \xi}^{\bo g}}{\mathbb{E}_{i}\left[w(\bo L)\right]}.
\]
\end{enumerate}
Finally, we  provide an intuitive probabilistic construction of $\nu^{(d,k,2)}$, inspired by Aldous to Kingman's coalescent, and extending the construction of $\nu^{(1,k,2)}$ given in \cite{MR4133376}, Theorem 4. We omit the proof, as   it follows directly  from the latter by incorporating an additional  colouring,  which is independent of both the tree topology and the splitting times.

Let $X_1, X_2, \ldots$ be i.i.d.\ random variables on $(0,\infty)$ with density $(1+x)^{-2}$. Define $M_k := \max_{1 \leq i \leq k} X_i, $
and let $I$ be such that $X_I = M_k$. For $1 \leq i \leq k$, set
\[
T_i := 1 - \frac{X_i}{M_k}.
\]
Then similarly as in \cite{MR4133376}, Theorem 4, we have that $(T_1, \ldots, T_{I-1}, T_{I+1}, \ldots, T_k)$
has the same distribution as $(\tau_1, \ldots, \tau^{\,k}_{k-1})$ under $\mathbb{P}^{(d,k,2)}_r$ (its density is given by \eqref{eq:binarysplits}).

Moreover, the ancestral tree with types generated by $k$ uniformly chosen particles admits the following construction. Let $(U_i)_{i\ge 1}$ and $(Y_i)_{i\ge 1}$ be independent sequences such that:
\begin{itemize}
    \item $(U_i)_{i\ge 1}$ are i.i.d.\ uniform random variables on $[0,1]$,
    \item $(Y_i)_{i\ge 1}$ are i.i.d.\ random variables taking values in $[d]^3$ with distribution
\[
\mathbb{P}\big((Y^{(-)}_i, Y^{(+,1)}_i, Y^{(+,2)}_i) = (j,m,n)\big)
= \frac{\zeta_j}{\zeta}
\frac{\mathbb{E}_{j}\!\left[\bo L^{\lfloor \bo g \rfloor}\right]\bm{\xi}^{\bo g}}{\mathbb{E}_{j}\!\left[w(\bo L)\right]},
\]
where $\bo g=(g_1, \ldots, g_d)$ satisfies $g_m + g_n = 2$.
\end{itemize}
The tree is constructed in the unit square as follows.
\begin{itemize}
\item[i)]{\it Vertical segments.}  
For each $1 \leq i \leq k$, draw a vertical segment from $(U_i,0)$ to $(U_i, 1 - T_i)$. These segments represent the initial branches of the tree.

\item[ii)]{\it Horizontal connections and branching points.}  For each $i \in \{1, \dots, k\} \setminus \{I\}$, draw a horizontal segment starting from $(U_i, 1 - T_i)$ toward $(U_I, 1 - T_i)$, stopping at its first intersection with another vertical segment. Each such intersection point represents a branching (or coalescence) event.
At each intersection point corresponding to index $i$, assign the type $Y^{(-)}_i$. At the starting point $(U_i, 1 - T_i)$, assign the type $Y^{(+,2)}_i$, and immediately after the intersection assign the type $Y^{(+,1)}_i$ along the outgoing segment.

\item[iii)]{\it Colour evolution along segments.}  
Conditionally on the tree structure and the types assigned at intersection points, the colours along each segment are described by a collection of independent color processes $(C^{(h)})_{1 \leq h \leq k}$. The law of $C^{(h)}$ is given in Theorem~\ref{propoColorOfTheSpineIsMCPUnifLimit}, where $h$ denotes the number of marks carried by the segment and its lifespan runs  between two consecutive branching points (or between  a branching point and an endpoint if the branch is external). More precisely, consider a segment between two consecutive intersection points (or between  a branching point and an endpoint). If a branching point carries $a$ marks and splits into two groups of sizes $a_{1}$ and $a_{2}$, then along the corresponding descendant segments we run independent colour processes $C^{(a_{j,1})}$ and $C^{(a_{j,2})}$, starting from the types prescribed by the corresponding components of $Y_j$. Each process evolves up to the next intersection point (if any).
\end{itemize}
An illustration of this construction is provided in Figure \ref{figForwardConstructionUsingUniformRVs}.
\begin{center}
		\begin{figure}
			\includegraphics[width=.8\textwidth]{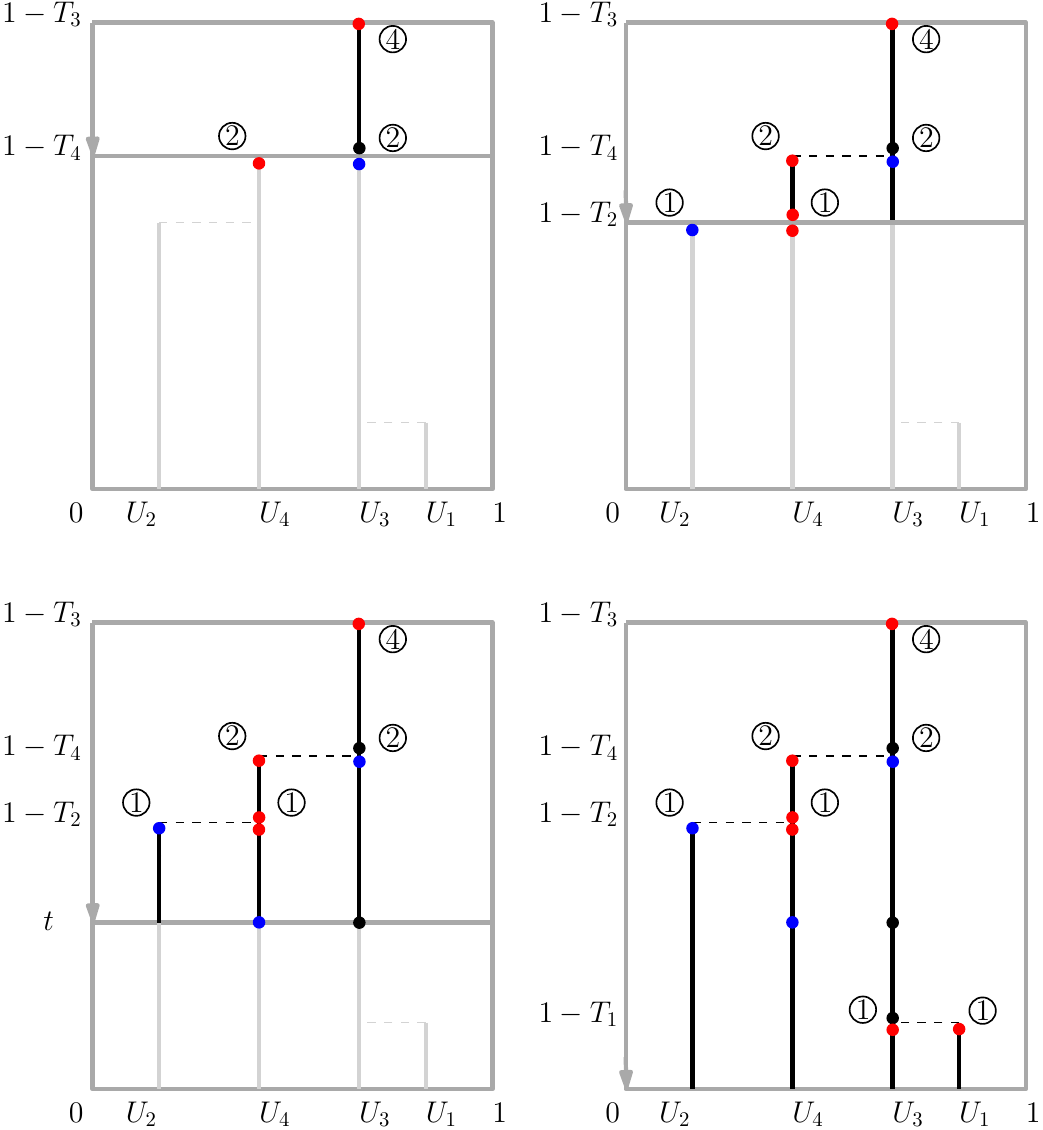}\caption{Probabilistic construction of $\nu^{(3,4,2)}$. In the first picture, we sample $U_1, \ldots, U_4$ i.i.d. uniform r.v.'s and  let  $X_3=\max\{X_1,\ldots,X_4\}$. We then  draw a vertical  segmente from  $(U_3, 1)$ to $(U_3, 1-T_4)$, carrying 4 marks. This initial segment starts with type 2 (Red) and evolves according to  $C^{(4)}$. At time $1-T_4$, the marks splits into 2 blocks of  size 2. Just before the split, we sample a  r.v. $Y_1=(1,3, 2)$: the parent vertex takes colour 1 (Black), while the two offspring segments start with colours 3 (Blue) and 2. From this point on, each segment evolves independently according to a copy of $C^{(2)}$, with its corresponding initial colour and lifetime rescaled to the length of the segment up to the next split time. The construction proceeds recursively in the same way at each split time, until all segments corresponding to the singletons reach the bottom.  In the third picture at time $t$, the colours along the ancestral tree which are obtained from independent copies of  $C^{(1)}$ and  $C^{(2)}$, started from the appropriate colours and run for their respective rescaled lifetimes. }\label{figForwardConstructionUsingUniformRVs}	
		\end{figure}
\end{center}

We briefly outline the strategy used to establish our main results. The first two results, Theorems \ref{propofsplitting2Intro} and \ref{propofsplitting3Intro}, are derived via suitable changes of measure that relate the relevant probabilities to those obtained under uniform sampling in Theorem \ref{propofsplitting1}, the main result of \cite{AHP-p1}.

The central idea is to introduce a collection of distinguished lineages, or \emph{spines}, evolving within a continuous-time MBGW tree initiated from a single individual of type~$r$, and adapted to the chosen sampling scheme. Building on the spine techniques developed by Harris et al.~\cite{AHP-p1}, we construct a change of measure $\Q^{(k),\bm{\theta}}_{\cdot, T,r}$ under which the spines are biased so that, at time~$T$, they form a sample of $k$ individuals consistent with the prescribed sampling procedure. Simultaneously, the population is reweighted by the size-biased functional
\[
\bo z \longmapsto n(n-1)\cdots(n-k+1)\, \mathrm e^{-\bm{\theta}\cdot \bo z},
\]
where $\bo z=(n_1,\ldots,n_d)$ denotes the vector of type counts with total population size 
$n=\sum_{m=1}^d n_m$. 
In the single type setting and uniform sampling of Harris et al.~\cite{MR4133376}, this corresponds to the special case where  $\bm{\theta}$ reduces to the scalar  $\theta=0$. 

The vector $\bm{\theta}$ acts as an exponential discounting parameter, regulating the growth of the tree and enabling an interpretation in terms of sampling from a $k$-fold size-biased multitype tree, even in the absence of higher-order moment assumptions. Related exponential tilting techniques in the single type setting were developed in~\cite{MR4718398}.

Under the change of measure $\Q^{(k),\bm{\theta}}_{\cdot, T,r}$, the model becomes significantly more tractable. Although the formal definitions of these measures are technical and deferred to \eqref{defMeasureQkTrGivenAllInfo}, \eqref{defMeasureQkTrGivenAllInfoS}, and \eqref{defMeasureQkTrGivenAllInfoS3}, an intuitive understanding of the resulting dynamics is essential for the proofs of Theorems \ref{propofsplitting2Intro} and \ref{propofsplitting3Intro}.

To build intuition, we first consider the case of uniform sampling, corresponding to Theorem \ref{propofsplitting1}, proved in \cite{AHP-p1}. This setting captures the key ideas while remaining simpler to describe. Under $\Q^{(k),\bm{\theta}}_{T,r}$, the Ulam--Harris-labelled population $\mathcal N$ is augmented with $k$ distinguished spines. The underlying branching dynamics remain unchanged: individuals evolve as under $\mathbb P_r$, except that some may carry one or more spines, while particles without spines behave exactly as in the original MBGW process. The spines are constrained to be distinct at time~$T$ and, crucially, are distributed at that time as a uniform sample without replacement from the population alive at~$T$. Moreover, conditional on the genealogical structure of the tree up to time~$T$, the spines evolve independently. This conditional independence is the key structural feature that enables explicit computations under $\Q^{(k),\bm{\theta}}_{T,r}$. The final step consists in relating this measure to $\mathbb{P}^{(k)}_{\mathrm{unif},T,r}$ and then invoking Yaglom's limit (see Proposition \ref{propSophie}). This step is technically involved and requires several delicate convergence arguments to identify the limiting genealogy in the critical regime.

The remainder of the paper is devoted to the proofs of the main results. In Section \ref{Spines}, we introduce the multiple-spine framework and the associated changes of measure corresponding to each sampling scheme. Within this framework, Section \ref{subsecuniformsampling} recalls the key structural properties of the spines under $\Q^{(k),\bm{\theta}}_{T,r}$, including a forward construction of the multitype branching tree, as developed in \cite{AHP-p1}.
In Sections \ref{subsecunsamconf} and \ref{subsecunsamweights}, we establish the corresponding structural properties under the measures $\Q^{(k),\bm{\theta}}_{\bo c, T,r}$ and $\Q^{(k),\bm{\theta}}_{\bo w, T,r}$, describe their relationships with $\Q^{(k),\bm{\theta}}_{T,r}$, and rigorously define the associated sampling laws $\mathbb{P}^{(k)}_{\bo c,T,r}$ and $\mathbb{P}^{(k)}_{\bo w,T,r}$.
Finally, building on the main identities from Theorem \ref{propofsplitting1} in the uniform sampling setting and the results developed in Section \ref{Spines}, we prove Theorems \ref{propofsplitting2Intro} and \ref{propofsplitting3Intro} in Section \ref{proofmain}. The subsequent sections are devoted to analysing the large-$T$ asymptotic behaviour of the spines under $\Q^{(k),\bm{\theta}}_{T,r}$ (and the analogous tilted measures for the other sampling schemes), as well as under the corresponding sampling laws induced by the original process. Understanding the limiting joint distributions of the spines constitutes the key step in establishing Theorem \ref{propoColorOfTheSpineIsMCPUnifLimit}.

\subsection{Related literature}

Recent years have seen significant progress in the study of ancestral processes in the single type setting under uniform sampling; see \cite{MR4003147,MR4133376,MR4674065,MR4718398}. Related results for discrete-time BGW trees in varying environments were obtained by Boenkost et al.~\cite{boenkost2024genealogynearlycriticalbranching} and Harris et al.~\cite{MR4840489}. In particular, \cite{MR4003147} analyses the limiting genealogy of subcritical and supercritical continuous-time BGW trees, while \cite{MR4718398} treats the heavy-tailed case, where multiple mergers arise in the limit.

In contrast, genealogical questions for the multitype setting remain less developed. Most existing works focus on large-time behaviour, typically under conditioning on non-extinction.
A systematic treatment is provided by J.-Y. Hong and coauthors \cite{MR2942128,MR3306444,MR3317481,MR3570095,MR3749363}. Hong's thesis \cite{MR2942128} introduces a general framework based on uniform sampling without replacement, tracing ancestral lineages backward until coalescence. Within this setting, the distribution of the most recent common ancestor (MRCA), including its generation, type, and death time, is characterised for both discrete and continuous-time models across all regimes. Further refinements are obtained in subsequent works: the critical case (under finite variance) is studied via an associated point process \cite{MR3570095}, while the subcritical and supercritical cases (under an $X\log X$ condition) yield explicit expressions for the MRCA and related type distributions \cite{MR3317481}. Together, these results provide a detailed description of the MRCA for finite samples.

The work of Foutel-Rodier and Schertzer \cite{MR4664584} is particularly close to ours. They study large-time genealogies of finite samples using a many-to-few formula and moment methods within a general critical branching Markov framework (including MBGW processes), proving convergence to coalescent-type limits described by the marked Brownian coalescing point process. A key contribution is the introduction of spinal probability measures for uniform sampling, combined with a size-biased change of measure preserving the branching structure.
Their approach relies on the existence of moments of order $k$ and a notion of criticality via a harmonic Doob $h$-transform. In contrast, we focus on MBGW processes at fixed times $T$, without requiring criticality, and use different type of sampling procedures. We also obtain explicit descriptions of the full genealogical structure, including splitting times, offspring configurations, and types. In the asymptotic regime, we assume only the existence of a Yaglom limit (equivalently, finite variance in our setting). Extensions beyond finite-moment assumptions and criticality are left for future work , see \cite{AHP-p3, AHP-p4}.
\section{Spines and changes of measures}\label{Spines}

 Throughout, we use $\mathbb{R}_+:= [0,\infty)$ and adopt the standard Ulam-Harris labelling system to encode the genealogical structure of particles.  We  recall that $\mathbf{Z}=(\mathbf{Z}_t)_{t\ge 0}$ is a continuous time $\mathbb{Z}^d_+$-valued BGW branching process with  probabilities $(\mathbb{P}_{\mathbf{z}})_{\mathbf{z}\in \mathbb{Z}_+^d}$ on  the filtered probability space $(\Omega, \mathcal{F}, (\mathcal{F}_t)_{t\ge 0})$.

\subsection{Spines \&  measure $\mathbb{P}^{(k)}_r$}\label{subsectionChangeOfMEasure}

Let  $k\in \mathbb{N}$ and $r\in[d]$ be fixed. We briefly describe the reference measure $\mathbb{P}^{(k)}_r$, introduced in \cite{AHP-p1},  which plays a central role in our analysis.  

Under $\mathbb{P}^{(k)}_r$,  the population process $\mathcal{N}=(\mathcal{N}_t)_{t\ge 0}$, started from a single ancestor of type $r$, is equipped with  $k$ distinguished lines of descent, called {\it spines}. 
 We denote the spines  by  $\bm{\varsigma}=(\varsigma^{(1)},\ldots, \varsigma^{(k)})$, where $\varsigma^{(i)}$ represents the $i$-th \emph{spine}. 
 
Each spine is a path in the genealogical tree, encoded by a sequence of Ulam-Harris labels $v_0 v_1 v_2 \dots$ with $v_0=\emptyset$ and $v_{i+1}=v_i\ell$ for some $\ell\in\{1,\dots,\bo 1\cdot \bo L_{v_i}\}, $ for $m\in [d]$,  where $\bo{L}_{u}=(L^{(1)}_u, \ldots, L^{(d)}_u)$ denotes the offspring of particle $u$ which has  distribution $\bo{p}$.  A spine may be an infinite line of descent, or a finite path which terminates at a leaf in the underlying genealogical tree of the population. If a particle $u$  has $j$ distinct spines passing though it, we say that $u$ carries $j$ spines.

The couple $(\mathcal{N},\bm{\varsigma})$ under $\mathbb{P}^{(k)}_r$ is constructed as an extension of the original MBGW process: particles evolve as under $\mathbb{P}_r$, while marks (spines) are assigned and propagated as follows:
\begin{enumerate}
	\item the initial particle carries $k$ marks,
	\item each mark follows a \emph{spine},  
	\item particles branch according to the original offspring law,
	\item 
	given that $\bm{\ell}=(\ell_1,\ldots,\ell_d)$ particles, say $w_1,\ldots, w_{\ell_m}$, type $m$ are born at a branching event as above, the $q$ marks each choose independently to follow type $m$ with probability $\ell_m\xi_m /\bm{\ell}\cdot \bm{\xi}$, and then they follow particle $w_{h}$ with probability $1/\ell_m$ for $h\in [\ell_m]$. 
\end{enumerate}
In Figure \ref{figTreeUnderP8}, we show an example of a MBGW tree under $\p^{(8)}_2$. 
\begin{center}
		\begin{figure}
			\includegraphics[width=.6\textwidth]{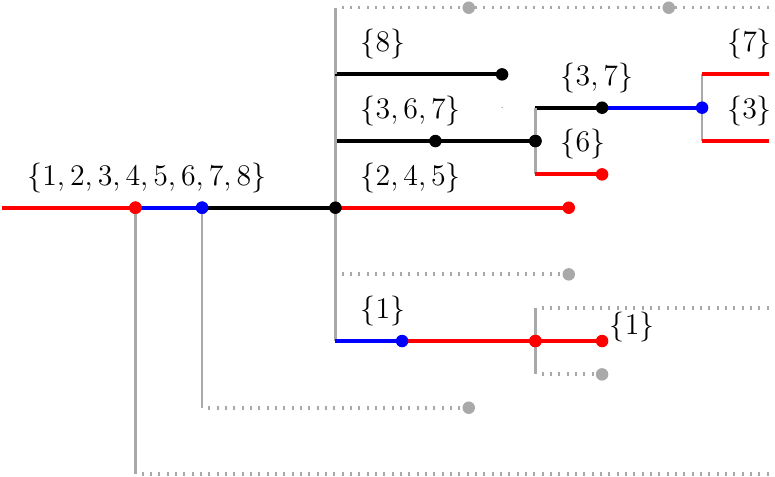}\caption{3-type MBGW tree under $\p^{(8)}_2$, where individuals type 1 are depicted with color Black, type 2 with color Red, and type 3 with color Blue. The time of death of a particle is represented by a dot of its color. The dotted lines represent those particles that carry no marks. }\label{figTreeUnderP8}	
		\end{figure}
\end{center}
Let $\F^{(k)}_t$ denote the filtration containing both the tree and spine information up to time $t$, and write $\bm{\varsigma}_t=(\varsigma_t^{(1)},\ldots,\varsigma_t^{(k)})$ for the particles carrying the spines at time $t$. For $h\in [k]$, we define 
\[
\textrm{spine}(\varsigma_t^{(h)}):=\Big((u^{(h)}_1,c^{(h)}_{2}),(u^{(h)}_2,c^{(h)}_{3}),\ldots, (u^{(h)}_{m-1},c^{(h)}_{m}),(u^{(h)}_{m},c^{(h)}_{m+1})\Big)
\]
where $m$ is such that $\varsigma_t^{(h)}=u^{(h)}_m$,   be the spine generated by $\varsigma_t^{(h)}$ together with the color of the vertex that the $h$-th mark decided to follow.

Assume that each vertex $u^{(h)}_{g}$ has children $\bo{L}_{h,g}:=(\ell^{(1)}_{h,g},\ldots, \ell^{(d)}_{h,g})$. 
By definition, for each label $v$ we have
\begin{align*}
	\p^{(k)}_r\paren{\left.\varsigma_t^{(h)}=v\right|\ \F_t}&=\frac{\ell^{(c^{(h)}_{2})}_{h,1}\xi_{c^{(h)}_{2}}}{\bo{L}_{h,1}\cdot \bm{\xi}}\frac{1}{\ell^{(c^{(h)}_{2})}_{h,1}}\frac{\ell^{(c^{(h)}_{3})}_{h,2}\xi_{c^{(h)}_{3}}}{\bo{L}_{h,2}\cdot \bm \xi}\frac{1}{\ell^{(c^{(h)}_{3})}_{h,2}}\cdots \frac{\ell^{(c^{(h)}_{m})}_{h,m}\xi_{c^{(h)}_{m+1}}}{\bo{L}_{h,m}\cdot \bm \xi }\frac{1}{\ell^{(c^{(h)}_{m+1})}_{h,m}} = \prod_{g=1}^{m}\frac{\xi_{c^{(h)}_{g+1}}}{ \bo{L}_{h,g}\cdot\bm \xi},
	\end{align*}
or equivalently 
\[
\p^{(k)}_r\paren{\left.\varsigma_t^{(h)}=v\right|\ \F_t}=\prod_{(w,c_w)\in \textrm{spine}(v)}\frac{\xi_{c_{w}}}{\bo L_{w}\cdot \bm \xi },
\]
where $c_w$ represents the color of the offspring of $w$ that  is followed by the mark, and $\bo L_{w}$ denotes the offspring of $w$. 

 Recall that  $\mathcal{N}_{t}^{(k)}$ denote the set of all possible $k$-tuples of particles which are alive at time $t$ and  observe that $N_{t}^{\floor{k}}$ is precisely the cardinality of $\mathcal{N}_{t}^{(k)}$.
  
 \subsection{Uniform sampling}\label{subsecuniformsampling} Following \cite{AHP-p1}, we  introduce the probability  measure $\Q^{(k),\bm{\theta}}_{t,r}$,  under which the $k$-spines constitute a uniform sample without replacement from the particles alive at time $t$. This measure is defined via a change of measure from $\mathbb{P}_r^{(k)}$ and yields a more transparent description of the sampling process and renders many functionals of interest more tractable.
 
 Define
\begin{equation}\label{eqnDefinitionOfG_ktAndZeta_kt}
\begin{split}
g_{k,t}:%=\mathbf{1}_{\{\varsigma^{(i)}_t\neq \varsigma^{(j)}_t,\mbox{ for }\ i\neq j\}}\prod_{h\in [k]}\prod_{g\in [t-1]}\frac{ \bo{L}_{h,g}\cdot \bm{\xi}}{\xi_{c_{g+1}^{(h)}}}\\
=\sum_{\bo{v}\in \mathcal{N}_{t}^{(k)}}\mathbf{1}_{\{\bo{\varsigma}_t=\bo{v}\}}\prod_{h\in [k]}\prod_{(w,c_w)\in \mathrm{spine}(v^{(h)})}\frac{ \bo{L}_{w}\cdot \bm{\xi}}{\xi_{c_{w}}},
\end{split}
\end{equation}
with the convention that $g_{0,t}=1$. For $\bm{\theta}\in \mathbb{R}_+^{d}$, define the change of measure 
\begin{equation}\label{defMeasureQkTrGivenAllInfo}
	\left.\frac{{\rm d}\Q^{(k),\bm{\theta}}_{t,r}}{{\rm d}\p^{(k)}_r}\right|_{\F^{(k)}_t}:=\frac{g_{k,t}e^{-\bm \theta\cdot \bo Z_{t}}}{\E^{(k)}_r\left[N^{\floor{k}}_{t}e^{-\bm \theta\cdot \bo Z_{t}}\right]}, \qquad t\ge 0.
\end{equation}
Since $
 \E^{(k)}_r\left[\left. g_{k,t} \right|\ \F_t\right]=N^{\floor{k}}_{t}$, the measure  $\Q^{(k),\bm{\theta}}_{t,r}$ is  a probability  and satisfies 
 \begin{equation}\label{defMeasureQkTrGivenTopologicalInfo}
\left.\frac{{\rm d}\Q^{(k),\bm{\theta}}_{t,r}}{{\rm d}\p^{(k)}_r}\right|_{\F_t}=\frac{N^{\floor{k}}_{t}e^{-\bm \theta\cdot \bo Z_{t}}}{\E^{(k)}_r\left[N^{\floor{k}}_{t}e^{-\bm \theta\cdot \bo Z_{t}}\right]}.
\end{equation}

By \cite{AHP-p1},  under the measure $\Q^{(k),\bm \theta}_{t,r}$, the $k$-spines are a uniform choice without replacement from all  particles alive at time $t$, that is
\begin{equation}\label{eqnQIsUniformGivenF0}
	\Q^{(k),\bm \theta}_{t,r}\paren{\left.\bm \varsigma_t=\bo v \right|\ \F_t}= \frac{1}{N^{\floor{k}}_{t}}.
\end{equation}
Consequently, 
 \begin{equation}\label{defMeasureQkTrGivingTwoSteps}
	\left.\frac{{\rm d}\Q^{(k),\bm \theta}_{t,r}}{{\rm d}\p^{(k)}_r}\right|_{\F^{(k)}_t}=\frac{N^{\floor{k}}_{t}e^{-\bm \theta\cdot \bo Z_{t}}}{\E^{(k)}_r\left[N^{\floor{k}}_{t}e^{-\bm \theta\cdot \bo Z_{t}}\right]}\frac{1}{N^{\floor{k}}_{t}}g_{k,t},
\end{equation}
which states that we first apply $k$-size biasing and $\bm \theta$-discounting to the process given $\F_t$, and then, conditional on $\F_t$ under $\Q^{(k),\bm \theta}_{t,r}$, we select $k$ spines uniformly without replacement.

It is important to note that under $\mathbb{Q}^{(k),\bm{\theta}}_{t,r}$, the population process $\mathcal{N}$ admits a forward-in-time description as a size-biased and discounted multitype Galton-Watson process with $k$ spines see \cite{AHP-p1}, Proposition 1.  Moreover, under $\mathbb{Q}^{(k),\bm{\theta}}_{t,r}$, the  process $\mathcal{N}$ satisfies the branching Markov property. Specifically,  a particle carrying $h\in [k]$ spines with type $i$ and alive at time $s$ generates a subtree that evolves independently of the rest of the population,  with law $\mathbb{Q}^{(h),\bm{\theta}}_{t-s,i}$. 

In particular, particles carrying spines evolve independently and reproduce according to a tilted offspring distribution. More precisely,  consider a particle type $i\in [d]$ carrying $h\in [k]$ spines at time $s\in (0,t)$.  This particle branches into $\bm \ell\in \z^d_+\setminus\{\bo 0\}$ offspring, and the $h$ spines are distributed among the offspring as follows: for each $m\in [d]$, the spines are split into $g_m\in \{0,1,\ldots,\ell_m\}$ groups, where the $q$-th group carries $a_{m,q} \in [h]$ marks, with $\sum_{m\in[d]} \overline{a}_{m}=h$ and $q\in [g_m]$. This branching event occurs  at rate
	\begin{equation}
		\begin{split}
			& \alpha_i\E_i\left[\bo L^{\floor{\bo g}}\prod_{m=1}^d\E_m\left[e^{-\bm \theta \cdot \bo Z_{t-s}}\right]^{L^{(m)}-g_m}\right]
			\frac{\bm \ell^{\floor{\bo g}}			\prod_{m=1}^d\E_m\left[e^{-\bm \theta \cdot \bo Z_{t-s}}\right]^{\ell_m-g_m}p_i(\bm{\ell})}{\E_i\left[\bo L^{\floor{\bo g}}\prod_{m=1}^d\E_m\left[e^{-\bm \theta \cdot \bo Z_{t-s}}\right]^{L^{(m)}-g_m}\right]}\\
			& \hspace{7cm} \times
			\frac{\prod_{\substack{m\in [d]\\ g_m\neq 0}}\prod_{q=1}^{g_m}\E_m\left[N^{\floor{a_{m,q}}}_{t-s}e^{-\bm \theta \cdot \bo Z_{t-s}}\right]}{\E_i\left[N^{\floor{h}}_{t-s}e^{-\bm \theta\cdot \bo Z_{t-s}}\right]}.
		\end{split}
\end{equation}
Furthermore, the allocation of spines among the offspring is exchangeable and depends only on the group sizes; see \cite{AHP-p1}, Proposition 1 part (4).

 Particles that do not carry spines evolve independently according to a corresponding tilted dynamics. More precisely, 
a particle $v$ of type $i$, alive at time $s$ and carrying no spines,  undergoes branching  and produces $\bm \ell$  offspring at rate
\[
\alpha_i\frac{\E_i\left[\prod_{m=1}^d\E_m\left[e^{-\bm \theta\cdot \bo Z_{t-s} }\right]^{ L^{(m)}} \right]}{\E_i\left[e^{-\bm \theta\cdot \bo Z_{t-s} }\right]}p_i(\bm \ell)\frac{\prod_{m=1}^d\E_m\left[e^{-\bm \theta\cdot \bo Z_{t-s} }\right]^{\ell_m}}{\E_i\left[\prod_{m=1}^d\E_m\left[e^{-\bm \theta\cdot \bo Z_{t-s} }\right]^{ L^{(m)}} \right]}.
\]

Next, we  introduce the probability measure $\p^{(k)}_{unif,t,r}$, which corresponds to uniform sampling without replacement from the population at time $t$.  For any $A\in \F^{(k)}_t$, define
\begin{equation}\label{probuni1}
\left.\frac{{\rm d}\p^{(k)}_{unif,t,r}}{{\rm d}\p^{(k)}_r}\right|_{\F^{(k)}_t}:=\frac{1}{\p_r\paren{N_t\geq k}}\frac{g_{k,t}}{N^{\floor{k}}_{t}},
\end{equation}
which indeed defines  a  probability measure (see  \cite{AHP-p1}).

According to \cite{AHP-p1}, if $\bm \varsigma_t$ is a uniform sample without replacement at time $t$ and $\bo v\in \mathcal{N}^{(k)}_t$, then 
\[
\mathbf{1}_{\{N_t\geq k\}}\p^{(k)}_{unif,t,r}\paren{\left.\bm \varsigma_t=\bo v \right|\ \F_t}
	= \mathbf{1}_{\{N_t\geq k\}}\frac{1}{N^{\floor{k}}_{t}}.
\]
Moreover, for  a function of $k$ distinct vertices of the tree at time $t$, say $f$,  and $A\in \F_t$; we have
\begin{equation}\label{probuni}
\p^{(k)}_{unif,t,r}(A)=\p_r\paren{A\big| N_t\geq k}\quad \mbox{and} \quad \mathbb{E}^{(k)}_{unif,t,r}\left[f( \bm{\varsigma})\right]=\mathbb{E}\left[\frac{1}{N^{\floor{k}}_t}\sum_{\bo{v}\in \mathcal{N}^{(k)}_t} f(\bo{v})\Bigg| N_t\ge k\right].
\end{equation}
On the event $\{N_t\geq k\}$, the Radon-Nikodym derivative of $\Q^{(k),\bm \theta}_{t,r}$ with respect to $\p^{(k)}_{unif,t,r}$ is given by
\begin{equation}\label{defMeasureQkTrandPUnif}
\begin{split}
	\left.\frac{{\rm d}\Q^{(k),\bm \theta}_{t,r}}{{\rm d}\p^{(k)}_{unif,t,r}}\right|_{\F^{(k)}_t}& =\frac{N^{\floor{k}}_{t}e^{-\bm \theta\cdot \bo Z_{t}}}{\E^{(k)}_{unif,t,r}\left[N^{\floor{k}}_{t}e^{-\bm \theta\cdot \bo Z_{t}}\right]}.
\end{split}
\end{equation}
In other words,  when passing from $\p^{(k)}_{unif,t,r}$ to $\Q^{(k),\bm \theta}_{t,r}$, events in $\F^{(k)}_t$ are modified  through the tree topology (size-biasing and discounting), while the marks remain  unaffected. Furthermore, for $A\in \F^{(k)}_t$, 
	\[
\p^{(k)}_{unif,t,r}\big( A\big)=\E^{(k)}_r\Big[\left.N^{\floor{k}}_{t}e^{-\bm \theta\cdot \bo Z_{t}}\right|\ N_t\geq k\Big]\Q^{(k),\bm \theta}_{t,r}\left[\frac{\mathbf{1}_{A}}{N^{\floor{k}}_{t}e^{-\bm \theta\cdot \bo Z_{t}}}\right].
		\]
		This property relationship plays a key role in the proof of Theorem \ref{propofsplitting1}. 
		
Finally, we emphasise that the law of the underlying MBGW tree is the same under $\mathbb{P}^{(k)}_{\mathbf{z}}$ as under $\mathbb{P}_{\mathbf{z}}$, for $\mathbf{z}\in \mathbb{Z}_+^d$, that is, for all $t\ge 0$,
\[
\mathbb{P}^{(k)}_{\mathbf{z}}=\mathbb{P}_{\mathbf{z}}\qquad \textrm{on } \mathcal{F}_t.
\] 
Further details can be found in Section 2.1.1 of \cite{AHP-p1}.

\subsection{Uniform sampling  given a fixed type configuration}\label{subsecunsamconf}
We now introduce the auxiliary measure, analogous to the uniform sampling framework, under which the genealogical tree evolves as a branching process, thus enabling a more tractable analysis of the sampling procedure. To this end, let $\bo c:=(c_1,\ldots, c_k)$ be a vector with entries $c_h\in [d]$ for all $h\in [k]$,  representing a prescribed configuration. 
Recall that the sample degree vector is denoted by $\bo D_{\bo c}=(D_{\bo c,1},\ldots, D_{\bo c,d})$.

Let $\mathcal{N}^{(\bo c)}_{t}$ denote the set of distinct $k$-tuples $(v^{(1)},\ldots,v^{(k)})$ of individuals in $\mathcal{N}_{t}$, 
such that the  $h$-th individual has type  $c_h$, for each $h\in [d]$.  In other words, we select $k$ individuals  alive at time $t$ matching the  type configuration $\mathbf{c}$.  The number of such $k$-tuples is 
\[
\textrm{card}\{\mathcal{N}^{(\bo c)}_{t}\}=\bo Z^{\floor{\bo D_{\bo c}}}_t=\prod_{m=1}^d\big(Z^{(m)}_t\big)^{\floor{D_{\bo c,m}}}=\prod_{m=1}^d Z^{(m)}_{t}(Z^{(m)}_{t}-1)\cdots (Z^{(m)}_{t}-D_{\bo c,m}+1),
\]
 provided that  $Z^{(m)}_{t}\geq D_{\bo c,m}$, for all $m\in [d]$;  and is zero otherwise.
 
  This  follows since, for each type 
 $m$, there are $Z_t^{(m)}$ individuals of that type, and we  choose $D_{\mathbf{c},m}$ of them  without replacement. The number of  ordered selections is  the falling factorial $(Z^{(m)}_{t})^{\floor{D_{\bo c,m}}}$, and the  total count is obtained by multiplying  over all types.

Similarly as in \eqref{eqnDefinitionOfG_ktAndZeta_kt}, we observe that our proposed Radon-Nikodym derivative might be of the form 
\begin{equation}\label{eqnDefinitionOfG_ktAndZeta_ktS}
\begin{split}
g_{\bo c,t}&:=g_{k,t}\mathbf{1}_{\{c( \varsigma^{(h)}_t)=c_h,\forall\ h\in[k]\}}\\
&=\mathbf{1}_{\{\varsigma^{(i)}_t\neq \varsigma^{(j)}_t,\mbox{ for }\ i\neq j,c( \varsigma^{(h)}_t)=c_h,\forall\ h\in[k]\}}\prod_{h\in [k]}\prod_{g\in [t-1]}\frac{ \bo{L}_{h,g}\cdot \bm{\xi}}{\xi_{c_{g+1}^{(h)}}}\\
&=\sum_{\bo{v}\in \mathcal{N}_{t}^{(\bo c)}}\mathbf{1}_{\{\bm \varsigma_t=\bo{v}\}}\prod_{h\in [k]}\prod_{(w,c_w)\in \mathrm{spine}(v^{(h)})}\frac{ \bo{L}_{w}\cdot \bm{\xi}}{\xi_{c_{w}}},
\end{split}
\end{equation}
with the convention that $g_{\bo c,t}=1$ in the case $k=0$ with $\bo c=\emptyset$. For $\bm{\theta}\in \mathbb{R}_+^{d}$, we also define
\begin{equation}\label{cebolla}
 \zeta^{\bm{\theta}}_{\bo c,t}:=\frac{g_{\bo c,t}e^{-\bm \theta\cdot \bo Z_{t}}}{\E_r\left[\bo Z^{\floor{\bo D_{\bo c}}}_te^{-\bm \theta\cdot \bo Z_{t}}\right]}.
\end{equation}Then, for $t>0$ and $r\in [d]$,  it is clear that,
\[
\begin{split}
 \E^{(k)}_r\left[\left. g_{\bo c,t} \right|\ \F_t\right]&= \sum_{\bo{v}\in \mathcal{N}_{t}^{(\bo c)}}\prod_{h\in [k]}\prod_{(w,c_w)\in \mathrm{spine}(v^{(h)})}\frac{ \bo{L}_{w}\cdot \bm{\xi}}{\xi_{c_{w}}}\p^{(k)}_r\paren{\left.\bo{\varsigma}_t=\bo{v} \right|\ \F_t}={\rm card}\{\mathcal{N}_{t}^{(\bo c)}\}=\bo Z^{\floor{\bo D_{\bo c}}}_t,
\end{split}
\] 
and thus
\begin{equation}\label{tomate1}
 \E^{(k)}_r\left[\left.  g_{\bo c,t}e^{-\bm \theta\cdot \bo Z_{t}} \right|\ \F_t\right]
=\bo Z^{\floor{\bo D_{\bo c}}}_te^{-\bm \theta\cdot \bo Z_{t}}.
\end{equation}
The latter suggest that we may construct, for $t>0$,  the  following probability measure 
\begin{equation}\label{defMeasureQkTrGivenAllInfoS}
	\left.\frac{{\rm d}\Q^{(k),\bm{\theta}}_{\bo c,t,r}}{{\rm d}\p^{(k)}_r}\right|_{\F^{(k)}_t}:=\zeta^{\bm{\theta}}_{\bo c,t}.
\end{equation}
Since $\E^{(k)}_r\left[\left.g_{\bo c,t} \right|\ \F_t\right]= \bo Z^{\floor{\bo D_{\bo c}}}_t$, it follows 
\begin{equation}\label{defMeasureQkTrGivenTopologicalInfoS}
\left.\frac{{\rm d}\Q^{(k),\bm{\theta}}_{\bo c,t,r}}{{\rm d}\p^{(k)}_r}\right|_{\F_t}=\frac{\bo Z^{\floor{\bo D_{\bo c}}}_te^{-\bm \theta\cdot \bo Z_{t}}}{\E_r\left[\bo Z^{\floor{\bo D_{\bo c}}}_te^{-\bm \theta\cdot \bo Z_{t}}\right]}.
\end{equation}We  say that this change of measure, applies a $\bo D_{\bo c}$-size biasing and $\bm \theta$-discounting to the process.

From \eqref{defMeasureQkTrGivenAllInfo} and \eqref{defMeasureQkTrGivenAllInfoS}, we observe that  the probability measures $\Q^{(k),\bm{\theta}}_{\bo c,t,r}$ and $\Q^{(k),\bm{\theta}}_{t,r}$, satisfy the following relationship
\begin{equation}\label{eqnComparingQSAndQk}
\left.\frac{{\rm d}\Q^{(k),\bm{\theta}}_{\bo c,t,r}}{{\rm d}\Q^{(k),\bm{\theta}}_{t,r}}\right|_{\F^{(k)}_t}=\frac{\mathbf{1}_{\{c( \varsigma^{(h)}_t)=c_h,\forall\ h\in[k]\}}}{\Q^{(k),\bm \theta }_{t,r}(c( \varsigma^{(h)}_t)=c_h,\forall\ h\in[k])}.
\end{equation}
The denominator can be computed, using \eqref{cebolla}, \eqref{defMeasureQkTrGivenAllInfo} and \eqref{eqnDefinitionOfG_ktAndZeta_ktS}. Indeed, we have 
\begin{align}\label{eqnonumber}
\Q^{(k),\bm \theta }_{t,r}(c( \varsigma^{(h)}_t)=c_h,\forall\ h\in[k])& =\frac{\E^{(k)}_r\left[g_{\bo c,t}e^{-\bm \theta\cdot \bo Z_{t}}\right]}{\E_r\left[N^{\floor{k}}_te^{-\bm \theta\cdot \bo Z_{t}}\right]}=\frac{\E_r\left[\bo Z^{\floor{\bo D_{\bo c}}}_te^{-\bm \theta\cdot \bo Z_{t}}\right]}{\E_r\left[N^{\floor{k}}_te^{-\bm \theta\cdot \bo Z_{t}}\right]}.
\end{align}

We now introduce the uniform measure $\p^{(k)}_{\bo c,t,r}$, which samples $k$ individuals at time $t$  uniformly among those whose   types match the configuration $\bo c$.  We firs show that $\mathbb{Q}^{(k),\bm \theta}_{\bo c, t,r}$ corresponds to uniform sampling  without replacement from the  individuals alive at time $t$, restricted to those with type configuration   $\bo c$.  To this end,  we apply the following useful result which  is in \cite{MR4133376}, Lemma 13. We recall it for simplicity on exposition.

\begin{lemma}\label{lemHJR} Let $\mu$ and $\nu$ be two probability measures on a $\sigma$-algebra $\F$ and $\G\subseteq \F$ is also a $\sigma$-algebra with Radon-Nikodym derivatives
\[
\left.\frac{{\rm d}\mu}{{\rm d}\nu}\right|_{\F}=:Y\qquad \mbox{ and }\qquad \left.\frac{{\rm d}\mu}{{\rm d}\nu}\right|_{\G}=:Z.
\]
Then for any non-negative random variable $X$, $\F$-measurable, we have
\[
Z\mu[X |\, \G]=\nu[XY|\, \G]\qquad \nu-a.s.
\]
\end{lemma}
Thus taking $\F=\F^{(k)}_t$, $\G=\F_t$, $\mu=\Q^{(k),\bm \theta}_{\bo c,t,r}$, $\nu=\p^{(k)}_r$ and $X$ a non-negative $\F^{(k)}_t$-measurable r.v., we get
\[
Y:=\frac{g_{\bo c,t}e^{-\bm \theta\cdot \bo Z_{t}}}{\E^{(k)}_r\left[\bo Z^{\floor{\bo D_{\bo c}}}_te^{-\bm \theta\cdot \bo Z_{t}}\right]}\qquad \mbox{and}\qquad Z:=\frac{\bo Z^{\floor{\bo D_{\bo c}}}_te^{-\bm \theta\cdot \bo Z_{t}}}{\E^{(k)}_r\left[\bo Z^{\floor{\bo D_{\bo c}}}_te^{-\bm \theta\cdot \bo Z_{t}}\right]},
\]implying that
\[
	\frac{\bo Z^{\floor{\bo D_{\bo c}}}_te^{-\bm \theta\cdot \bo Z_{t}}}{\E^{(k)}_r\left[\bo Z^{\floor{\bo D_{\bo c}}}_te^{-\bm \theta\cdot \bo Z_{t}}\right]}\Q^{(k),\bm{\theta}}_{\bo c,t,r}\Big[ X \Big|\ \F_t\Big]=\E^{(k)}_r\left[\left. X \frac{g_{\bo c,t}e^{-\bm \theta\cdot \bo Z_{t}}}{\E^{(k)}_r\left[\bo Z^{\floor{\bo D_{\bo c}}}_te^{-\bm \theta\cdot \bo Z_{t}}\right]}\right|\ \F_t\right].
\]
Considering $X=\mathbf{1}_{\{{\bm \varsigma_t}=\bo v\}}$, for $\bo v \in \mathcal{N}^{(\bo c)}_t$, we get
\[
\begin{split}
	\frac{\bo Z^{\floor{\bo D_{\bo c}}}_te^{-\bm \theta\cdot \bo Z_{t}}}{\E^{(k)}_r\left[\bo Z^{\floor{\bo D_{\bo c}}}_te^{-\bm \theta\cdot \bo Z_{t}}\right]}\Q^{(k),\bm \theta}_{\bo c, t,r}\paren{\left.\bm \varsigma_t=\bo v \right|\ \F_t}&=\E^{(k)}_r\left[\left. \mathbf{1}_{\{\bm \varsigma_t=\bo v\}}\frac{g_{\bo c,t}e^{-\bm \theta\cdot \bo Z_{t}}}{\E^{(k)}_r\left[\bo Z^{\floor{\bo D_{\bo c}}}_te^{-\bm \theta\cdot \bo Z_{t}}\right]}\right|\ \F_t\right]\\
	&=\frac{e^{-\bm \theta\cdot \bo Z_{t}}}{\E^{(k)}_r\left[\bo Z^{\floor{\bo D_{\bo c}}}_te^{-\bm \theta\cdot \bo Z_{t}}\right]}\E^{(k)}_r\left[\left. \mathbf{1}_{\{\bm \varsigma_t=\bo v \}} g_{\bo c,t}\right|\ \F_t\right].
\end{split}
\]
Recalling the definition of $g_{\bo c,t}$, we deduce
\begin{equation}\label{eqnQIsUniformGivenF0S}
\begin{split}
	\Q^{(k),\bm \theta}_{\bo c, t,r}\paren{\left.\bm \varsigma_t=\bo v \right|\ \F_t}
	&=\frac{1}{\bo Z^{\floor{\bo D_{\bo c}}}_t}\E^{(k)}_r\left[\left. \mathbf{1}_{\{\bm \varsigma_t=\bo v \}} g_{\bo c,t}\right|\ \F_t\right]\\
	&= \frac{1}{\bo Z^{\floor{\bo D_{\bo c}}}_t}\prod_{h\in [k]}\prod_{(w,c_w)\in {\rm spine}(v^{(h)})}\frac{\bo{L}_w\cdot \bm \xi}{\xi_{c_w}}\p^{(k)}_r\paren{\left.   \varsigma_t^{(h)}= v^{(h)}  \right|\ \F_t}\\
	&= \frac{1}{\bo Z^{\floor{\bo D_{\bo c}}}_t}.
\end{split}
\end{equation}
Thus, conditionally on $\mathcal{F}_t$, the $k$ spines  under  $\Q^{(k),\bm \theta}_{\bo c, t,r}$, are a uniform choice without replacement from all the alive particles at time $t$, having degree type $\bo D_{\bo c}$. 
This also provides a more explicit description of $\Q^{(k),\bm \theta}_{\bo c, t,r}$:
\begin{equation}\label{defMeasureQkTrGivingTwoStepsS1}
\begin{split}
	\left.\frac{{\rm d}\Q^{(k),\bm \theta}_{\bo c, t,r}}{{\rm d}\p^{(k)}_r}\right|_{\F^{(k)}_t}&=\frac{\bo Z^{\floor{\bo D_{\bo c}}}_te^{-\bm \theta\cdot \bo Z_{t}}}{\E^{(k)}_r\left[\bo Z^{\floor{\bo D_{\bo c}}}_te^{-\bm \theta\cdot \bo Z_{t}}\right]}\frac{1}{\bo Z^{\floor{\bo D_{\bo c}}}_t}g_{\bo c,t}.\\
%&\hspace{2cm}\mathbf{1}_{\{\varsigma^{(i)}_t\neq \varsigma^{(j)}_t,\mbox{ for }\ i\neq j,c( \varsigma^{(h)}_t)=c_h,\forall\ h\in[k]\}}\prod_{h\in [k]}\prod_{(w,c_w)\in {\rm spine}(\varsigma^{(h)}_t)}\frac{\bo{L}_w\cdot \bm \xi}{\xi_{c_w}},
\end{split}\end{equation}
In other words, given $\mathcal{F}_t$, the measure first applies a $\mathbf{D}_{\mathbf{c}}$-size bias and a $\bm \theta$ exponential tilt, and then selects $k$ individuals uniformly without replacement with types $(c_1,\ldots,c_k)$.

We now define the uniforme measure  $\p^{(k)}_{\bo c,t,r}$ as follows,
\begin{equation}\label{probuni2S}
\left.\frac{{\rm d}\p^{(k)}_{\bo c,t,r}}{{\rm d}\p^{(k)}_r}\right|_{\F^{(k)}_t}:=
\frac{g_{\bo c,t}}{\p_r\paren{\bo Z_{t}\geq \bo D_{\bo c}}\bo Z^{\floor{\bo D_{\bo c}}}_{t}},
%\frac{\mathbf{1}_{\{\varsigma^{(i)}_t\neq \varsigma^{(j)}_t,\mbox{ for }\ i\neq j,c( \varsigma^{(h)}_t)=c_h,\forall\ h\in[k]\}}}{\p_r\paren{\bo Z_{t}\geq \bo D_{\bo c}}\bo Z^{\floor{\bo D_{\bo c}}}_{t}}\prod_{h\in [k]}\prod_{(w,c_w)\in {\rm spine}(\varsigma^{(h)}_t)}\frac{\bo{L}_w\cdot \bm \xi}{\xi {c_w}}. 
\end{equation}
where the event $\{\bo Z_{t}\geq \bo D_{\bo c}\}=\{Z^{(m)}_{t}\geq D_{\bo c,m}, \textrm{ for all } m\in [d]\}.$

Given the previous definition, we would like to show that such a measure is indeed uniform on the given set. 
To that end, given $\F_t$, consider $\bo v\in \mathcal{N}^{(\bo c)}_t$. 
Then, just as in \eqref{eqnQIsUniformGivenF0S}, we have
\begin{equation}\label{eqnPlsUniformGivenF0S}
\begin{split}
	\p^{(k)}_{\bo c,t,r}\paren{\left.\bm \varsigma_t=\bo v \right|\ \F_t}
	&=\frac{1}{\bo Z^{\floor{\bo D_{\bo c}}}_{t}}\E^{(k)}_r\left[\left. \mathbf{1}_{\{\bm \varsigma_t=\bo v \}} g_{\bo c,t}\right|\ \F_t\right]\\
	&= \frac{1}{\bo Z^{\floor{\bo D_{\bo c}}}_{t}}\prod_{h\in [k]}\prod_{(w,c_w)\in {\rm spine}(v^{(h)})}\frac{\bo{L}_w\cdot \bm \xi}{\xi_{c_w}}\p^{(k)}_r\paren{\left.   \varsigma_t^{(h)}= v^{(h)}  \right|\ \F_t}= \frac{1}{\bo Z^{\floor{\bo D_{\bo c}}}_{t}}.
\end{split}
\end{equation}That is, the measure $\p^{(k)}_{\bo c,t,r}$ is sampling uniformly from $\mathcal{N}^{(\bo c)}_t$.

Finally, we establish an important relationship between $\mathbb{Q}^{(k),\bm{\theta}}_{\bo c,t,r}$ and  $\p^{(k)}_{\bo c,t,r}$, which will be used latter (see Lemma \ref{lemmaFromTheMeasureQToTheMeasureEunifS}). Specifically,
\begin{equation}\label{eqnRelationBetweenQSAndPSUnif}
\left.\frac{{\rm d}\Q^{(k),\bm \theta}_{\bo c, t,r}}{{\rm d}\p^{(k)}_{\bo c,t,r}}\right|_{\F^{(k)}_t}=\frac{\bo Z^{\floor{\bo D_{\bo c}}}_te^{-\bm \theta\cdot \bo Z_{t}}}{\E_r\left[\bo Z^{\floor{\bo D_{\bo c}}}_te^{-\bm \theta\cdot \bo Z_{t}}\Big| \bo Z_t\geq \bo D_{\bo c}\right]},
\end{equation}which is the analogue of  \eqref{defMeasureQkTrandPUnif}. This identity follows directly from  \eqref{defMeasureQkTrGivingTwoStepsS1} and \eqref{probuni2S}. The following lemma is the analogue of  \cite{AHP-p1}, Lemma 2.

\begin{lemma}\label{lemmaFromTheMeasureQToTheMeasureEunifS}
Suppose $A\in \mathcal{F}^{(k)}_t$. Then
	\[
\p^{(k)}_{\bo c,t,r}(A)=\E_r\Big[\left.\bo Z^{\floor{\bo D_{\bo c}}}_{t}e^{-\bm \theta\cdot \bo Z_{t}}\right|\ \bo Z_t\geq \bo D_{\bo c}\Big]\Q^{(k),\bm \theta}_{\bo c, t,r}\left[\frac{\mathbf{1}_{A}}{\bo Z^{\floor{\bo D_{\bo c}}}_{t}e^{-\bm \theta\cdot \bo Z_{t}}}\right].
		\]
\end{lemma}
\begin{proof}
The proof follows directly from \eqref{eqnRelationBetweenQSAndPSUnif}, by repeating the same steps as in the proof of \cite{AHP-p1}, Lemma 2. 
\end{proof}

\subsection{Sampling according to type dependent weights}\label{subsecunsamweights} Similarly to the previous case, we now introduce the auxiliary measure associated with sampling according to type dependent weights.
Fix a vector of strictly positive weights over types $\bo w:=(w(1),\ldots,w(d))\in \re^d_+$. For $\bo v=(v^{(1)},\ldots, v^{(k)})\in \mathcal{N}_t^{(k)}$, we define its type-degree vector $\bo D_{\bo v}:=(D_{\bo v, 1},\ldots, D_{\bo v, d})$, where 
\[
D_{\bo v, m}={\rm card}\{h\in [k]:c(v^{(h)})=m\}.
\]
The Radon-Nikodym derivative corresponding to this sampling scheme is given by \begin{equation}\label{cebolla3}
 \zeta^{\bm{\theta}}_{\bo w,t}:=\frac{g_{k,t}e^{-\bm \theta\cdot \bo Z_{t}}\prod_{h=1}^kw(c(\varsigma^{(h)}_t))}{\E_r\left[e^{-\bm \theta\cdot \bo Z_{t}}\sum_{\bo{v}\in \mathcal{N}_{t}^{(k)}}\prod_{h=1}^k w(c(v^{(h)}))\right]}.
\end{equation}
%\comentario{Here I erased the extra combinatorial term}
At first glance, the preceding change of measure may appear unnatural due to the combinatorial structure present in both the numerator and the denominator. Nevertheless, as shown in the estimate \eqref{eqnApproximationToMultinomialDenominatorQS}, when the population is sufficiently large, both terms asymptotically resemble those of a multinomial factor.

On the other hand, we observe
\[
\bo w^{\bo D_{\bo v}}=\prod_{m=1}^dw(m)^{D_{\bo v,m}}=\prod_{h=1}^kw(c(v^{(h)})). 
\]
Moreover, for $t>0$ and $r\in [d]$,
\[
\begin{split}
 \E^{(k)}_r\left[\left. g_{k,t}{\bo w}^{{\bo D}_{\bm \varsigma_t}}\right|\ \F_t\right]&= \sum_{\bo{v}\in \mathcal{N}_{t}^{(k)}}\bo w^{\bo D_{\bo v}}\prod_{h\in [k]}\prod_{(w,c_w)\in \mathrm{spine}(v^{(h)})}\frac{ \bo{L}_{w}\cdot \bm{\xi}}{\xi_{c_{w}}}\p^{(k)}_r\paren{\left.\bo{\varsigma}_t=\bo{v} \right|\ \F_t}\\
&=\sum_{\bo{v}\in \mathcal{N}_{t}^{(k)}}\bo w^{\bo D_{\bo v}}.
\end{split}
\] 
Therefore,
\begin{equation}\label{tomate3}
 \E^{(k)}_r\left[\left.  g_{k,t}e^{-\bm \theta\cdot \bo Z_{t}}{\bo w}^{{\bo D}_{\bm \varsigma_t}} \right|\ \F_t\right]
=e^{-\bm \theta\cdot \bo Z_{t}}\sum_{\bo{v}\in \mathcal{N}_{t}^{(k)}}\bo w^{\bo D_{\bo v}}.
\end{equation}

%%%%%%%%%%%%%%%%%%%%%%%%%%%%%%%%%%%%%%%%%%%%

We now provide an exact combinatorial expression for the sum of the weights over all ordered $k$-tuples. 
Let $\binom{k}{\mathbf{D}}=\frac{k!}{D_1!\dots D_d!}$ denote the multinomial coefficient. 
We partition the set of all valid ordered tuples $\mathcal{N}_t^{(k)}$ into equivalence classes based on their type-degree vectors $\bo D = (D_1, \dots, D_d)$. For a fixed degree vector $\bo D$ such that $\sum_{m=1}^d D_m = k$, consider the subset of tuples $\mathcal{S}_{\bo D} = \{ \bo v \in \mathcal{N}_t^{(k)} : \bo D_{\bo v} = \bo D \}$. 
Then one can prove that
\[
\textrm{card}(\mathcal{S}_{\bo D}) = {k \choose \bo D} \bo Z_t^{\floor{\bo D}}.
\]
Indeed, the number of vectors $\bo v\in \mathcal{N}^{(k)}_t$ having exactly $D_{m}$ components of type $m$ is $(Z^{(m)}_t)^{\floor{D_{m}}}$, the number of ways to choose $D_{m}$ distinct individuals from $Z^{(m)}_t$ without replacement. 
Also, the number of ways to partition the $k$ sequence slots into $d$ groups of exact sizes $D_1, \dots, D_d$ is precisely the multinomial coefficient.
Hence
\begin{equation}\label{eqnMultinomialDenominatorQS0}
\sum_{\bo{v}\in \mathcal{N}_{t}^{(k)}} \bo w^{\bo D_{\bo v}} = \sum_{\substack{\bo D\in \{0,1,\ldots,k\}^d:\\ \sum_{m=1}^d D_m=k}}{k \choose \bo D}\bo Z^{\floor{\bo D}}_t\bo w^{\bo D}. 
\end{equation}

Later, we will be  interested in the behaviour of the population conditioned on survival. Under this conditioning, the population ultimately grows without bound,  and for large $n$ with $k$ fixed, we may informally use the approximation $n^{\floor{k}}\sim n^k$. Applying this heuristic in our setting yields
\begin{equation}\label{eqnApproximationToMultinomialDenominatorQS}
\sum_{\substack{\bo D\in \{0,1,\ldots,k\}^d:\\ \sum_{m=1}^dD_{m}=k}}{k \choose \bo D}\bo Z^{\floor{\bo D}}_t\bo w^{\bo D}\sim \sum_{\substack{\bo D\in \{0,1,\ldots,k\}^d:\\ \sum_{m=1}^dD_{m}=k}}{k \choose \bo D}\bo Z^{\bo D}_t\bo w^{\bo D}=\paren{\sum_{m=1}^dZ^{(m)}_tw(m)}^k=\big(\bo Z_t\cdot \bo w\big)^k,
\end{equation} where the first equality follows from  the multinomial theorem. This approximation accounts for the combinatorial factor appearing in  \eqref{cebolla3}. 

Thus from the above discussions we can construct, for $t>0$,  the  following probability measure 
\begin{equation}\label{defMeasureQkTrGivenAllInfoS3}
	\left.\frac{{\rm d}\Q^{(k),\bm{\theta}}_{\bo w, t,r}}{{\rm d}\p^{(k)}_r}\right|_{\F^{(k)}_t}:=\zeta^{\bm{\theta}}_{\bo w,t},
\end{equation}
which implies that
\begin{equation}\label{defMeasureQkTrGivenTopologicalInfoS3}
\left.\frac{{\rm d}\Q^{(k),\bm{\theta}}_{\bo w, t,r}}{{\rm d}\p^{(k)}_r}\right|_{\F_t}=\frac{e^{-\bm \theta\cdot \bo Z_{t}}\sum_{\bo{v}\in \mathcal{N}_{t}^{(k)}}\bo w^{\bo D_{\bo v}}}{\E_r\left[e^{-\bm \theta\cdot \bo Z_{t}}\sum_{\bo{v}\in \mathcal{N}_{t}^{(k)}}\bo w^{\bo D_{\bo v}}\right]}.
\end{equation}
%\comentario{Here I erased the extra combinatorial term}

%An important remark is that when $w(h)=1$ for every $h\in [k]$, then 
%\[
%\Q^{(\bo w),\bm{\theta}}_{T,r}=\Q^{(k),\bm{\theta}}_{T,r}. 
%\]

Similarly as in \eqref{eqnComparingQSAndQk}, we can also compare the probability measures $\Q^{(k),\bm{\theta}}_{\bo w, t,r}$ and $\Q^{(k),\bm{\theta}}_{t,r}$. Indeed, we have
\begin{equation}\label{eqnComparingQSAndQk3}
\left.\frac{{\rm d}\Q^{(k),\bm{\theta}}_{\bo w, t,r}}{{\rm d}\Q^{(k),\bm{\theta}}_{t,r}}\right|_{\F^{(k)}_t}=\frac{\bo w^{\bo D_{\bm \varsigma_t}}}{\Q^{(k),\bm \theta }_{t,r}\left[\bo w^{\bo D_{\bm \varsigma_t}}\right]}.
\end{equation}
%\comentario{Here I erased the extra combinatorial term}
The denominator above can be computed as follows. By \eqref{defMeasureQkTrGivenAllInfo} and \eqref{tomate3}, we have 
\[
\begin{split}
\Q^{(k),\bm \theta }_{t,r}\left[\bo w^{\bo D_{\bm \varsigma_t}}\right]& =\frac{\E^{(k)}_r\left[e^{-\bm \theta\cdot \bo Z_{t}}g_{k,t}\bo w^{\bo D_{\bm \varsigma_t}}\right]}{\E_r\left[N^{\floor{k}}_te^{-\bm \theta\cdot \bo Z_{t}}\right]}=\frac{\E^{(k)}_r\left[e^{-\bm \theta\cdot \bo Z_{t}}\sum_{\bo{v}\in \mathcal{N}_{t}^{(k)}}\bo w^{\bo D_{\bo v}}
\right]}{\E_r\left[N^{\floor{k}}_te^{-\bm \theta\cdot \bo Z_{t}}\right]},
\end{split}
\]
then 
\[
\left.\frac{{\rm d}\Q^{(k),\bm{\theta}}_{\bo w, t,r}}{{\rm d}\Q^{(k),\bm{\theta}}_{t,r}}\right|_{\F^{(k)}_t}=\frac{\bo w^{\bo D_{\bm \varsigma_t}}\E_r\left[N^{\floor{k}}_te^{-\bm \theta\cdot \bo Z_{t}}\right]}{\E^{(k)}_r\left[e^{-\bm \theta\cdot \bo Z_{t}}\sum_{\bo{v}\in \mathcal{N}_{t}^{(k)}}\bo w^{\bo D_{\bo v}}
\right]}.
\]
%\comentario{Here I erased the extra combinatorial term}
As one  might expect, there is also a direct connection between $\Q^{(k),\bm{\theta}}_{\bo w, t,r}$ and $\Q^{(k),\bm{\theta}}_{\bo c, t,r}$. 
For  $\bo c=(c_{1},\ldots, c_{k})\in [d]^k$,  we write $ 
\{c(\bm \varsigma_t)=\bo c\}=\{c(\varsigma^{(h)}_t)=c_{h},\forall\ h\in [k] \},$
 and recall that $\bo D_{\bo c}=(D_{\bo c, 1}, \ldots, D_{\bo c, d})$ (see the beginning of Section \ref{subsecunsamconf}). Conditionally under $\F^{(k)}_t$, summing over all the possible colour configurations of   marks  at time $t$, the right-hand side of \eqref{eqnComparingQSAndQk3} can be rewritten as 
\begin{align*}
\frac{\bo w^{\bo D_{\bm \varsigma_t}}}{\Q^{(k),\bm \theta }_{t,r}\left[\bo w^{\bo D_{\bm \varsigma_t}}\right]}&= \sum_{ \bo c\in \bo [d]^k}\frac{{k \choose \bo D_{\bo c }}\prod_{h=1}^{k}w(c^{(h)})\mathbf{1}_{\{c(\bm \varsigma_t)=\bo c\}}}{\Q^{(k),\bm \theta }_{t,r}\left[\bo w^{\bo D_{\bm \varsigma_t}}\right]}\\
& = \sum_{ \bo c\in \bo [d]^k}\frac{\Q^{(k),\bm \theta }_{t,r}\paren{c(\bm \varsigma_t)=\bo c}\prod_{h=1}^{k}w(c^{(h)})}{\Q^{(k),\bm \theta }_{t,r}\left[\bo w^{\bo D_{\bm \varsigma_t}}\right]}\frac{\mathbf{1}_{\{c(\bm \varsigma_t)=\bo c\}}}{\Q^{(k),\bm \theta }_{t,r}\paren{c(\bm \varsigma_t)=\bo c}}.
\end{align*}
%\comentario{Here I erased the extra combinatorial term}
From this identity and \eqref{eqnComparingQSAndQk}, it follows that $\Q^{(k),\bm{\theta}}_{\bo w, t,r}$ is a finite mixture of  the probability measures $\big(\Q^{(k),\bm{\theta}}_{\bo c, t,r},\bo c\in [d]^k\big)$. Indeed
\begin{equation}\label{eqnQWIsAMixtureOfQS}
\Q^{(k),\bm{\theta}}_{\bo w, t,r}(A)=\sum_{ \bo c\in \bo [d]^k} q^{(\bo w),\bo c,\bm \theta}_{t,r}\Q^{(k),\bm{\theta}}_{\bo c, t,r}(A)\qquad \forall\ A\in \F^{(k)}_t,
\end{equation}
where
\begin{equation}\label{defqbowBoc}
q^{(\bo w),\bo c,\bm \theta}_{t,r}:=\frac{\Q^{(k),\bm \theta }_{t,r}\paren{c(\bm \varsigma_t)=\bo c}\prod_{h=1}^{k}w(c^{(h)})}{\Q^{(k),\bm \theta }_{t,r}\left[\bo w^{\bo D_{\bm \varsigma_t}}\right]}.
\end{equation}
%\comentario{Here I erased the extra combinatorial term}
Clearly $\sum_{\bo c\in [d]^k}q^{(\bo w),\bo c,\bm \theta}_{t,r}=1.$
In fact from \eqref{eqnComparingQSAndQk3}, we have that
$q^{(\bo w),\bo c,\bm \theta}_{t,r}=\Q^{(k),\bm{\theta}}_{\bo w, t,r}\paren{c(\bm \varsigma_t)=\bo c}$
and hence it follows from \eqref{eqnQWIsAMixtureOfQS} that, for all $A\in \F^{(k)}_t$,
\[
\Q^{(k),\bm{\theta}}_{\bo w, t,r}\paren{ A \,|\,c(\bm \varsigma_t)=\bo c}=\Q^{(k),\bm{\theta}}_{\bo c, t,r}(A).
\]

Now we give an interpretation of the sampling under $\Q^{(k),\bm{\theta}}_{\bo w, t,r}$. 

\begin{lemma}\label{LemmaChoosingAVertexUndrQw}
We have 
\[
\Q^{(k),\bm{\theta}}_{\bo w, t,r}\paren{\left.\bm \varsigma_t=\bo v \right|\ \F_t}=\frac{\bo w^{\bo D_{\bo v}}}{\sum_{\bo v\in \mathcal{N}^{(k)}_t}\bo w^{\bo D_{\bo v}}}.
\]
\end{lemma}
%\comentario{Here I erased the extra combinatorial term}
\begin{proof}
Following the computations that lead us to \eqref{eqnQIsUniformGivenF0S} and using \eqref{eqnComparingQSAndQk3}, we deduce
\begin{align*}
	\Q^{(k),\bm{\theta}}_{\bo w, t,r}\paren{\bm \varsigma_t=\bo v \big|\ \F_t}
	&=\frac{1}{\Q^{(k),\bm \theta}_{t,r}\left[\bo w^{\bo D_{\bm \varsigma_t}}\big| \F_t\right]}\Q^{(k),\bm \theta}_{t,r}\left[\left.\mathbf{1}_{\{\bm \varsigma_t=\bo v\}} \bo w^{\bo D_{\bm \varsigma_t}}\right|\ \F_t\right].
\end{align*}To compute the denominator, we sum over all possible values of $\bm \varsigma_t$, which gives us
\begin{align*}
\Q^{(k),\bm \theta}_{t,r}\left[\bo w^{\bo D_{\bm \varsigma_t}}\Big| \F_t\right]& = \sum_{\bo v\in \mathcal{N}^{(k)}_t}\bo w^{\bo D_{\bo v}}\Q^{(k),\bm \theta}_{t,r}\paren{\bo \varsigma_t=\bo v\Big| \F_t}\\
& =\frac{1}{N^{(k)}_t}\sum_{\bo v\in \mathcal{N}^{(k)}_t}\bo w^{\bo D_{\bo v}}.
\end{align*}
%\comentario{Here I erased the extra combinatorial term}
A similar computation for the numerator implies the result. 
\end{proof}

We now provide an interpretation of  the previous lemma.  It states that 
under $\Q^{(k),\bm{\theta}}_{\bo w, t,r}$ and conditional on $\F_t$, the probability that the spines take a specific value $\bo v$ is proportional to the product of the weights associated with  the colours appearing in $\bo v$.  In other words, types with larger weights are more likely to be selected. In contrast, under $\Q^{(k),\bm \theta}_{t,r}$ and conditional on $\F_t$, the choice of $\bo v$ is uniform and therefore independent of the colours in $\bo v$.

Another important remark is that, by using the approximation  in \eqref{eqnApproximationToMultinomialDenominatorQS} with Lemma \ref{LemmaChoosingAVertexUndrQw}, we may write, for sufficiently large $t$,
\begin{equation}\label{eqnInterpretationQw}
\Q^{(k),\bm{\theta}}_{\bo w, t,r}\paren{\bm \varsigma_t=\bo v \big|\ \F_t}\sim \frac{\bo w^{\bo D_{\bo v}}}{\big(\bo Z_t\cdot \bo w\big)^k}=\frac{1}{\bo Z_t^{\bo D_{\bo v}}}\prod_{m=1}^{d}\paren{\frac{Z^{(m)}_t w(m)}{\bo Z_t\cdot \bo w}}^{D_{\bo v,m}}.
\end{equation}
%\comentario{Here I erased the extra combinatorial term}
This shows  that,  for any  vertex $\bo v$ with fixed degree  vector $\bo D_{\bo v}$, the selection probability is approximately multinomial  with parameters $k$ and 
$(\frac{Z^{(1)}_tw(1)}{\bo Z_t\cdot \bo w},\ldots, \frac{Z^{(d)}_tw(d)}{\bo Z_t\cdot \bo w}).$

Similarly as the previous two sampling schemes, we can define a new probability measure  $\p^{(k)}_{\bo w,t,r}$ as follows,
\begin{equation}\label{probuni2S3}
\left.\frac{{\rm d}\p^{(k)}_{\bo w,t,r}}{{\rm d}\p^{(k)}_r}\right|_{\F^{(k)}_t}:=\frac{1}{\p_r\paren{N_t\geq k}}\frac{g_{k,t}\bo w^{\bo D_{\bm \varsigma_t}}}{\sum_{\bo v\in \mathcal{N}^{(k)}_t}\bo w^{\bo D_{\bo v}}}. 
\end{equation}
%\comentario{Here I erased the extra combinatorial term}
That is to say the measure $\p^{(k)}_{\bo w,t,r}$  selects individuals according to the weight vector $\bo w$.
Indeed, similarly as in \eqref{eqnQIsUniformGivenF0S}, we have for $\bo v\in \mathcal{N}^{(k)}_t$ given $\F_t$, that 
\begin{equation}\label{eqnQIsUniformGivenF0S3}
\begin{split}
	\p^{(k)}_{\bo w,t,r}\paren{\left.\bm \varsigma_t=\bo v \right|\ \F_t}
	&=\frac{1}{\sum_{\bo v\in \mathcal{N}^{(k)}_t}\bo w^{\bo D_{\bo v}}}\E^{(k)}_r\left[\left. \mathbf{1}_{\{\bm \varsigma_t=\bo v \}} g_{k,t}\bo w^{\bo D_{\bm \varsigma_t}}\right|\ \F_t\right]\\
	&= \frac{\bo w^{\bo D_{\bo v}}}{\sum_{\bo v\in \mathcal{N}^{(k)}_t}\bo w^{\bo D_{\bo v}}}.
\end{split}
\end{equation}
%\comentario{Here I erased the extra combinatorial term}
From \eqref{defMeasureQkTrGivingTwoSteps}, \eqref{probuni1} and \eqref{probuni}, we observe the following identity  under the event $\{N_t\geq k\}$,
\begin{equation}\label{eqnRelationBetweenQSAndPSUnif3}
\begin{split}
	\left.\frac{{\rm d}\Q^{(k),\bm \theta}_{\bo w, t,r}}{{\rm d}\p^{(k)}_{\bo w,t,r}}\right|_{\F^{(k)}_t}& =\left.\frac{{\rm d}\Q^{(k),\bm \theta}_{\bo w, t,r}}{{\rm d}\p^{(k)}_r}\right|_{\F^{(k)}_t}\times\left.\frac{{\rm d}\p^{(k)}_r}{{\rm d}\p^{(k)}_{\bo w,t,r}}\right|_{\F^{(k)}_t}\\
& =\frac{g_{k,t}e^{-\bm \theta\cdot \bo Z_{t}}\bo w^{\bo D_{\bm \varsigma_t}}}{\E_r\left[e^{-\bm \theta\cdot \bo Z_{t}}\sum_{\bo{v}\in \mathcal{N}_{t}^{(k)}}\bo w^{\bo D_{\bo v}}\right]}\p_r\paren{N_t\geq k}\frac{\sum_{\bo v\in \mathcal{N}^{(k)}_t}\bo w^{\bo D_{\bo v}}}{g_{k,t}\bo w^{\bo D_{\bm \varsigma_t}}}\\
&=\frac{e^{-\bm \theta\cdot \bo Z_{t}}\sum_{\bo v\in \mathcal{N}^{(k)}_t}\bo w^{\bo D_{\bo v}}}{\E^{(k)}_{\bo w,t,r}\left[e^{-\bm \theta\cdot \bo Z_{t}}\sum_{\bo{v}\in \mathcal{N}_{t}^{(k)}}\bo w^{\bo D_{\bo v}}\right]}.
\end{split}
\end{equation}
%\comentario{Here I erased the extra combinatorial term}

\begin{lemma}\label{lemmaFromTheMeasureQToTheMeasureEunifS3}
Suppose $A\in \mathcal{F}^{(k)}_t$. Then
	\[
\p^{(k)}_{\bo w,t,r}\big( A \big)=\E_r\left[e^{-\bm \theta\cdot \bo Z_{t}}\sum_{\bo{v}\in \mathcal{N}_{t}^{(k)}}\bo w^{\bo D_{\bo v}}\Bigg|  N_t\geq k\right]\Q^{(k),\bm \theta}_{\bo w, t,r}\left[\frac{\mathbf{1}_{A}}{e^{-\bm \theta\cdot \bo Z_{t}}\sum_{\bo v\in \mathcal{N}^{(k)}_t}\bo w^{\bo D_{\bo v}}}\right].
\]
%\comentario{Here I erased the extra combinatorial term}
\end{lemma}
\begin{proof}
The proof follows easily from \eqref{eqnRelationBetweenQSAndPSUnif3}. 
Indeed, since $\bo Z_t$ is finite $\Q^{(k),\bm \theta}_{\bo w, t,r}$ and  $\p^{(k)}_{\bo w,t,r}$-a.s., then for any $A\in \mathcal{F}^{(k)}_t$, we have
\begin{equation}
\left.\frac{{\rm d}\p^{(k)}_{\bo w,t,r}}{{\rm d}\Q^{(k),\bm \theta}_{\bo w, t,r}}\right|_{\F^{(k)}_t}=\frac{1}{\p_r\paren{N_t\geq k}}\frac{\E_r\left[e^{-\bm \theta\cdot \bo Z_{t}}\sum_{\bo{v}\in \mathcal{N}_{t}^{(k)}}\bo w^{\bo D_{\bo v}}\right]}{e^{-\bm \theta\cdot \bo Z_{t}}\sum_{\bo v\in \mathcal{N}^{(k)}_t}\bo w^{\bo D_{\bo v}}}.
\end{equation}This proves the claim. 
%\comentario{Here I erased the extra combinatorial term}
\end{proof}

Before proceeding to the proof of the main results, we emphasize that each sampling scheme induces an auxiliary measure $\mathbb{Q}^{(k),\bm{\theta}}_{\cdot, t,r}$, under which the genealogical tree evolves as a branching process. Moreover, since the measures $\mathbb{Q}^{(k),\bm{\theta}}_{\bo c, t,r}$ and $\mathbb{Q}^{(k),\bm{\theta}}_{{\bo w},t,r}$ are absolutely continuous with respect to $\mathbb{Q}^{(k),\bm{\theta}}_{t,r}$, it is sufficient to carry out the analysis under $\mathbb{Q}^{(k),\bm{\theta}}_{t,r}$, for instance via the forward construction of the population process $\mathcal{N}$. The corresponding results for the other sampling schemes then follow by applying the appropriate Radon-Nikodym derivatives.

\section{Proof of the main results}\label{proofmain} 
\subsection{Proof of Theorem \ref{propofsplitting2Intro}}

The proof of Theorem \ref{propofsplitting2Intro} proceeds in several steps. We begin with two auxiliary results that help  characterise the event $\Delta_T(n)$ under $\Q^{(k),\bm \theta}_{\bo c,T,r}$.

To this end, we first recall the notation from Theorem \ref{propofsplitting1} and introduce some additional definitions. 
For each $h\in [n]$, let $\widetilde t_h$ denote the first time at which mark $h$ separates from all other marks. 
By construction, $\widetilde t_h = t_{v(h',m,q)}$ for some $h'\in [n]$, $m\in [d]$, and $q\in [g_{h',m}]$. 
Next, let $\widetilde m_h$ denote the type (or color) of the individual born at time $\widetilde t_h$ that carries mark $h$. 
This information is equivalently encoded in the sequence of coloured partitions $(\bo P_h; h\in [n])$.

\begin{lemma}\label{coroJointLawAndGivenSampleAreIndependent0}
For $k\geq 2$, $n\in [k-1]$, and $0=t_0<t_1<\cdots < t_n<T$, we have
\begin{equation}\label{eqnJointSplittingEventWithPartitionChildAndTypeV3}
\Q^{(k),\bm \theta }_{T,r}\Big(\Delta_T(n),c(\varsigma^{(h)}_{T})=c_h,\forall h\in [k]\Big)
=
\Q^{(k),\bm \theta }_{T,r}\Big(\Delta_T(n)\Big)
\prod_{h=1}^{k}
\Q^{(1),\bm \theta}_{T-\widetilde t_h,\widetilde m_h}\Big(c(\varsigma^{(h)}_{T-\widetilde t_h})=c_h\Big).
\end{equation}
\end{lemma}

\begin{proof}
This follows from the branching Markov property (see \cite{AHP-p1}, Lemma 3).  We only outline the main argument to o keep the length of the paper concise.

We apply the branching property at each spine splitting time. 
Observe that the event $\Delta_T(n)$ already encodes the full partition structure, and in particular determines the precise splitting time at which each mark separates from the others. 
Thus, conditional on $\Delta_T(n)$, the evolution after each splitting time is independent across marks. 
Arguing as in \cite{AHP-p1}, Propositions 2 and 3, we obtain the desired factorization.
\end{proof}

Next, we express the terms in the product on the right-hand side of \eqref{eqnJointSplittingEventWithPartitionChildAndTypeV3} in terms of the original law of $\bo Z$.

\begin{lemma}\label{lemmaLimitUnderQColorOfSingleSpineIsGiven0}
For any $t\in [0,T)$, $m,c_h\in [d]$, and $h\in [k]$, we have
\begin{equation}\label{eqnLimitUnderQColorOfSingleSpineIsGiven}
\Q^{(1),\bm \theta}_{T-t,m}\Big(c(\varsigma^{(h)}_{T-t})=c_h\Big)
=
\frac{\E_m[Z^{(c_h)}_{T-t}e^{-\bm \theta \cdot \bo Z_{T-t}}]}
{\E_m\left[N_{T-t}e^{-\bm \theta\cdot \bo Z_{T-t}}\right]}.
\end{equation}
\end{lemma}

\begin{proof}
From \eqref{defMeasureQkTrGivingTwoSteps}, we have
\[
\Q^{(1),\bm \theta}_{T-t,m}\Big(c(\varsigma^{(h)}_{T-t})=c_h\Big)
=
\frac{\E_m\left[g_{k,T-t}\mathbf{1}_{\{c(\varsigma^{(h)}_{T-t})=c_h\}}\right]}
{\E_m\left[N_{T-t}e^{-\bm \theta \cdot \bo Z_{T-t}}\right]}.
\]
Conditioning on $\mathcal{F}_{T-t}$ and arguing as in \eqref{tomate1}, we obtain
\[
\E_m\left[g_{k,T-t}\mathbf{1}_{\{c(\varsigma^{(h)}_{T-t})=c_h\}}\right]
=
\E_m\left[Z^{(c_h)}_{T-t}e^{-\bm \theta \cdot \bo Z_{T-t}}\right],
\]
which yields the claim.
\end{proof}

The following proposition is instrumental in the proof of Theorem \ref{propofsplitting2Intro}.

\begin{propo}\label{teoJointLawUnderQkbutnewver}
For any $k\geq 1$, $T\in \re_+$, $r\in [d]$, $\bm \theta\in \mathbb{R}^d_+$, and $\bo c=(c_1,\ldots, c_k)\in [d]^k$, we have
\[
\Q^{(k),\bm \theta}_{\bo c, T,r}\Big(\Delta_T(n)\Big)
=
\Q^{(k),\bm \theta }_{T,r}\Big(\Delta_T(n)\Big)
\frac{\E_{r}\left[N^{\floor{k}}_{T}e^{-\bm \theta\cdot \bo Z_{T}}\right]}
{\E_{r}[\bo Z^{\floor{\bo D_{\bo c}}}_{T}e^{-\bm \theta \cdot \bo Z_{T}}]}
\prod_{h=1}^{k}
\frac{\E_{\widetilde m_h}[Z^{(c_h)}_{T-\widetilde t_h}e^{-\bm \theta\cdot \bo Z_{T -\widetilde t_h}}]}
{\E_{\widetilde m_h}\left[N_{T -\widetilde t_h}e^{-\bm \theta\cdot \bo Z_{T -\widetilde t_h}}\right]}.
\]
\end{propo}

\begin{proof}
Combining \eqref{eqnComparingQSAndQk}, \eqref{eqnJointSplittingEventWithPartitionChildAndTypeV3}, and Lemmas 
\ref{coroJointLawAndGivenSampleAreIndependent0} and \ref{lemmaLimitUnderQColorOfSingleSpineIsGiven0}, we obtain
\[
\begin{split}
\Q^{(k),\bm \theta}_{\bo c, T,r}\Big(\Delta_T(n)\Big)
&=
\Q^{(k),\bm \theta }_{T,r}\Big(\Delta_T(n)\Big)
\frac{\E_{r}\left[N^{\floor{k}}_{T}e^{-\bm \theta\cdot \bo Z_{T}}\right]}
{\E_{r}[\bo Z^{\floor{\bo D_{\bo c}}}_{T}e^{-\bm \theta \cdot \bo Z_{T}}]} 
\prod_{h=1}^{k}
\frac{\E_{\widetilde m_h}[Z^{(c_h)}_{T-\widetilde t_h}e^{-\bm \theta \cdot \bo Z_{T -\widetilde t_h}}]}
{\E_{\widetilde m_h}\left[N_{T -\widetilde t_h}e^{-\bm \theta\cdot \bo Z_{T -\widetilde t_h}}\right]},
\end{split}
\]
as claimed.
\end{proof}

We now turn to the proof of Theorem \ref{propofsplitting2Intro}.

\begin{proof}[Proof of Theorem \ref{propofsplitting2Intro}]
We prove the result for an arbitrary event $A\in \mathcal{F}^{(k)}_T$; the statement then follows by taking $A=\Delta_T(n)$.

Starting from Lemma~\ref{lemmaFromTheMeasureQToTheMeasureEunifS}, we apply the Beta integral identity
\begin{equation}\label{gamma+beta}
\frac{\Gamma(k)}{x^{\floor{k}}}
=
\int_0^\infty (e^y-1)^{k-1}e^{-yx}\,{\rm d}y,
\end{equation}
componentwise to each factor $(Z^{(m)}_T)^{D_{\bo c,m}}$ in the denominator. 
This yields
\[
\begin{split}
\Q^{(k),\bm \theta}_{\bo c,T,r}\left[\frac{\mathbf{1}_{A}}{\bo Z^{\floor{\bo D_{\bo c}}}_{T}e^{-\bm \theta\cdot \bo Z_{T}}}\right]
&=
\left(\prod_{m\in S_{\bo c}}\frac{1}{(D_{\bo c,m}-1)!}\right)
\int_{\mathbb{R}^{|S_{\mathbf{c}}|}_+}(e^{\bm\phi}-\vec{1})^{\bo D_{\bo c}-\vec{1}} 
\Q^{(k),\bm \theta }_{\bo c, T,r}\left[\mathbf{1}_{A}e^{(\bm \theta-\bm \phi) \cdot \bo Z_{T}}\right]
\,{\rm d}\bm \phi.
\end{split}
\]

Using the change of measure \eqref{defMeasureQkTrGivenAllInfoS}, we observe that for any $\bm \phi\in \mathbb{R}^d_+$,
\[
\Q^{(k),\bm \theta}_{\bo c,T,r}\left[\mathbf{1}_{A}e^{(\bm \theta-\bm \phi)\cdot \bo Z_{T}}\right]
=
\Q^{(k),\bm \phi}_{\bo c, T,r}(A)
\frac{\E_{r}\left[\bo Z^{\floor{\bo D_{\bo c}}}_Te^{-\bm \phi\cdot \bo Z_T}\right]}
{\E_{r}\left[\bo Z^{\floor{\bo D_{\bo c}}}_Te^{-\bm \theta\cdot \bo Z_T}\right]}.
\]

Substituting this identity and combining with Lemma~\ref{lemmaFromTheMeasureQToTheMeasureEunifS} completes the proof.
\end{proof}

\subsection{Proof of Theorem \ref{propofsplitting3Intro}}

We begin with the following result, which plays a key role in the proof of Theorem \ref{propofsplitting3Intro} and characterizes the event $\Delta_T(n)$ under $\Q^{(k),\bm \theta}_{\bo w,T,r}$. Its proof is an immediate consequence of \eqref{eqnQWIsAMixtureOfQS} and Proposition \ref{teoJointLawUnderQkbutnewver}.

\begin{coro}\label{teoLimitQWEqualsLimitQk}
For any $k\geq 1$, $T\in \mathbb{R}_+$, $r\in [d]$, $\bm \theta\in \mathbb{R}^d_+$, and weights $\bo w=(w(1),\ldots,w(d))\in \mathbb{R}^d_+$, we have
\[
\begin{split}
\Q^{(k),\bm \theta}_{\bo w,T,r}\Big(\Delta_T(n)\Big)
&= \sum_{\bo c\in [d]^k} q^{(\bo w),\bo c,\bm \theta}_{T,r}
\,\Q^{(k),\bm \theta}_{T,r}\Big(\Delta_T(n)\Big) \\
&\qquad\qquad \times
\frac{\E_{r}\left[N^{\lfloor k \rfloor}_{T}e^{-\bm \theta\cdot \bo Z_{T}}\right]}
{\E_{r}\left[\bo Z^{\lfloor \bo D_{\bo c}\rfloor}_{T}e^{-\bm \theta \cdot \bo Z_{T}}\right]}
\prod_{h=1}^{k}
\frac{\E_{\widetilde m_h}\left[Z^{(c_h)}_{T-\widetilde t_h}e^{-\bm \theta \cdot \bo Z_{T -\widetilde t_h}}\right]}
{\E_{\widetilde m_h}\left[N_{T -\widetilde t_h}e^{-\bm \theta\cdot \bo Z_{T -\widetilde t_h}}\right]}.
\end{split}
\]
\end{coro}
We now turn to the proof of Theorem \ref{propofsplitting3Intro}.
\begin{proof}[Proof of Theorem \ref{propofsplitting3Intro}]
Arguing as in the proof of Theorem \ref{propofsplitting2Intro}, it suffices to establish the result for an arbitrary event $A\in \mathcal{F}^{(k)}_T$; the claim then follows by taking $A=\Delta_T(n)$.

Recalling Lemma \ref{lemmaFromTheMeasureQToTheMeasureEunifS3}, we apply the Gamma integral formula
\[
 \frac{\Gamma(k)}{x^{k}}=\int_0^\infty y^{k-1}e^{-yx}\,{\rm d}y,
\]
together with Fubini's theorem to the second term on the right-hand side. This yields
\[
\begin{split}
 \Q^{(k),\bm \theta}_{\bo w, T,r}&\left[\frac{\mathbf{1}_{A}e^{\bm \theta\cdot \bo Z_{T}}}{\sum_{\bo v\in \mathcal{N}^{(k)}_T}\bo w^{\bo D_{\bo v}}}\right] =
\Q^{(k),\bm \theta}_{\bo w, T,r}\left[
\frac{(\bo Z_T\cdot \bo w)^k}{\sum_{\bo v\in \mathcal{N}^{(k)}_T}\bo w^{\bo D_{\bo v}}}
\cdot
\frac{\mathbf{1}_{A}e^{\bm \theta\cdot \bo Z_{T}}}{(\bo Z_T\cdot \bo w)^k}
\right] \\
&\hspace{2.5cm} =
\frac{1}{(k-1)!}\int_0^\infty \phi^{k-1}
\Q^{(k),\bm \theta}_{\bo w, T,r}\left[
\frac{(\bo Z_T\cdot \bo w)^k}{\sum_{\bo v\in \mathcal{N}^{(k)}_T}\bo w^{\bo D_{\bo v}}}
\mathbf{1}_{A}e^{(\bm \theta-\phi\bo w)\cdot \bo Z_{T}}
\right]{\rm d}\phi.
\end{split}
\]
%\simon{Here I erased the extra combinatorial term}
Next, we use the change of measure \eqref{defMeasureQkTrGivenAllInfoS3}. For any $\bm \mu \in \mathbb{R}^d_+$, we have
\[
\begin{split}
\Q^{(k),\bm \theta}_{\bo w, T,r}&\left[
\frac{(\bo Z_T\cdot \bo w)^k \mathbf{1}_{A}e^{(\bm \theta-\bm \mu)\cdot \bo Z_{T}}}
{\sum_{\bo v\in \mathcal{N}^{(k)}_T}\bo w^{\bo D_{\bo v}}}
\right]  \\
&\qquad\qquad =
\Q^{(k),\bm \mu}_{\bo w, T,r}\left[
\frac{(\bo Z_T\cdot \bo w)^k \mathbf{1}_{A}}
{\sum_{\bo v\in \mathcal{N}^{(k)}_T}\bo w^{\bo D_{\bo v}}}
\right]
\frac{\E_{r}\left[e^{-\bm \mu\cdot \bo Z_T}\sum_{\bo v\in \mathcal{N}^{(k)}_T}\bo w^{\bo D_{\bo v}}\right]}
{\E_{r}\left[e^{-\bm \theta\cdot \bo Z_T}\sum_{\bo v\in \mathcal{N}^{(k)}_T} w^{\bo D_{\bo v}}\right]}.
\end{split}
\]

Combining the above identities yields the desired result.
\end{proof}

\subsection{Asymptotic formulas for the first splitting event}\label{sectionAsymptoticFpormulas}
In this subsection,  we rescale our multitype tree with $k$-spines by a factor of $T^{-1}$ and  replace the parameter  $\bm \theta$ with 
\begin{equation}\label{deftheta}
{\bm \theta}_T:=\bm \theta \sum_{m\in [d]}\p_{m}\paren{\bo Z_T\ne 0}.
\end{equation}
In other words, under $\Q^{(k),\bm \theta_T}_{\cdot, T,r}$, the resulting tree, together with its entire genealogical structure, has lifespan one. The purpose of this subsection is to establish the fundamental tools required to derive the asymptotic behaviour of the event $\Delta_T(n)$, under $\Q^{(k),\bm \theta_T}_{\cdot, T,r}$,  as $T\to \infty$.

We start with the first splitting time for the uniform sampling scheme. 
Take $\rho\in(0,1)$ and we replace $\tau_1$ and $\theta$ in \cite{AHP-p1}, Corollary 3, with  $\tau_1/T$  and $\bm \theta_T$ respectively, that is
\begin{equation}\label{eqnJointLawSplittingTimeChildrenPartitionScaledBeforeTheLimit}
\begin{split}
	\Q^{(k),\bm \theta_T}_{T,r}&\left(\frac{\tau_1}{T}\in {\rm d}\rho, \#\mathcal{P}_{\rho}=(a_{m,q})_{m\in[d], q\in [g_m]}, {\bo L}_{\rho}=\bm{\ell} , c(\varsigma^{(1)}_{\rho-})=i\right)\\
&
=
\frac{k!}{\prod_{m=1}^d\prod_{q=1}^{g_m} a_{m,q}!\prod_{n\geq 1}d_{m,n}!}\alpha_{i}\E_r\left[\bo L^{\floor{\bo g}}\prod_{m=1}^d\E_m\left[e^{-\bm \theta_T \cdot \bo Z_{T(1-\rho)}}\right]^{L^{(m)}-g_m}\right] 
\\
	&\times\frac{\bm \ell^{\floor{\bo g}}p_{i}(\bm \ell)\prod_{m=1}^d \E_m\left[e^{-\bm \theta_T \cdot \bo Z_{T(1-\rho)}}\right]^{\ell_m-g_m}}{\E_r\left[\bo L^{\floor{\bo g}}\prod_{m=1}^d\E_m\left[e^{-\bm \theta_T \cdot \bo Z_{T(1-\rho)}}\right]^{L^{(m)}-g_m}\right]}\E_r\left[ Z^{(i)}_{\rho T}\prod_{m=1}^d\E_m\left[e^{-\bm \theta_T \cdot \bo Z_{T(1-\rho)}}\right]^{Z^{(m)}_{\rho T}-\delta_{i,m}}\right]\\
&\hspace{5.8cm} \times \frac{\prod_{m=1}^d\prod_{q=1}^{g_m} \E_m\left[N^{\floor{a_{m,q}}}_{T(1-\rho)}e^{-\bm \theta_T \cdot \bo Z_{T(1-\rho)}}\right]}{\E_r\left[N^{\floor{k}}_{T}e^{-\bm \theta_T\cdot \bo Z_{T}}\right]}T{\rm d} \rho ,
\end{split}
\end{equation}
where we recall that ${\bo L}_{\rho}$ denotes the offspring distribution of the vertex involved in the first spine splitting event, here denoted by $v(1)$, and $d_{m,n}={\rm card}\{q:a_{m,q}=n \}$ . %To ensure clarity, we will use the notation ${\bo L}_{v(1)}$ later on when considering multiple spine-splitting events, see for instance Lemma  \ref{lemmaLimitOfJointLaw}.

Before we state our next result, we recall that
\[
{\bm\ell}^{\floor{\bo g}}{\bm \xi}^{\bo g}=\prod_{\substack{m=1,\\ g_m\ne 0}}^d (\ell_m)^{\floor{g_m}}\prod_{m=1}^d\xi_m^{g_{m}}.
\]
The following result identifies the limit of the above expression and reveals that, in the limit, only binary splittings are possible.

\begin{teo}\label{teoJointLawSplittingTimeChildrenPartitionScaledBeforeTheLimit}
Consider $k\geq 2$ and assume that assumption \eqref{hyprifa} is fulfilled. For $(a_{m,q})_{m\in [d],q\in [g_m]}$ such that either
\begin{itemize} 
\item[i)] there exist $m^*$ with $a_{m^*,1}>0$, $a_{m^*,2}>0$ and $a_{m^*,1}+a_{m^*,2}=k$ or 
\item[ii)] there are $m^*\ne \tilde{m}$ such that $a_{m^*,1}>0, a_{\tilde{m},1}>0$ with $a_{m^*,1}+a_{\tilde{m},1}=k$, 
\end{itemize}
then, 
\begin{equation}\label{limitQ}
\begin{split}
	\lim_{T\to \infty}\Q^{(k),\bm \theta_T}_{T,r}&\paren{\frac{\tau_1}{T}\in {\rm d}\rho ,\#\mathcal{P}_{\rho}=(a_{m,q})_{m\in[d], q\in [g_m]},{\bo L}_\rho=\bm{\ell} , c(\varsigma^{(1)}_{\rho-})=i}\\
	%& \hspace{2cm}= \frac{\alpha_i\eta_i p_i(\bm \ell)\bm \ell^{\floor{\bo g}}{\bm \xi}^{\bo g}}{1+\mathbf{1}_{\{\exists\, m^*: a_{m^*,1}=a_{m^*,2}\}}} (1-\rho)^{k-2}\frac{2}{\zeta} \frac{\paren{1+\bm \eta \cdot \bm \theta}^{k-1}}{ \paren{1+(1-\rho)\bm \eta \cdot \bm \theta}^{k}}{\rm d}\rho.\\
& 
\hspace{2cm}= \alpha_i\eta_i p_i(\bm \ell)\bm \ell^{\floor{\bo g}}{\bm \xi}^{\bo g} (1-\rho)^{k-2}\frac{2}{\zeta} \frac{\paren{1+\bm \eta \cdot \bm \theta}^{k-1}}{ \paren{1+(1-\rho)\bm \eta \cdot \bm \theta}^{k}}{\rm d}\rho.
\end{split}
\end{equation} 
Furthermore, for any other choice of $(a_{m,q})_{m\in [d],q\in [g_m]}$, the previous limit equals 0.
\end{teo}
It is important to note that in order to have existence of the limit we necessarily have that $\bo g$ is such that $\sum_{m=1}^d g_m=2$.  Also, note that ${\bm \ell}^{\floor{\bo g}}{\bm \xi}^{\bo g}$ can be simplified as follows
\[
{\bm\ell}^{\floor{\bo g}}{\bm \xi}^{\bo g}=\left\{
\begin{array}{ll}
\ell_{m^*}(\ell_{m^*}-1)\xi^2_{m^*}&  \textrm{in the first case (i),}\\
\ell_{m^*}\ell_{\tilde{m}}\xi_{m^*}\xi_{\tilde{m}}&  \textrm{in the second case (ii).}
\end{array}
\right .
\]

The proof will consist of three asymptotic formulas, which we state as separate Lemmas for simplicity of exposition. 
The latter should be clear from \eqref{eqnJointLawSplittingTimeChildrenPartitionScaledBeforeTheLimit}, since we need to study the asymptotic behaviours  of 
\[
\E_m\left[e^{-\bm \theta_T\cdot \bo Z_{\rho T}}\right],\qquad \E_r\left[ Z^{(i)}_{\rho T}\prod_{m=1}^d\E_m\left[e^{-\bm \theta_T\cdot \bo Z_{T(1-\rho)}}\right]^{Z^{(m)}_{\rho T}-\delta_{i,m}}\right]\quad \textrm{and}\quad 
\E_m\left[N^{\floor{k}}_{\rho T}e^{-\bm \theta_T \cdot \bo Z_{\rho T}}\right],
\]
for any $\rho\in (0,1)$, $m\in [d]$ and $i\in [d]$.

To analyse such asymptotic behaviours, we employ an asymptotic expression for the extinction probability of a MBGW process, under  \eqref{hyprifa}, which we recall below.  From \cite{Penisson2010}, Equation (2.3.12), in Proposition 2.3.4, under our assumptions,  for any $\bo u\in [0,1]^d$ and $m\in [d]$ it holds that
\begin{equation}\label{eqnAsymptoticApproximationOfLaplaceTransformOfMGW}
	F_{t,m}(\bo u)=\E_m\left[\bo u^{\bo Z_{t}}\right]\sim 1-\frac{\xi_m\bm \eta \cdot (\vec{1}-\bo u)}{1+\frac{\zeta }{2}\bm \eta \cdot (\vec{1}-\bo u)t},\qquad \mbox{as }t\to\infty,
\end{equation} 
we also refer  
%Satz 6.3.1 \cite{MR0408019}.
to \cite{MR0408019}, Section 6.3.1.

Recalling that $\bo e_j$ denotes the vector in $\mathbb{R}^d$ with  value 1 in the $j$-th coordinate and $0$ in all other coordinates,  
the last expression implies that for $t$ large enough
\[
\p_{m}\left(Z^{(j)}_t>0\right)=1-\E_{m}\left[(\vec{1}- \bo e_j)^{\bo Z_t}\right]\sim \frac{\xi_m \bm \eta\cdot  \bo e_j}{1+\frac{\zeta }{2}\bm \eta \cdot \bo e_j t}=\frac{\xi_m \eta_j}{1+\frac{\zeta }{2}\eta_j t}.
\]As $t\to \infty$, the latter is asymptotically
\[
\p_{m}\paren{Z^{(j)}_t>0}\sim \frac{2\xi_m}{\zeta }\frac{1}{t},
\]which does not depend on $j$. 
Similarly, 
\begin{equation}\label{eqnAsymptoticOfProbabilityOfNonExtinction}
\p_{m}\paren{\bo Z_t\neq \bo 0}\sim \frac{\xi_m \bm \eta \cdot \vec{1}}{1+\frac{\zeta }{2}(\bm \eta \cdot \vec{1}) t}\sim \frac{2\xi_m }{\zeta }\frac{1}{t},
\end{equation}
implying that  
\[
\bm \theta_T\sim\frac{2\bm \theta}{\zeta T}, \qquad \mbox{as }\quad T\to\infty,
\]
which follows from its definition, see \eqref{deftheta}. 
Hence, for large $t$, $\p_{m}\paren{\bo Z_t\neq \bo 0}\sim\p_{m}({Z^{(j)}_t>0})$.
We also note that, for any functional of $(\bo Z_s, s\le t)$ of the form $Z^{(j)}_t G(\bo Z_s, s\le t)$, we can either condition on the whole branching process to be non-extinct or on the $j$-th coordinate, since
%\begin{align*}
%	& \E_{\bo e_i}\paren{Z^{(j)}_t\bo F_t(\bo Z)\left|\ Z^{(j)}_t>0\right.}\\
%	& =\frac{\E_{\bo e_i}\paren{Z^{(j)}_t\bo F_t(\bo Z);\ Z^{(j)}_t>0}}{\p_{\bo e_i}\paren{Z^j_t>0}} \\
%%	& \sim \frac{\E_{\bo e_i}\paren{\bo s^{\bo Z_t}F_t(\bo Z);\ Z^j_t>0, Z^k_t>0\mbox{ for some }j\in [d]}}{\p_{\bo e_i}\paren{\bo Z_t>0}} \\
%	& \sim \frac{\E_{\bo e_i}\paren{Z^{(j)}_t\bo F_t(\bo Z);Z^{(j)}_t>0,\bo Z_t\neq \bo 0}}{\p_{\bo e_i}\paren{\bo Z_t\neq \bo 0 }} +\frac{\E_{\bo e_i}\paren{Z^{(j)}_t\bo F_t(\bo Z);\ Z^{(j)}_t>0, \bo Z_t=\bo 0}}{\p_{\bo e_i}\paren{\bo Z_t\neq \bo 0}}\\
%	& = \E_{\bo e_i}\paren{\left.Z^{(j)}_t\bo F_t(\bo Z);Z^{(j)}_t>0\ \right| \bo Z_t\neq \bo 0}\\
%	& = \E_{\bo e_i}\paren{\left.Z^{(j)}_t\bo F_t(\bo Z) \right|\ \bo Z_t\neq \bo 0}.
%\end{align*}
\begin{equation}\label{eqnChangeFromUnconditionedToConditionedWhenAcomponentMultiplies}
	\begin{split}
		 \E_{m}\left[Z^{(j)}_t G(\bo Z_s, s\le t)\right] 
		 & = \E_{m}\left[\left.Z^{(j)}_t G(\bo Z_s, s\le t)\ \right|  Z^{(j)}_t\neq 0\right]\p_{m}( Z^{(j)}_t\neq 0)\\
		& = \E_{m}\left[\left.Z^{(j)}_t G(\bo Z_s, s\le t)\ \right| \bo Z_t\neq \bo 0\right]\p_{m}(\bo Z_t\neq \bo 0). \\
	\end{split}
\end{equation}

In view of the preceding discussion, we are now in a position to derive asymptotic expressions for the three terms appearing in \eqref{eqnJointLawSplittingTimeChildrenPartitionScaledBeforeTheLimit}.  The following lemma provides an estimate of the Laplace transform of a MBGW process.

\begin{lemma}\label{lemmaLaplaceTransformOfMtypeBranchingProcess}
Let $\rho\in [0,1)$. Under assumption \eqref{hyprifa}, we have 
\begin{equation}\label{eqnLaplaceTransformOfMtypeBranchingProcess}
	\E_m\left[e^{-\bm \theta_T \cdot \bo Z_{T(1-\rho)}}\right]\sim 1-\frac{2\xi_m}{\zeta}\frac{\bm \eta\cdot \bm \theta}{1+(1-\rho)\bm \eta\cdot \bm \theta}\frac{1}{T}, \qquad \mbox{as }\quad T\to\infty.
\end{equation}
\end{lemma} 
\begin{proof}
Decomposing on whether the process survives or not at time $T(1-\rho)$, we have
\begin{align*}
	 \E_m\left[e^{-\bm \theta_T \cdot \bo Z_{T(1-\rho)}}\right]& = \E_m\left[e^{-\bm \theta_T \cdot \bo Z_{T(1-\rho)}} \Big|\bo Z_{T(1-\rho)}\neq \bo 0\right]\p_{m}\paren{\bo Z_{T(1-\rho)}\neq \bo 0}+\p_{m}\paren{\bo Z_{T(1-\rho)}=\bo 0}\\
	& =1-\p_{m}\paren{\bo Z_{T(1-\rho)}\neq \bo 0}\paren{1-\E_m\left[e^{-\bm \theta_T \cdot \bo Z_{T(1-\rho)}} \Big|\bo Z_{T(1-\rho)}\neq \bo 0\right]}.
\end{align*}
On the other hand from Yaglom's limit, see  Proposition \ref{propSophie}, we have 
\begin{align*}
	\E_m\left[e^{-\bm \theta_T \cdot \bo Z_{T(1-\rho)}} \Big|\bo Z_{T(1-\rho)}\neq 0\right] \xrightarrow[t\to \infty]{} \E \left[\exp\Big\{-(1-\rho)\gamma\,\bm \eta\cdot \bm \theta\Big\}\right]=\frac{1}{1+(1-\rho)\bm \eta\cdot \bm \theta}.
\end{align*}From the latter and  the asymptotic in  \eqref{eqnAsymptoticOfProbabilityOfNonExtinction}, we obtain
\begin{align*}
	\E_m\left[e^{-\bm \theta_T \cdot \bo Z_{T(1-\rho)}}\right]& \sim 1-\frac{\xi_m \bm \eta\cdot \vec{1}}{1+\frac{\zeta }{2}T(1-\rho)\bm \eta\cdot \vec{1}}\frac{(1-\rho)\bm \eta\cdot \bm \theta}{1+(1-\rho)\bm \eta\cdot \bm \theta}\\
	& \sim 1-\frac{2\xi_m}{\zeta}\frac{\bm \eta\cdot \bm \theta}{1+(1-\rho)\bm \eta\cdot \bm \theta}\frac{1}{T}\qquad \mbox{as }\quad T\to\infty,
\end{align*}
as expected.
\end{proof}

The next lemma provides control over the branching-off the spines terms.
Since it will be applied in different contexts and not only to the limit of Theorem \ref{teoJointLawSplittingTimeChildrenPartitionScaledBeforeTheLimit}, we present it in a more general form. 

\begin{lemma}
Let $\rho_0,\rho_1\in [0,1)$ with $\rho_0<\rho_1$, and $r,i_1\in [d]$.
 Under assumption \eqref{hyprifa}, as $T$ increases, we have 
\begin{equation}\label{eqnProductOfLaplaceTransformsShS0OfMtypeBranchingProcess}
	\E_{r}\left[Z^{(i_1)}_{(\rho_1-\rho_0)T}\prod_{m=1}^{d}\E_m \left[e^{-\bm \theta_{T}\cdot \bo Z_{T(1-\rho_1)}}\right]^{Z^{(m)}_{(\rho_1-\rho_0)T}}\right]\sim \eta_{i_1}\xi_{r} \paren{\frac{1+(1-\rho_1)\bm \eta\cdot \bm \theta}{1+(1-\rho_0)\bm \eta\cdot \bm \theta}}^{2}.
\end{equation}
\end{lemma}
\begin{proof}
	From \eqref{eqnAsymptoticOfProbabilityOfNonExtinction}, \eqref{eqnChangeFromUnconditionedToConditionedWhenAcomponentMultiplies} and \eqref{eqnLaplaceTransformOfMtypeBranchingProcess}, we observe that, for $T$ large enough,
\begin{align*}
	\E_{r}&\left[Z^{(i_1)}_{(\rho_1-\rho_0)T}\prod_{m=1}^d\E_m\left[e^{-\bm \theta_T\cdot \bo Z_{T(1-\rho_1)}}\right]^{Z^{(m)}_{(\rho_1-\rho_0)T}}\right]\\
%	& =\E_{i_0}\left(\left.Z^{(i_1)}_{(\rho_1-\rho_0)T}\prod_{m=1}^d\E_m\paren{e^{-\bm \theta_T\cdot \bo Z_{T(1-\rho_1),m}}}^{Z^{(m)}_{(\rho_1-\rho_0)T}}\right|Z^{(i_1)}_{(\rho_1-\rho_0)T}>0\right)\p_{i_0}\paren{Z^{(i_1)}_{(\rho_1-\rho_0)T}>0}\\
%	& \sim \E_{i_0}\left(\left.Z^{(i_1)}_{(\rho_1-\rho_0)T}\prod_{m=1}^d\E_m\paren{e^{-\bm \theta_T\cdot \bo Z_{T(1-\rho_1),m}}}^{Z^{(m)}_{(\rho_1-\rho_0)T}}\right|Z^{(i_1)}_{(\rho_1-\rho_0)T}>0\right)\p_{i_0}\paren{Z_{(\rho_1-\rho_0)T}>0}\\
%	& \sim \E_{i_0}\left(\left.Z^{(i_1)}_{(\rho_1-\rho_0)T}\indi{Z^{(i_1)}_{(\rho_1-\rho_0)T}>0}\prod_{m=1}^d\E_m\paren{e^{-\bm \theta_T\cdot \bo Z_{T(1-\rho_1),m}}}^{Z^{(m)}_{(\rho_1-\rho_0)T}}\right|Z_{(\rho_1-\rho_0)T}>0\right)\p_{i_0}\paren{Z_{(\rho_1-\rho_0)T}>0}\\
	& =\E_{r}\left[Z^{(i_1)}_{(\rho_1-\rho_0)T}\prod_{m=1}^d\E_m \left[e^{-\bm \theta_T\cdot \bo Z_{T(1-\rho_1)}}\right]^{Z^{(m)}_{(\rho_1-\rho_0)T}}\Big|\bo Z_{(\rho_1-\rho_0)T}\neq \bo 0\right]\p_{r}\paren{\bo Z_{(\rho_1-\rho_0)T}\neq \bo 0}\\
	%& \sim \E_{r}\left[Z^{(i_1)}_{(\rho_1-\rho_0)T}\prod_{m=1}^d\E_m \left[e^{-\bm \theta_T\cdot \bo Z_{T(1-\rho_1)}}\right]^{Z^{(m)}_{(\rho_1-\rho_0)T}}\Big|\bo Z_{(\rho_1-\rho_0)T}\neq \bo 0\right]\frac{\xi_{r} }{\frac{\zeta }{2}(\rho_1-\rho_0)T}\\
	%& \sim \E_{r}\left[Z^{(i_1)}_{(\rho_1-\rho_0)T}\prod_{m=1}^d\paren{1-\frac{c_{1-\rho_1,m}}{T}}^{Z^{(m)}_{(\rho_1-\rho_0)T}}\Big|\bo Z_{(\rho_1-\rho_0)T}\neq \bo 0 \right]\frac{\xi_{r} }{\frac{\zeta }{2}(\rho_1-\rho_0)T}\\
	& \sim \E_{r}\left[\frac{Z^{(i_1)}_{(\rho_1-\rho_0)T}}{(\rho_1-\rho_0)T}\prod_{m=1}^d\paren{1-\frac{c_{1-\rho_1,m}}{T}}^{\frac{Z^{(m)}_{(\rho_1-\rho_0)T}}{(\rho_1-\rho_0)T}(\rho_1-\rho_0)T}\Big|\bo Z_{(\rho_1-\rho_0)T}\neq \bo 0\right]\frac{2}{\zeta}\xi_{r}\\
	& \sim \E\left[\frac{\zeta}{2}\eta_{i_1} \gamma\,\prod_{m=1}^d\exp\left\{-(\rho_1-\rho_0)c_{1-\rho_1,m}\frac{\zeta}{2}\eta_m \gamma\right\} \right]\frac{2}{\zeta}\xi_{r} \\
	& =\xi_{r}\eta_{i_1}\E\left[ \gamma\,\exp\left\{-\gamma(\rho_1-\rho_0)\frac{\zeta}{2}\sum_{m=1}^d c_{1-\rho_1,m}\eta_m \right\} \right],
\end{align*}
 where, for $x> 0$,
\[
c_{x,m}:=\frac{\xi_m}{\frac{\zeta}{2}}\frac{\bm \eta\cdot \bm \theta}{1+x\bm \eta\cdot \bm \theta}.
\]
Recalling that $\bm \eta\cdot \bm \xi =1$, then
\[
\frac{\zeta}{2}\sum_{m=1}^d c_{1-\rho_1,m}\eta_m=\frac{\bm \eta\cdot \bm \theta}{1+(1-\rho_1)\bm \eta\cdot \bm \theta},
\]implying, when $T\to\infty$, that
\begin{align*}
	\E_{r}&\left[Z^{(i_1)}_{(\rho_1-\rho_0)T}\prod_{m=1}^d\E_m\left[e^{-\bm \theta_{T}\cdot \bo Z_{T(1-\rho_1)}}\right]^{Z^{(m)}_{(\rho_1-\rho_0)T}}\right]\\
&\qquad\qquad \qquad  \sim \xi_{r}\eta_{i_1}\E\left[ \gamma\exp\left\{-\gamma(\rho_1-\rho_0) \frac{\bm \eta\cdot \bm \theta}{1+(1-\rho_1)\bm \eta\cdot \bm \theta}\right\}\right]\\
	&\qquad \qquad \qquad = \xi_{r}\eta_{i_1} \paren{1+(\rho_1-\rho_0) \frac{\bm \eta\cdot \bm \theta}{1+(1-\rho_1)\bm \eta\cdot \bm \theta}}^{-2}.
\end{align*}This completes the proof.
%Here we used $\E\left[e^{-\lambda\gamma}\right]=(1+\lambda)^{-1}$ for any $\lambda\geq 0$ and so $\E\left[\gammae^{-\lambda\gamma}\right]=(1+\lambda)^{-2}$. 
\end{proof}

The next lemma gives us an asymptotic expression for the terms corresponding to splitting events.

\begin{lemma}
	Let $\rho\in [0,1)$.
 Under assumption \eqref{hyprifa}, for any $k\in \na$, as $T$ increases, we have 
\begin{equation}\label{eqnDiscountedPowerOfSizeOfGaltonWatson}\begin{split}
&\E_m\left[N_{(1-\rho)T}^{\floor{k}}e^{-\bm \theta_T \cdot \bo Z_{(1-\rho)T}}\right] \sim \paren{(1-\rho)T\frac{\zeta}{2}}^{k-1} \frac{k!\xi_m(\bm \eta\cdot \vec{1})^k}{\paren{1+(1-\rho)\bm \eta\cdot \bm \theta }^{k+1}}.
\end{split}\end{equation}
\end{lemma}
\begin{proof}
	
First, we decompose 
\begin{align*}
& \E_m \left[N_{(1-\rho)T}^{\floor{k}} e^{-\bm \theta_T \cdot \bo Z_{(1-\rho)T}}\right]= \E_m \left[\left.N_{(1-\rho)T}^{\floor{k}}e^{-\bm \theta_T \cdot \bo Z_{(1-\rho)T}}\right| \bo Z_{(1-\rho)T}\neq \bo 0\right]\p_m\paren{\bo Z_{(1-\rho)T}\neq \bo 0}.
\end{align*}Observe that when $\{\bo Z_{(1-\rho)T}\neq \bo 0,N_{(1-\rho)T}<k\}$, the above term is zero. 
Thus, using our previous computation for the survival probability, we obtain \begin{align*}
	 \E_m&\left[N_{(1-\rho)T}^{\floor{k}} e^{-\bm \theta_T \cdot \bo Z_{(1-\rho)T}}\right]\sim \frac{\xi_m}{\frac{\zeta }{2}(1-\rho)T} \E_m \left[\left.N_{(1-\rho)T}^{\floor{k}} e^{-\bm \theta_T \cdot \bo Z_{(1-\rho)T}}\right| \bo Z_{(1-\rho)T}\neq \bo 0\right]\\
%	&\comentario{\sim} \frac{\xi_m}{\zeta (1-\rho)T/2} \E^{k}_m\paren{\left.\paren{\sum_{\ell=1}^{d}\bo Z^\ell_{(1-\rho)T}}^{k}\times e^{-\bm \theta_T \cdot \bo Z_{(1-\rho)T}}\right| \bo Z_{(1-\rho)T}>\bo 0}\\
	&\hspace{.2cm} \sim ((1-\rho)T)^{k-1} \frac{\xi_m}{\zeta /2} \E_m \left[\left.\paren{\frac{N_{(1-\rho)T}}{(1-\rho)T} }^{k} \exp\left\{-(1-\rho)\frac{2}{\zeta}\bm \theta \cdot \frac{\bo Z_{(1-\rho)T}}{(1-\rho)T}\right\}\right| \bo Z_{(1-\rho)T}\neq \bo 0\right]\\
	&\hspace{.2cm} \sim \paren{(1-\rho)T\frac{\zeta}{2}}^{k-1} \xi_m \E \left[\left.\paren{\gamma\,\bm \eta\cdot \vec{1}}^{k}\exp\left\{-(1-\rho)\bm \eta\cdot \bm \theta \gamma\right\}\right.\right]\\
	& \hspace{.2cm}= \paren{(1-\rho)T\frac{\zeta}{2}}^{k-1} \frac{k!\xi_m(\bm \eta\cdot \vec{1})^k}{\paren{1+(1-\rho)\bm \eta\cdot \bm \theta }^{k+1}},
\end{align*}where we have used that $\E\left[\gamma^k\exp\{-\lambda \gamma\}\right]=k! (1+\lambda )^{-k-1}$. This completes the proof.
\end{proof}

With these asymptotic results at hand, we now derive Theorem \ref{teoJointLawSplittingTimeChildrenPartitionScaledBeforeTheLimit}.

\begin{proof}[Proof of Theorem \ref{teoJointLawSplittingTimeChildrenPartitionScaledBeforeTheLimit}]
Observe from Equation \eqref{eqnJointLawSplittingTimeChildrenPartitionScaledBeforeTheLimit}, that the only terms that depend on $T$ are
\begin{equation}\label{eqnJointLawSplittingTimeChildrenPartitionScaledBeforeTheLimitDependingOnT}
\begin{split}
	&\left(\prod_{m=1}^d \E_m\left[e^{-\bm \theta \cdot \bo Z_{T(1-\rho )}}\right]^{\ell_m-g_m}\right)\E_r\left[ Z^{(i)}_{\rho T}\prod_{m=1}^d\E_m\left[e^{-\bm \theta \cdot \bo Z_{T(1-\rho )}}\right]^{Z^{(m)}_{\rho T}-\delta_{i,m}}\right]\\
&\hspace{5cm} \times \frac{\prod_{m=1}^d\prod_{q=1}^{g_m} \E_m\left[N^{\floor{a_{m,q}}}_{\rho T}e^{-\bm \theta \cdot \bo Z_{T(1-\rho )}}\right]}{\E_r\left[N^{\floor{k}}_{T}e^{-\bm \theta\cdot \bo Z_{T}}\right]}T.
\end{split}
\end{equation}
From \eqref{eqnLaplaceTransformOfMtypeBranchingProcess}, \eqref{eqnProductOfLaplaceTransformsShS0OfMtypeBranchingProcess} and  \eqref{eqnDiscountedPowerOfSizeOfGaltonWatson}, as $T$ increases, we deduce
\begin{align*}
	& \E_r\left[ Z^{(i)}_{\rho T}\prod_{m=1}^d\E_m\left[e^{-\bm \theta \cdot \bo Z_{T(1-\rho )}}\right]^{Z^{(m)}_{\rho T}}\right]\frac{\prod_{m=1}^d\prod_{q=1}^{g_m} \E_m\left[N^{\floor{a_{m,q}}}_{\rho T}e^{-\bm \theta \cdot \bo Z_{T(1-\rho )}}\right]}{\E_r\left[N^{\floor{k}}_{T}e^{-\bm \theta\cdot \bo Z_{T}}\right]}T\\
	& \sim \eta_i\xi_r \paren{1+\rho \frac{\bm \eta \cdot\bm \theta}{1+(1-\rho )\bm \eta \cdot\bm \theta}}^{-2}\\
	& \qquad\qquad\times \frac{\prod_{m=1}^d\prod_{q=1}^{g_m}\paren{(1-\rho )T\frac{\zeta}{2}}^{a_{m,q}-1} \xi_m(\bm \eta \cdot\vec{1})^{a_{m,q}}a_{m,q}!\paren{1+(1-\rho )\bm \theta \cdot \bm \eta }^{-a_{m,q}-1}}{\paren{T\frac{\zeta}{2}}^{k-1} \xi_r(\bm \eta \cdot\vec{1})^kk!\paren{1+\bm \theta \cdot \bm \eta }^{-k-1}}T\\
	& =\eta_i \paren{1+\rho \frac{\bm \eta \cdot\bm \theta}{1+(1-\rho )\bm \eta \cdot\bm \theta}}^{-2}(\bm \eta \cdot\vec{1})^{\sum_m \sum_q a_{m,q}}\paren{1+(1-\rho )\bm \theta \cdot \bm \eta }^{-\sum_m \sum_q(a_{m,q}+1)}\\
	& \qquad\qquad\qquad\qquad\qquad\qquad\times \prod_{m=1}^d\prod_{q=1}^{g_m} \xi_ma_{m,q}! \frac{\paren{(1-\rho )T\frac{\zeta}{2}}^{\sum_m \sum_q (a_{m,q}-1)}}{\paren{T\frac{\zeta}{2}}^{k-1} (\bm \eta \cdot\vec{1})^kk!\paren{1+\bm \theta \cdot \bm \eta }^{-k-1}}T.
\end{align*}On the other hand, from Lemma \ref{lemmaLaplaceTransformOfMtypeBranchingProcess}, the remainder terms in \eqref{eqnJointLawSplittingTimeChildrenPartitionScaledBeforeTheLimitDependingOnT} behave, as $T\to \infty$, as follows
\[
\prod_{m=1}^d \E_m\left[e^{-\bm \theta \cdot \bo Z_{T(1-\rho)}}\right]^{\ell_m-g_m}\E_i\left[e^{-\bm \theta \cdot \bo Z_{T(1-\rho)}}\right]^{-1}=1+o(1). 
\]In other words, recalling that $\sum_m \sum_q a_{m,q}=k$ and that $\sum_m g_m=:\# b$ denotes the number of blocks in the partition associated with $(a_{m,q})_{m\in[d], q\in [g_m]}$, we obtain that, as $T$ increases, expression \eqref{eqnJointLawSplittingTimeChildrenPartitionScaledBeforeTheLimitDependingOnT} behaves  as 
\begin{align*}
%\Q^{(k),\bm \theta_T}_{T,r}&\paren{\frac{\tau_1}{T}\in {\rm d}\rho ,\#\mathcal{P}_{\rho}=(a_{m,q})_{m\in[d], q\in [g_m]},{\bo L}_\rho=\bm{\ell} , c(\varsigma^{(1)}_{\rho-})=i}\\
%& \eta_i \paren{\frac{1+\bm \eta \cdot\bm \theta}{1+(1-\rho)\bm \eta \cdot\bm \theta}}^{-2}\frac{(1-\rho)^{k-\#b}}{\paren{1+(1-\rho)\bm \eta \cdot \bm \theta }^{k+\# b}} \bm \xi^{\bo g}\prod_{q=1}^{g_m}a_{m,q}!\\
	%& \qquad\qquad \qquad \qquad \qquad \qquad \qquad \qquad \qquad  \times \frac{1}{k!}\frac{\paren{T\frac{\zeta}{2}}^{-\# b}}{\paren{T\frac{\zeta}{2}}^{-1} \paren{1+\bm \eta \cdot \bm \theta  }^{-k-1}}T{\rm d}\rho \\
&\eta_i \paren{\frac{1+\bm \eta \cdot\bm \theta}{1+(1-\rho)\bm \eta \cdot\bm \theta}}^{-2}\frac{(1-\rho)^{k-\#b}}{\paren{1+(1-\rho)\bm \eta \cdot \bm \theta }^{k+\# b}} \bm \xi^{\bo g}\prod_{q=1}^{g_m}a_{m,q}!\\
& \qquad\qquad \qquad \qquad \qquad \qquad \qquad \qquad \qquad  \times \paren{\frac{\zeta}{2}}^{1-\# b}\frac{1}{k!}\frac{1}{\paren{1+\bm \eta \cdot \bm \theta  }^{-k-1}}T^{2-\# b}{\rm d}\rho \\
%	& =\eta_i \paren{\frac{1+(1-\rho)\bm \eta \cdot\bm \theta}{1+\bm \eta \cdot\bm \theta}}^{2}\\
%	& \qquad \times \frac{(1-\rho)^{k-\#b}\paren{\frac{\zeta}{2}}^{-\# b}\paren{1+(1-\rho)\bm \theta \cdot \bm \eta }^{-k-\# b}\prod_{m=1}^d \xi_m^{g_m}\prod_{q=1}^{g_m}a_{mq}!}{\paren{\frac{\zeta}{2}}^{-1} \paren{1+\bm \theta \cdot \bm \eta }^{-k-1}}T^{2-\# b}\\
	& =\eta_i   (1-\rho)^{k-\#b}\paren{\frac{\zeta}{2}}^{1-\# b} \frac{ \paren{1+\bm \eta \cdot\bm \theta}^{k-1}}{\paren{1+(1-\rho)\bm \eta \cdot\bm \theta}^{k+\#b-2}}T^{2-\# b}\frac{1}{k!}\bm \xi^{\bo g}\left(\prod_{m=1}^d\prod_{q=1}^{g_m} a_{m,q}!\right){\rm d}\rho.
\end{align*}From the latter, we observe that the previous probability depends on $T$ by the factor $T^{2-\# b}$. 
That is to say, that the only way that the previous probability is not negligible is when we have binary splitting.  
Therefore, when $\# b=2$, we only have two cases: either all marks in the splitting event follow two different individuals of the same type or all marks follow two different individuals with different types. 
The first scenario, consists of marks following type $m^*\in [d]$, where $a_{m^*,1}>0$, $a_{m^*,2}>0$, $a_{m^*,1}+a_{m^*,2}=k$, $g_{m^*}=2$ and everything else equals zero. 
On the second case, for some $m^*,n^*\in [d]$ we have $a_{m^*,1}>0$, $a_{n^*,1}>0$, $a_{m^*,1}+a_{n^*,1}=k$, $g_{m^*}=1$ and $g_{n^*}=1$.
Hence, we finally obtain
\begin{align*}
	\lim_{T\to \infty}\Q^{(k),\bm \theta_T}_{T,r}&\paren{\frac{\tau_1}{T}\in {\rm d}\rho ,\#\mathcal{P}_{\rho}=(a_{m,q})_{m\in[d], q\in [g_m]},{\bo L}_\rho=\bm{\ell} , c(\varsigma^{(1)}_{\rho-})=i}\\
%&=\frac{k!}{\prod_{m=1}^d\overline a_{m}!}\prod_{m=1}^d{\ell_m \choose g_m}\frac{\overline a_{m}!}{\prod_{q=1}^{g_m} a_{m,q}!}\frac{g_m!}{\prod_{n\geq 1}d_{m,n}!}\alpha_{i} p_i(\bm \ell)\\
&
=\frac{k!}{\prod_{m=1}^d\prod_{q=1}^{g_m} a_{m,q}!}\alpha_{i} p_i(\bm \ell)\bm \ell^{\floor{\bo g}}\\
	& \hspace{2.5cm}\times \eta_i   (1-\rho)^{k-2}\frac{2}{\zeta} \frac{\paren{1+\bm \eta \cdot\bm \theta}^{k-1}}{ \paren{1+(1-\rho)\bm \eta \cdot\bm \theta}^{k}}\frac{1}{k!}\bm\xi ^{\bo g}\paren{\prod_{m=1}^d\prod_{q=1}^{g_m} a_{m,q}!}{\rm d}\rho\\
	%&=  \alpha_i\eta_i p_i(\bm \ell)\bm \ell^{\floor{\bo g}}\bm \xi^{\bo g}\frac{1}{1+\mathbf{1}_{\{\exists\, m^*: a_{m^*,1}=a_{m^*,2}\}}} (1-\rho)^{k-2}\frac{2}{\zeta} \frac{\paren{1+\bm \eta \cdot \bm \theta}^{k-1}}{ \paren{1+(1-\rho)\bm \eta \cdot \bm \theta}^{k}}{\rm d}\rho.\\
&
=  \alpha_i\eta_i p_i(\bm \ell)\bm \ell^{\floor{\bo g}}\bm \xi^{\bo g} (1-\rho)^{k-2}\frac{2}{\zeta} \frac{\paren{1+\bm \eta \cdot \bm \theta}^{k-1}}{ \paren{1+(1-\rho)\bm \eta \cdot \bm \theta}^{k}}{\rm d}\rho.
\end{align*}
%\comentario{
%Note that we have used that $\prod_{n\geq 1}d_{m,n}!=1+\mathbf{1}_{\{\exists\, m^*: a_{m^*,1}=a_{m^*,2}\}}$.  The latter follows from the following argument: if the marks follow two different colors then $d_{m,n}!=1$ for all $m$ and $n$, on the contrary, they will follow the same color but with the possibility that either the two groups have different sizes or not. That is $a_{m^*,1}\neq a_{m^*,2}$, in which case again $d_{m,n}!=1$ for all $m$ and $n$, or they have the same size, in which case $d_{m^*,a_{m^*,1}}=2$.}
This completes the proof.
\end{proof}

Next, we sum over $\bm \ell\in \mathbb{Z}^d_+$ and $i\in [d]$ to derive the following corollary. To facilitate this, we recall several important identities. Recall that 
\[
f_i(\bo{r})=\sum_{{\bm \ell}\in \mathbb{Z}^d_+}p_i(\bm{\ell})\bm{r^\ell}, \qquad\textrm{and}\qquad m_{i,j}=\frac{\partial}{\partial r_j}f_i(\vec{1}). 
\]
For simplicity on exposition, we write
\[
Q^{(i)}_{j,k}:=\frac{\partial^2 f_i}{\partial r_j \partial r_k}(\vec{1})=\sum_{\bm \ell\in \z_+^d}p_i(\bm \ell)\ell_{j}(\ell_{k}-\mathbf{1}_{\{k=j\}})\quad\textrm{thus} \quad \zeta=\sum_{i,\ell,k=1}^d\alpha_i Q^{(i)}_{\ell,k}\eta_i\xi_{\ell}\xi_k.
\]
\begin{coro}\label{corolarioantesconst} Under assumption \eqref{hyprifa},
	we have
	\[
	\lim_{T\to \infty}\Q^{(k),\bm \theta_T}_{T,r}\paren{\frac{\tau_1}{T}\in {\rm d}\rho}=(k-1)(1-\rho)^{k-2} \frac{\paren{1+\bm \eta \cdot \bm \theta}^{k-1}}{ \paren{1+(1-\rho)\bm \eta \cdot \bm \theta}^{k}}{\rm d}\rho.
	\]
\end{coro}
\begin{proof}
We first sum over $\bm \ell\in \mathbb{Z}^d_+$ in \eqref{limitQ}, and deduce
\begin{align*}
	\lim_{T\to \infty}\Q^{(k),\bm \theta_T}_{T,r}&\paren{\frac{\tau_1}{T}\in {\rm d}\rho ,\#\mathcal{P}_{\rho}=(a_{m^*,1},a_{m^*,2}) , c(\varsigma^{(1)}_{\rho-})=i}\\
	%& = \frac{\alpha_i\eta_i Q^{(i)}_{m^*,m^*}}{1+\mathbf{1}_{\{\exists m^*:\,\,a_{m^*,1}=a_{m^*,2}\}}} \xi_{m^*}^{2}(1-\rho)^{k-2}\frac{2}{\zeta} \frac{\paren{1+\bm \eta \cdot \bm \theta}^{k-1}}{ \paren{1+(1-\rho)\bm \eta \cdot \bm \theta}^{k}}{\rm d}\rho,\\
&
= \alpha_i\eta_i Q^{(i)}_{m^*,m^*}\xi_{m^*}^{2}
	(1-\rho)^{k-2}\frac{2}{\zeta} \frac{\paren{1+\bm \eta \cdot \bm \theta}^{k-1}}{ \paren{1+(1-\rho)\bm \eta \cdot \bm \theta}^{k}}{\rm d}\rho,
\end{align*}or
\begin{align*}
	\lim_{T\to \infty}\Q^{(k),\bm \theta_T}_{T,r}&\paren{\frac{\tau_1}{T}\in {\rm d}\rho ,\#\mathcal{P}_{\rho}=(a_{m^*,1},a_{n^*,1}), c(\varsigma^{(1)}_{\rho-})=i}\\
	& = \alpha_i\eta_i Q^{(i)}_{m^*,n^*}\xi_{m^*}\xi_{n^*}
	(1-\rho)^{k-2}\frac{2}{\zeta} \frac{\paren{1+\bm \eta \cdot \bm \theta}^{k-1}}{ \paren{1+(1-\rho)\bm \eta \cdot \bm \theta}^{k}}{\rm d}\rho,
\end{align*}
accordingly to each situation. 
%\comentario{ERASE: Observe that when $k$ is odd, $\mathbf{1}_{\{\exists m^*:\,\,a_{m^*,1}=a_{m^*,2}\}}=0$,  since all the possible pairs $(a_{m^*,1},a_{m^*,2})$ are $\{(1,k-1),(2,k-2), \ldots, (k-1,1)\}$. If $k$ is even, then when $a_{m^*,1}=a_{m^*,2}$ we are counting twice the value $a_{m^*,1}$ in the sum over all possible pairs $(a_{m^*,1},a_{m^*,2})$ such that $a_{m^*,1}+a_{m^*,2}=k$. In other words, when we sum over all possible $(a_{m^*,1},a_{m^*,2})$ such that $a_{m^*,1}+a_{m^*,2}=k$, and all $i\in [d]$, we get}
Now we sum over all possible block sizes $(a_{m^*,1},a_{m^*,2})$ or $(a_{m^*,1},a_{n^*,2})$. 
In both cases, this is equivalent to summing over all pairs $(j,k-j)$ with $j\in [k-1]$. 
Since there are $k-1$ such pairs, we obtain
\begin{align*}
	\lim_{T\to \infty}\Q^{(k),\bm \theta_T}_{T,r}&\paren{\frac{\tau_1}{T}\in {\rm d}\rho, c(\varsigma^{(1)}_{\rho-})=i}
=(1-\rho)^{k-2}\frac{2}{\zeta}\frac{\paren{1+\bm \eta \cdot \bm \theta}^{k-1}}{ \paren{1+(1-\rho)\bm \eta \cdot \bm \theta}^{k}}\\
&\hspace{2.5cm}\times\Bigg(\sum_{m\in [d]}\sum_{j\in [k-1]}\alpha_i\eta_i Q^{(i)}_{m,m}\xi_{m}^{2}+\sum_{\substack{m, n\in [d]\\m\ne n}}\sum_{j\in [k-1]}\alpha_i\eta_i Q^{(i)}_{m,n}\xi_{m}\xi_{n}
	 \Bigg){\rm d}\rho\\
& =(k-1)(1-\rho)^{k-2}\frac{2}{\zeta}\frac{\paren{1+\bm \eta \cdot \bm \theta}^{k-1}}{ \paren{1+(1-\rho)\bm \eta \cdot \bm \theta}^{k}}\left(\sum_{m,n\in[d]}\alpha_i\eta_i Q^{(i)}_{m,n}\xi_{m}\xi_{n}\right){\rm d}\rho.
\end{align*}Finally, summing over all types $i$ and using the definition of $\zeta$ above, gives us the desired result.
\end{proof}

\label{constructionlimit} Let us now provide an interpretation of all the terms involved in the limit law obtained in Theorem \ref{teoJointLawSplittingTimeChildrenPartitionScaledBeforeTheLimit}. 
First note   that
\begin{align*}
	\zeta&=\sum_{i=1}^d\alpha_i\eta_i\paren{\sum_{\substack{m,n\in[d]\\ m\neq n}} \E_{i}\left[L^{(m)}L^{(n)}\right]\xi_{m}\xi_{n}+\sum_{m\in[d]} \E_i\left[L^{(m)}\paren{L^{(m)}-1}\right]\xi_{m}^2}.\\
	%&=\sum_{i=1}^d \alpha_i\eta_i\paren{\sum_{(g_1,\ldots,g_d):\sum_{m\in [d]} g_m=2} \E_i\left[{\bo L}^{\floor{\bo g}}{\bm \xi}^{\bo g}\right]},
\end{align*}
%where the sum is over all $(g_1,\ldots, g_d)$ such that $g_m+g_n=2$, $g_m,g_n\geq 1$ and $m,n\in [d]$. 

Recalling from \eqref{zetayw} the definitions of  $\zeta_i$, for $i\in [d]$, and $w(\bm \ell)$, for $\bm \ell\in \z^d_+\setminus \{\bo 0\}$, as well as the identity  $\zeta=\sum_{i=1}^d \zeta_i$,  we can now rewrite
\begin{align*}
	\lim_{T\to \infty}\Q^{(k),\bm \theta_T}_{T,r}&\paren{\frac{\tau_1}{T}\in {\rm d}\rho ,\#\mathcal{P}_{\rho}=(a_{m,q})_{m\in[d], q\in [g_m]},{\bo L}_\rho=\bm{\ell} , c(\varsigma^{(1)}_{\rho-})=i}\\
	%& =\frac{\alpha_i\eta_i p_i(\bm \ell) \bm \ell^{\floor{\bo g}}\bm \xi^{\bo g}}{\sum_{j\in [d]}\alpha_j\eta_j\E_j\left[w({\bo L})\right]}
	%\comentario{\frac{1+\mathbf{1}_{\{\forall m\in [d]:\,\, a_{m,1}\neq a_{m,2}\}}}{k-1}}\\
%& =\frac{\alpha_i\eta_i p_i(\bm \ell) \bm \ell^{\floor{\bo g}}\bm \xi^{\bo g}}{\sum_{j\in [d]}\alpha_j\eta_j\E_j\left[w({\bo L})\right]}
	%\frac{1}{k-1}\\
	%&\hspace{4cm}  \times (k-1)(1-\rho)^{k-2} \frac{\paren{1+\bm \eta \cdot \bm \theta}^{k-1}}{ \paren{1+(1-\rho)\bm \eta \cdot \bm \theta}^{k}}{\rm d}\rho \\
	 %& = \frac{\zeta_i}{\zeta}\frac{ p_i(\bm \ell) \bm \ell^{\floor{\bo g}}\bm \xi^{\bo g}}{\zeta_i} 
	 %\\
	% & \qquad \frac{1+\mathbf{1}_{\{\forall m\in [d]:\,\, a_{m,1}\neq a_{m,2}\}}}{k-1} \times
	 %(k-1)(1-\rho)^{k-2} \frac{\paren{1+\bm \eta \cdot \bm \theta}^{k-1}}{ \paren{1+(1-\rho)\bm \eta \cdot \bm \theta}^{k}}{\rm d}\rho\\
	 & = \frac{\zeta_i }{\zeta}\frac{p_i(\bm \ell)w(\bm \ell)}{\E_i\left[w({\bo L})\right]}\frac{\bm \ell^{\floor{\bo g}}\bm \xi^{\bo g}}{w(\bm \ell)}\frac{1}{k-1}\times
	 (k-1)(1-\rho)^{k-2} \frac{\paren{1+\bm \eta \cdot \bm \theta}^{k-1}}{ \paren{1+(1-\rho)\bm \eta \cdot \bm \theta}^{k}}{\rm d}\rho.
\end{align*}
The interpretation of each term is as follows:
\begin{itemize}
	\item[a)] the time  at which the $k$ spines first split apart, under $\lim_{T\to \infty}\Q^{(k),\bm \theta_T}_{T,r}$, has density
	\[
	(k-1)(1-\rho)^{k-2} \frac{\paren{1+\bm \eta \cdot \bm \theta}^{k-1}}{ \paren{1+(1-\rho)\bm \eta \cdot \bm \theta}^{k}}{\rm d}\rho,
	\]
	\item[b)] the $k$ spines split into groups of sizes $(\mathcal{U}, k-\mathcal{U})$ where $\mathcal{U}$ has a uniform distribution on $[k-1]$, so the probability that the spines split into groups of sizes $(h,k-h)$ is  $\frac{1}{k-1}.$
	\item[c)] The probability that the type of the spine is $i$, immediately before splitting,  is $
	\zeta_i/\zeta.$
	\item[d)] Given that the type of the spine is $i$, immediately before splitting, the probability there are $\bm \ell$ children at the  splitting event is
	\[
	\frac{p_i(\bm \ell)w(\bm \ell)}{\E_i\left[w({\bo L})\right]}.
	\]
	\item[e)] Given that the type of the spine is $i$, immediately before splitting, and there are $\bm \ell$ children at the  splitting event, the probability that the types of the offspring carrying the spines  are $(n,m)$, is
\[
\begin{cases}
\frac{\ell_m\ell_n\xi_m\xi_n}{w(\bm \ell)}&\mbox{if $m\neq n$} \\
\frac{\ell_m(\ell_m-1)\xi_m^2}{w(\bm \ell)} &\mbox{if $m= n$}. \end{cases}
\]
In other words from the $\bm \ell$ offspring, we choose two offspring without replacement according to the weights $\bm \xi$. 
\end{itemize}
It is important to note  that (a) and (b) characterise the sample genealogy completely. Actually (a) provides the split times and (b) the tree topology.

Alternatively,  we may have another description by replacing (d) and (e) by the following:
\begin{itemize}
	\item[d')] Given that the type of the spine is $i$, immediately before splitting, the probability that the types of the offspring carrying the spines are $(n,m)$, is
	\[
\begin{cases}
	\frac{\xi_m\xi_n\E_i\left[L^{(m)}L^{(n)}\right]}{\mathbb{E}_i[w(\bo L)]}& \mbox{if $m\neq n$,}\\
	\frac{\xi_m^2\E_i\left[L^{(m)}(L^{(m)}-1)\right]}{\mathbb{E}_i[w(\bo L)]}&\mbox{if $m= n$}.
\end{cases}
\]
	\item[e')] Given that the type of the spine is $i$, immediately before splitting and the types of the offspring carrying the spines are $(n,m)$, the probability there are $\bm \ell$ children at the splitting event is
	\[
\begin{cases}
	\frac{p_i(\bm \ell)\ell_m\ell_n}{\E_i\left[L^{(m)}L^{(n)}\right]}&\mbox{if $m\neq  n$}\\
	\frac{p_i(\bm \ell)\ell_m(\ell_m-1)}{\E_i\left[L^{(m)}(L^{(m)}-1)\right]}& \mbox{if $m= n$}.
\end{cases}
	\]
	\end{itemize}
This approach involves size-biasing the offspring distribution without replacement. It ensures that at least two types are available for selection and increases the probability of selecting offspring of type $m$, proportional to their prevalence in the population.

\subsection{Joint law of spine splitting events under \texorpdfstring{ $\lim_{T\to \infty}\Q^{(k),\bm \theta_T }_{T,r}$}{TEXT}}
Next, we  consider identity \eqref {eqnJointSplittingEventWithPartitionChildAndTypeV2} in Theorem \ref{propofsplitting1}  with $\bm \theta_T$ instead of $\phi\vec{1}$ and for  specific family of sequences of splitting times, and take limit as $T$ goes to infinity.  
More precisely, we consider the event $\Delta_T(n)$ with $t_h=\rho_h T$, for $h\in [n]$, and  $0<\rho_1<\cdots<\rho_{n}<1$. 

Note that $\Delta_{T}(n)$ fully describes the coalescent structure of the sample. In other words, it specifies the times at which each coalescence event occurs, the partition involved in each event, the number of offspring of the individuals participating in a coalescence, and their types.  Since, in the limit we only  observe binary splitting, we only consider the case  $M=k-1$.  In the binary splitting case, we necessarily have that $\sum_{m\in [d]} g_{h,m}=2$, for $h\in[k-1]$.  

\begin{propo}\label{lemmaLimitOfJointLaw}
Under assumption \eqref{hyprifa}. For $r\in[d]$, we have
\begin{equation}\label{eqnLimitOfJointLawLemma}
	\begin{split}
	\lim_{T\to \infty}&\Q^{(k),\bm \theta_T }_{T,r}\Big(\Delta_T(k-1)\Big) =\frac{(1+\bm \eta \cdot \bm \theta)^{k-1}}{k!}\paren{\frac{2}{\zeta}}^{k-1} \prod_{h=1}^{k-1}\frac{p_{i_h}(\bm \ell_{h})\bm \ell_{h}^{\floor{\bo g_h}}\bm \xi^{\bo g_h}\eta_{i_h} }{(1+(1-\rho_h)\bm \eta \cdot \bm \theta)^{2}}\alpha_{i_{h}} {\rm d}\rho_h.
\end{split}
\end{equation}
\end{propo}

\begin{proof}
	Let us study  the first  term  and all  Laplace transforms of $\bo Z_{\cdot}$ with a deterministic power in the second row  in identity \eqref{eqnJointSplittingEventWithPartitionChildAndTypeV2}. In other words, using the asymptotic in \eqref{eqnLaplaceTransformOfMtypeBranchingProcess}, we see 
	\[
	\begin{split}
	&\left(\prod_{h=1}^{k-1}\prod_{m=1}^d\E_{m}\left[e^{-\bm \theta \cdot \bo Z_{T-t_h}}\right]^{\ell_{h,m} -g_{h,m}}p_{i_h}(\bm \ell_{h})\bm \ell_{h}^{\floor{\bo g_h}}\right)\\
		&\hspace{5cm}\times\prod_{h=0}^{k-2}\prod_{m=1}^d\prod_{\substack{ q\leq g_{h,m}:\\ k_{v(h,m,q)}\geq 2}}\prod_{j=1}^d\E_{j}\left[e^{-\bm \theta\cdot \bo Z_{T-t_{v(h,m,q)}}}\right]^{ -\delta_{i_{v(h,m,q)},j}}\\
	&\hspace{5cm}\sim \prod_{h=1}^{k-1}p_{i_h}(\bm \ell_{h})\bm \ell_{h}^{\floor{\bo g_h}}(1+o(1)). 
\end{split}
\]	
For the remainder terms in the  second row of  \eqref{eqnJointSplittingEventWithPartitionChildAndTypeV2}, we observe,  using \eqref{eqnProductOfLaplaceTransformsShS0OfMtypeBranchingProcess},  that 
\begin{equation}\label{eqnProductOfExponentialOfJointLawAsympoticContribution}
\begin{split}
\prod_{h=0}^{k-2}\prod_{m=1}^d\prod_{\substack{ q\leq g_{h,m}:\\ k_{v(h,m,q)}\geq 2}}\E_{m}&\left[Z^{(c(v(h,m,q)))}_{t_{v(h,m,q)}-t_h} \prod_{j=1}^d\E_{j}\left[e^{-\bm \theta\cdot \bo Z_{T-t_{v(h,m,q)}}}\right]^{ Z^{(j)}_{t_{v(h,m,q)}-t_h}}\right]\\
	&\sim \prod_{h=0}^{k-2}\prod_{m=1}^d\prod_{\substack{ q\leq g_{h,m}:\\ k_{v(h,m,q)}\geq 2}} \xi_{m}\eta_{i_{v(h,m,q)}}\paren{\frac{1+(1-\rho_{v(h,m,q)})\bm \eta \cdot \bm \theta}{1+(1-\rho_{h})\bm \eta \cdot \bm \theta}}^2. 
\end{split}
\end{equation}
Next, we have the following identity
\begin{equation}\label{eqnLimitOfJointLaw2}
\prod_{h=0}^{k-2}\prod_{m=1}^d\prod_{\substack{ q\leq g_{h,m}:\\ k_{v(h,m,q)}\geq 2}} \paren{1+(1-\rho_{v(h,m,q)})\bm \eta \cdot \bm \theta}^2=\prod_{h=1}^{k-1} (1+(1-\rho_h)\bm \eta \cdot \bm \theta)^2,
\end{equation}which follows from the fact that in each vertex $v(h,m,q)$ there is a binary splitting event and there are exactly $k-1$ of such events.

On the other hand, it is straightforward to see
\[
\begin{split}
 \prod_{h=0}^{k-2}\prod_{m=1}^d\prod_{\substack{ q\leq g_{h,m}:\\ k_{v(h,m,q)}\geq 2}} &\paren{\frac{1}{1+(1-\rho_{h})\bm \eta \cdot \bm \theta}}^2=\prod_{h=0}^{k-2}\prod_{m=1}^d \paren{\frac{1}{1+(1-\rho_{h})\bm \eta \cdot \bm \theta}}^{2\# \{q\le g_{h,m} :k_{v(h,m,q)}\geq 2 \}}\\
 &\hspace{2cm}=(1+\bm \eta \cdot \bm \theta)^{-2}\prod_{h=1}^{k-1} \paren{\frac{1}{1+(1-\rho_{h})\bm \eta \cdot \bm \theta}}^{2\# \{(m,q):k_{v(h,m,q)}\geq 2 \}}.
 \end{split}
\]
Now, we note that $\# \{(m,q):k_{v(h,m,q)}\geq 2 \}$ are exactly the number of branches after a splitting event that carry at least two marks, i.e. it takes the values in $\{0,1,2\}$ since $\sum_{m=1}^d g_{h, m}=2$ for any $h\in [k-1]$. 
Thus, the overall contribution of the second line in \eqref{eqnProductOfExponentialOfJointLawAsympoticContribution} is
\begin{align*}
&(1+\bm \eta \cdot \bm \theta)^{-2}\xi_{r}\eta_{i_1}\prod_{h=1}^{k-1} \frac{(1+(1-\rho_h)\bm \eta \cdot \bm \theta)^2}{\paren{1+(1-\rho_{h})\bm \eta \cdot \bm \theta}^{2\# \{(m,q):k_{v(h,m,q)}\geq 2 \}}}\prod_{m=1}^d\prod_{\substack{ q\leq g_{h, m}:\\ k_{v(h,m,q)}\geq 2}}\xi_{m}\eta_{i_{ v(h,m,q)}}.
\end{align*}For the last line in  \eqref{eqnJointSplittingEventWithPartitionChildAndTypeV2}, using  \eqref{eqnDiscountedPowerOfSizeOfGaltonWatson}, we deduce that the numerator contributes as follows
\begin{align*}
	& \prod_{h=1}^{k-1}\prod_{m=1}^{d}\E_{m}\left[N_{T-t_{h}}e^{-\bm \theta \cdot \bo Z_{T-t_{h}}}\right]
		^{\#\{ q\leq g_{h,m}:k_{v(h,m,q)}=1\}}\\
	& \hspace{3cm}\sim \prod_{h=1}^{k-1} \prod_{m=1}^d \paren{\frac{\xi_m (\vec{1}\cdot \bm \eta)}{\big(1+(1-\rho_h)\bm \eta \cdot \bm \theta\big)^{2}}}^{\#\{ q\leq g_{h,m}:k_{v(h,m,q)}=1\}}.
\end{align*}
To simplify the right-hand side of the previous asymptotic, we have that the following identity holds 
\[
\sum_{m=1}^d\#\{ q\leq g_{h,m}:k_{v(h,m,q)}=1\}=\#\{ (m,q):k_{v(h,m,q)}=1\}
\] 
which,  for each $h\in [k-1]$, denotes the number of branches produced in  a splitting event with  only one mark. Thus 
\[
\sum_{h=1}^{k-1}\sum_{m=1}^d\#\{ q\leq g_{h,m}:k_{v(h,m,q)}=1\}=k
\]
  and 
\begin{align*}
	  \prod_{h=1}^{k-1}\prod_{m=1}^{d}\E_{m}\left[N_{T-t_{h}}e^{-\bm \theta \cdot \bo Z_{T-t_{h}}}\right]
		^{\#\{ q\leq g_{h,m}:k_{v(h,m,q)}=1\}}&\sim (\bm \eta\cdot \vec{1})^k\prod_{h=1}^{k-1} \prod_{m=1}^d \xi_m^{\#\{ q\leq g_{h, m}:k_{v(h,m,q)}=1\}} \\
	&\times\prod_{h=1}^{k-1}\frac{1}{\big(1+(1-\rho_h)\bm \eta \cdot \bm \theta\big)^{2\#\{ (m,q):k_{v(h,m,q)}=1\}}}.
\end{align*}The remaining terms on the last line in \eqref{eqnJointSplittingEventWithPartitionChildAndTypeV2}, after using again \eqref{eqnDiscountedPowerOfSizeOfGaltonWatson}, contribute
\begin{align*}
\frac{1}{\E_{r}\left[N_{T}^{\floor{k}}e^{-\bm \theta\cdot \bo Z_{T}}\right]} \prod_{h=1}^{k-1}\alpha_{i_{h}} T{\rm d}\rho_h& \sim\frac{(1+\bm \eta \cdot \bm \theta)^{k+1}}{k!\big(T\frac{\zeta}{2}\big)^{k-1}\xi_{r}(\vec{1}\cdot \bm \eta)^k}T^{k-1}\prod_{h=1}^{k-1}\alpha_{i_{h}}{\rm d}\rho_h\\
	&=\frac{(1+\bm \eta \cdot \bm \theta)^{k+1}}{k!\xi_{r}(\vec{1}\cdot \bm \eta)^k}\paren{\frac{2}{\zeta}}^{k-1}\prod_{h=1}^{k-1}\alpha_{i_{h}} {\rm d}\rho_h.
\end{align*}Putting all pieces together and  noting that 
\[
\# \{(m,q):k_{v(h,m,q)}\geq 1 \}=\sum_{m=1}^dg_{h,m}=2,
\]
 we obtain
\begin{align*}
	\lim_{T\to \infty}&\Q^{(k),\bm \theta }_{T,r}\Big(\Delta_T(k-1)\Big)=\frac{1}{(1+\bm \eta \cdot \bm \theta)^{2}}\prod_{h=1}^{k-1}p_{i_h}(\bm \ell_{h})\bm \ell_{h}^{\floor{\bo g_h}} \frac{(1+(1-\rho_h)\bm \eta \cdot \bm \theta)^2}{\paren{1+(1-\rho_{h})\bm \eta \cdot \bm \theta}^{2\# \{(m,q):k_{v(h,m,q)}\geq 2 \}}}\\
	&\hspace{8cm} \times \xi_{r}\eta_{i_1}\prod_{h=1}^{k-1}\prod_{m=1}^d\prod_{\substack{ q\leq g_{h, m}:\\ k_{v(h,m,q)}\geq 2}}\xi_{m}\eta_{i_{ v(h,m,q)}}\\
	& \qquad \times (\bm \eta\cdot \vec{1})^k\left(\prod_{h=1}^{k-1} \prod_{m=1}^d \xi_m^{\#\{ q\leq g_{h, m}:k_{v(h,m,q)}=1\}}\right) \prod_{h=1}^{k-1}\frac{1}{\big(1+(1-\rho_h)\bm \eta \cdot \bm \theta\big)^{2\#\{ (m,q):k_{v(h,m,q)}=1\}}}\\
	& \hspace{7.5cm} \times \frac{(1+\bm \eta \cdot \bm \theta)^{k+1}}{k!\xi_{r}(\vec{1}\cdot \bm \eta)^k}\paren{\frac{2}{\zeta}}^{k-1}\prod_{h=1}^{k-1}\alpha_{i_{h}} {\rm d}\rho_h\\
	& = \frac{(1+\bm \eta \cdot \bm \theta)^{k-1}}{k!}\prod_{h=1}^{k-1}\frac{2}{\zeta}\frac{p_{i_h}(\bm \ell_{h})\bm \ell_{h}^{\floor{\bo g_h}} \xi_m^{ \bo g_{h}}\eta_{i_h}}{(1+(1-\rho_h)\bm \eta \cdot \bm \theta)^{2}}\alpha_{i_{h}} {\rm d}\rho_h,
\end{align*}
where in the last identity we have used a similar argument as in \eqref{eqnLimitOfJointLaw2}, to deduce 
\[
\eta_{i_1}\prod_{h=1}^{k-1}\prod_{m=1}^d\prod_{\substack{ q\leq g_{h, m}:\\ k_{v(h,m,q)}\geq 2}}\eta_{i_{ v(h,m,q)}}=\prod_{h=1}^{k-1} \eta_{i_h}
\]
and the following identity
\[
\left(\prod_{h=1}^{k-1}\prod_{m=1}^d\prod_{\substack{ q\leq g_{h, m}:\\ k_{v(h,m,q)}\geq 2}}\xi_{m}\right) \left(\prod_{h=1}^{k-1} \prod_{m=1}^d \xi_m^{\#\{ q\leq g_{h, m}:k_{v(h,m,q)}=1\}}\right)= \prod_{m=1}^d \xi_m^{\sum_{h=1}^{k-1}\#\{ q\leq g_{h, m}:k_{v(h,m,q)}\ge 1\}},
\]
which is straightforward to obtain. This concludes the proof.
\end{proof}
Similarly as in the case of one splitting event (see Theorem \ref{teoJointLawSplittingTimeChildrenPartitionScaledBeforeTheLimit} and the comments after the proof),  we may rewrite the terms appearing in Proposition \ref{lemmaLimitOfJointLaw}. 
Recall the construction given in page \pageref{constructionlimit} (just after the proof of Corollary \ref{corolarioantesconst}) and the definitions of $w(\bm \ell)$ and $\zeta_i$ that appear in such construction. Thus \eqref{eqnLimitOfJointLawLemma} can be rewritten as 
\begin{equation}\label{eqnLimitOfJointLaw}
	\begin{split}&\lim_{T\to \infty}\Q^{(k),\bm \theta_T }_{T,r}\Big(\Delta_T(k-1)\Big)\\
&\hspace{1cm} =\frac{(1+\bm \eta \cdot \bm \theta)^{k-1}2^{k-1}}{k!}\prod_{h=1}^{k-1}\frac{\zeta_{i_h}}{\zeta}\frac{p_{i_h}(\bm \ell_{h})w(\bm \ell_{h})}{\mathbb{E}_{i_h}\left[w(\bo L)\right]}\frac{\bm \ell_{h}^{\floor{\bo g_h}}\bm \xi^{\bo g_h}}{w(\bm \ell_{h})}\frac{1 }{(1+(1-\rho_h)\bm \eta \cdot \bm \theta)^{2}}{\rm d}\rho_h.\\
			\end{split}
\end{equation}
\subsection{Colours of the spine before the first spine splitting event under \texorpdfstring{ $\lim_{T\to \infty}\Q^{(k),\bm \theta_T}_{T,r}$}{TEXT}} We now turn out attention to the distribution of the colour process of the vertex carrying  the $k$ marks up to the first splitting time,  under both $\Q^{(k),\bm \theta}_{T,r}$ and its limiting regime $\lim_{T\to \infty}\Q^{(k),\bm \theta_T}_{T,r}$. 

The next result shows, in particular, that under $\Q^{(k),\bm \theta}_{T,r}$,  the colour of the vertex carrying the  $k$ marks evolves as an inhomogeneous Markov chain.

\begin{lemma}\label{lemmaColorOfTheSpineIsMC}
Fix $n\in \na$ and $r\in[d]$.  Consider $(i_h)_{h\in [n]}\in [d]^{n+1}$ and $(t_h)_{h\in [n]}\in [d]^{n}$ with $0<t_1<\cdots <t_n<T$. 
Then
\begin{align*}
 \Q^{(k),\bm \theta}_{T,r}&\left(c(\varsigma_{t_1}^{(1)})=i_1,\ldots, c(\varsigma_{t_n}^{(1)})=i_n,\tau_1>t_n\right)\\
&\hspace{2cm} = \frac{\E_{i_n}\left[N^{\floor{k}}_{T-t_n}e^{-\bm \theta\cdot \bo Z_{T-t_n}}\right]}{\E_{r}\left[N^{\floor{k}}_{T}e^{-\bm \theta\cdot \bo Z_{T}}\right]}\prod_{h=1}^{n}p_{T-t_{h-1}}( C_{t_h-t_{h-1}}=i_h| C_0=i_{h-1}),
\end{align*}
where $t_0=0$, $i_0=r$ and $(C_s)_{s\in[0,T]}$ is an inhomogeneous continuous-time Markov chain with values in $[d]$ and with transitions
\[
p_{T-s}( C_{t-s}=j| C_0=i):=\frac{\E_i\left[Z^{(j)}_{t-s}\prod_{m\in [d]}\E_m\left[e^{-\bm \theta\cdot \bo Z_{T-t}}\right]^{Z^{(m)}_{t-s}}\right]}{\E_i\left[e^{-\bm \theta\cdot \bo Z_{T-t}}\right]}\qquad i,j\in [d],\ 0\le s\le t<T.
\]
\end{lemma}
\begin{proof}
For the proof, we restrict ourselves to the case $n=2$, since the general case can then be obtained directly from this argument by applying the Markov branching property stated in \cite{AHP-p1}, Lemma 3.
Note that on the event of interest, the (unique) vertex carrying all $k$ marks at time $t_1$ has type $i_1$. 
Thus, by the Markov branching property, we have
\begin{align*}
& \Q^{(k),\bm \theta}_{T,r}\left(c(\varsigma_{t_1}^{(1)})=i_1,c(\varsigma_{t_2}^{(1)})=i_2,\tau_1>t_2\right)\\
&\hspace{3cm} =\Q^{(k),\bm \theta}_{T,r}\left(c(\varsigma_{t_1}^{(1)})=i_1,\tau_1>t_1\right)\Q^{(k),\bm \theta}_{T-t_1,i_1}\left(c(\varsigma_{t_2-t_1}^{(1)})=i_2,\tau_1>t_2-t_1\right).
\end{align*}Therefore from \cite{AHP-p1}, Lemma 4,  we obtain
\begin{align*}
& \Q^{(k),\bm \theta}_{T,r}\left(c(\varsigma_{t_1}^{(1)})=i_1,c(\varsigma_{t_2}^{(1)})=i_2,\tau_1>t_2\right)\\
&\hspace{1cm} =\frac{\E_{i_1}\left[N^{\floor{k}}_{T-t_1}e^{-\bm \theta\cdot \bo Z_{T-t_1}}\right]}{\E_{r}\left[N^{\floor{k}}_{T}e^{-\bm \theta\cdot \bo Z_{T}}\right]}\frac{\E_{r}\left[Z^{(i_1)}_{t_1}\prod_{m\in [d]}\E_m\left[e^{-\bm \theta\cdot \bo Z_{T-t_1}}\right]^{Z^{(m)}_{t_1}}\right]}{\E_{i_1}\left[e^{-\bm \theta\cdot \bo Z_{T-t_1}}\right]}\\
& \hspace{3cm}\times\frac{\E_{i_2}\left[N^{\floor{k}}_{T-t_2}e^{-\bm \theta\cdot \bo Z_{T-t_2}}\right]}{\E_{i_1}\left[N^{\floor{k}}_{T-t_1}e^{-\bm \theta\cdot \bo Z_{T-t_1}}\right]}\frac{\E_{i_1}\left[Z^{(i_2)}_{t_2-t_1}\prod_{m\in [d]}\E_m\left[e^{-\bm \theta\cdot \bo Z_{T-t_2}}\right]^{Z^{(m)}_{t_2-t_1}}\right]}{\E_{i_2}\left[e^{-\bm \theta\cdot \bo Z_{T-t_2}}\right]}\\
&\hspace{0cm} =\frac{\E_{i_2}\left[N^{\floor{k}}_{T-t_2}e^{-\bm \theta\cdot \bo Z_{T-t_2}}\right]}{\E_{r}\left[N^{\floor{k}}_{T}e^{-\bm \theta\cdot \bo Z_{T}}\right]}p_{T}( C_{t_1}=i_1| C_0=r)p_{T-t_1}( C_{t_2-t_1}=i_2| C_0=i_1).
\end{align*}
This completes the proof.
\end{proof}

Next, we turn our attention to the limiting finite-dimensional distributions of the colour process of the vertex carrying all marks prior to the first spine splitting event.

\begin{lemma}\label{lemmaColorOfTheSpineIsMCQLimit}
Fix $n\in \na$, $0<\rho_1<\cdots <\rho_n<1$,$ r\in [d]$ and $(i_h)_{h\in [n]}\in [d]^{n}$. 
For any $m\in [d]$, we  let $D_m:=\#\{h\in [n]:i_h=j \}$, and define $\bo D=(D_1,\ldots, D_d)$. 
Then, we have
\[
\begin{split}
\lim_{T\to \infty} \Q^{(k),\bm \theta_T}_{T,r}&\left(c(\varsigma_{\rho_1T}^{(1)})=i_1,\ldots, c(\varsigma_{\rho_nT}^{(1)})=i_n,\tau_1>\rho_nT\right)\\
&\hspace{3cm} =\paren{1-\frac{\rho_n}{1+(1-\rho_n)\bm \eta\cdot \bm \theta }}^{k-1}\bm \eta^{\bo D}\bm \xi^{\bo D}.
\end{split}
\]
\end{lemma}
\begin{proof}
First, we use the asymptotic results obtained in  \eqref{eqnDiscountedPowerOfSizeOfGaltonWatson}, to deduce
\begin{align*}
\lim_{T\to \infty}\frac{\E_{i_n}\left[N^{\floor{k}}_{T(1-\rho_n)}e^{-\bm \theta\cdot \bo Z_{T(1-\rho_n)}}\right]}{\E_{r}\left[N^{\floor{k}}_{T}e^{-\bm \theta\cdot \bo Z_{T}}\right]}&=\lim_{T\to \infty} \frac{\paren{(1-\rho_n)T\frac{\zeta}{2}}^{k-1} \frac{k!\xi_{i_n}(\vec{1}\cdot\bm \eta)^k}{\paren{1+(1-\rho_n)\bm \eta\cdot \bm \theta }^{k+1}}}{\paren{T\frac{\zeta}{2}}^{k-1} \frac{k!\xi_{r}(\vec{1}\cdot\bm \eta)^k}{\paren{1+\bm \eta\cdot \bm \theta }^{k+1}}}\\
& =(1-\rho_n)^{k-1}\frac{\xi_{i_n}}{\xi_{r}}\paren{\frac{1+\bm \eta\cdot \bm \theta }{1+(1-\rho_n)\bm \eta\cdot \bm \theta }}^{k+1}. 
\end{align*}
Recalling the definition of the transition probability $p_T$ introduced in Lemma \ref{lemmaColorOfTheSpineIsMC}, we observe from \eqref{eqnLaplaceTransformOfMtypeBranchingProcess} that the scaled denominator of $p_\cdot$ converges to one.
Combining this fact with \eqref{eqnProductOfLaplaceTransformsShS0OfMtypeBranchingProcess}, we obtain
\begin{align*}
& \lim_{T\to \infty}p_{T(1-\rho_{h-1})}( C_{(\rho_h-\rho_{h-1})T}=i_h| C_0=i_{h-1})= \eta_{i_h}\xi_{i_{h-1}} \paren{\frac{1+(1-\rho_h)\bm \eta\cdot \bm \theta}{1+(1-\rho_{h-1})\bm \eta\cdot \bm \theta}}^{2}.
\end{align*}Putting all pieces together and using Lemma \ref{lemmaColorOfTheSpineIsMC}, we deduce
\[
\begin{split}
\lim_{T\to \infty} \Q^{(k),\bm \theta_T}_{T,r}&\left(c(\varsigma_{\rho_1T}^{(1)})=i_1,\ldots, c(\varsigma_{\rho_nT}^{(1)})=i_n,\tau_1>\rho_nT\right)\\
&\hspace{1cm} =\frac{\xi_{i_n}}{\xi_{r}}\paren{(1-\rho_n)\frac{1+\bm \eta\cdot \bm \theta }{1+(1-\rho_n)\bm \eta\cdot \bm \theta }}^{k-1}\xi_{r}\prod_{h=1}^{n}\eta_{i_h}\prod_{h=1}^{n-1}\xi_{i_h}\\
&\hspace{1cm} =\paren{1-\frac{\rho_n}{1+(1-\rho_n)\bm \eta\cdot \bm \theta }}^{k-1}\prod_{h=1}^{n}\eta_{i_h}\xi_{i_h}.
\end{split}
\]
The result follows from the definition of $\bo D$. 
\end{proof}

\subsection{Joint law of spine splitting events under \texorpdfstring{ $\lim_{T\to \infty}\Q^{(k),\bm \theta_T}_{\bo c, T,r}$}{TEXT}}
The  aim  of this subsection is to establish the analogue of \eqref{eqnLimitOfJointLaw}, under  uniform sampling given a fixed  type configuration scheme, namely $\lim_{T\to \infty}\Q^{(k),\bm \theta_T}_{\bo c, T,r}\big(\Delta_T(k-1)\big)$. To achieve this, we employ the change of measure with respect to $\mathbb{Q}^{(k),\bm{\theta}}_{T,r}$ as introduced in \eqref{eqnComparingQSAndQk}.

Since the first term on the right-hand side of Lemma \ref{coroJointLawAndGivenSampleAreIndependent0} is provided in Theorem \ref{propofsplitting1}, it remains to analyse the limit of each factor in the product of the right-hand side of the identity in \eqref{eqnJointSplittingEventWithPartitionChildAndTypeV3}. We emphasise that the limit in the following result does not depend on $\rho$ or $m$. 
%This independence explains the deliberately vague formulation of $\widetilde t_h$ and $\widetilde m_h$ in \eqref{eqnJointSplittingEventWithPartitionChildAndTypeV3}.

\begin{lemma}\label{lemmaLimitUnderQColorOfSingleSpineIsGiven} Under assumption \eqref{hyprifa},
for any $\rho\in [0,1)$,  $m,c_h\in [d]$ and $h\in [k]$ we have
\begin{equation}\label{eqnLimitUnderQColorOfSingleSpineIsGivenAsym}
\lim_{T\to\infty}\Q^{(1),\bm \theta_T}_{T(1-\rho),m}\paren{c(\varsigma^{(h)}_{T(1-\rho)})=c_h}=\frac{\eta_{c_h}}{\bm \eta\cdot \vec{1}},
\end{equation}
and
\begin{equation}\label{eqnLimitUnderQColorOfSingleSpineIsGiven1}
\lim_{T\to\infty}\Q^{(k),\bm \theta_T}_{T,m}\paren{c(\varsigma^{(h)}_{T})=c_h,\forall\ h\in [k]} =\frac{\bm \eta^{\bo D_{\bo c}}}{(\bm \eta\cdot \vec{1})^{k}}.
\end{equation}
\end{lemma}

Note that this lemma implies that the colours assigned to  each mark  at time $T$ are asymptotically independent, as $T$ increases. The proof requires the following auxiliary result, closely related to \eqref{eqnDiscountedPowerOfSizeOfGaltonWatson}. 
We state this auxiliary result  in a general form, since it will be used later. 

\begin{lemma}\label{eqnAsymptoticZDsExp} Assume that \eqref{hyprifa} is satisfied.
Fix $\rho\in [0,1)$ and a sequence of colours $\bo c^\prime=(c^\prime_1, \cdots, c^\prime_k)$. Then, for any $k\in \na$, as $T$ increases, we have 
\begin{equation}\label{eqnDiscountedPowerOfSizeOfGaltonWatson2}\begin{split}
&\E_m\left[\bo Z^{\floor{\bo D_{\bo c^\prime}}}_{(1-\rho)T}e^{-\bm \theta_T\cdot \bo Z_{(1-\rho)T}}\right] \sim \paren{(1-\rho)T\frac{\zeta}{2}}^{k-1} \frac{k!\xi_m\bm \eta^{\bo D_{\bo c^\prime}}}{\paren{1+(1-\rho)\bm \eta\cdot \bm \theta }^{k+1}}.
\end{split}\end{equation}
In particular when $k=1$ and ${\bo c}^\prime=c^\prime\in[d]$, we have
\begin{equation}\label{eqnDiscountedPowerOfSizeOfGaltonWatson2-1}
\begin{split}
&\E_m\left[ Z^{(c^\prime)}_{(1-\rho)T}e^{-\bm \theta_T\cdot \bo Z_{(1-\rho)T}}\right] \sim  \frac{\xi_m \eta_{c^\prime}}{\paren{1+(1-\rho)\bm \eta\cdot \bm \theta }^{2}}.
\end{split}
\end{equation}
\end{lemma}
\begin{proof}
First, we decompose 
\begin{align*}
& \E_m\left[\bo Z^{\floor{\bo D_{\bo c}}}_{(1-\rho)T}e^{-\bm \theta_T\cdot \bo Z_{(1-\rho)T}}\right]= \E_m \left[\left.\bo Z^{\floor{\bo D_{\bo c}}}_{(1-\rho)T}e^{-\bm \theta_T \cdot \bo Z_{(1-\rho)T}}\right| \bo Z_{(1-\rho)T}\neq \bo 0\right]\p_m\paren{\bo Z_{(1-\rho)T}\neq \bo 0}.
\end{align*}Observe that when $\{\bo Z_{(1-\rho)T}\neq \bo 0\}\cap\{\bo Z_{(1-\rho)T}\geq \bo D_{\bo c}\}^c$, the first term on the right-hand side is zero, implying
\[
\begin{split}
 \E_m\left[\bo Z^{\floor{\bo D_{\bo c}}}_{(1-\rho)T}e^{-\bm \theta_T\cdot \bo Z_{(1-\rho)T}}\right]&= \E_m \left[\left.\bo Z^{\floor{\bo D_{\bo c}}}_{(1-\rho)T}e^{-\bm \theta_T \cdot \bo Z_{(1-\rho)T}}\mathbf{1}_{\{\bo Z_{(1-\rho)T}\geq \bo D_{\bo c}\}}\right| \bo Z_{(1-\rho)T}\neq \bo 0\right]\\
 &\hspace{6cm}\times\p_m\paren{\bo Z_{(1-\rho)T}\neq \bo 0}.
\end{split}
\]
Thus  from Yaglom's limit in Proposition \ref{propSophie} and the asymptotic in \eqref{eqnAsymptoticOfProbabilityOfNonExtinction}, we deduce
\begin{align*}
	 \E_m&\left[\bo Z^{\floor{\bo D_{\bo c}}}_{(1-\rho)T} e^{-\bm \theta_T \cdot \bo Z_{(1-\rho)T}}\right]\sim \frac{\xi_m \paren{(1-\rho)T} ^{\sum_{m=1}^dD_{\bo c,m}}}{\frac{\zeta }{2}(1-\rho)T}\\
	 &\hspace{1.3cm}\times \E_m \left[\left.\prod_{m=1}^d\paren{\frac{Z^{(m)}_{(1-\rho)T}}{(1-\rho)T} }^{D_{\bo c,m}} e^{-(1-\rho)\frac{2}{\zeta}\bm \theta \cdot \frac{\bo Z_{(1-\rho)T}}{(1-\rho)T}}
\mathbf{1}_{\left\{\frac{\bo Z_{(1-\rho)T}}{(1-\rho)T}\geq \frac{\bo D_{\bo c}}{(1-\rho)T}\right\}}	 
	 \right| \bo Z_{(1-\rho)T}\neq \bo 0\right]\\
%	&\comentario{\sim} \frac{\xi_m}{\zeta (1-\rho)T/2} \E^{k}_m\paren{\left.\paren{\sum_{\ell=1}^{d}\bo Z^\ell_{(1-\rho)T}}^{k}\times e^{-\bm \theta_T \cdot \bo Z_{(1-\rho)T}}\right| \bo Z_{(1-\rho)T}>\bo 0}\\
& \sim  \frac{2\xi_m \paren{(1-\rho)T} ^{k-1}}{\zeta}\E \left[\left.\prod_{m=1}^d\paren{\frac{\zeta}{2}\gamma\, \eta_m}^{D_{\bo c,m}} e^{-(1-\rho)\bm \eta\cdot \bm \theta \gamma}\mathbf{1}_{\{\frac{\zeta}{2}\gamma\, \bm \eta\ge 0\}}\right.\right]\\
&= \frac{2\xi_m \paren{(1-\rho)T} ^{k-1}\prod_{m=1}^d\eta_m^{D_{\bo c,m}}}{\zeta} \E \left[\gamma^{k} e^{-(1-\rho)\bm \eta\cdot \bm \theta \gamma}\right]\\
	&  = \paren{(1-\rho)T\frac{\zeta}{2}}^{k-1} \frac{k!\xi_m\bm \eta^{\bo D_{\bo c}}}{\paren{1+(1-\rho)\bm \eta\cdot \bm \theta }^{k+1}},
\end{align*}
where we  used the facts that  $\sum_{m=1}^dD_{\bo c,m}=k$ and $\E\left[\gamma^k e^{-\lambda \gamma}\right]=k! (1+\lambda )^{-k-1}$. This completes the proof.
\end{proof}

\begin{proof}[Proof of Lemma \ref{lemmaLimitUnderQColorOfSingleSpineIsGiven}] From \eqref{eqnLimitUnderQColorOfSingleSpineIsGiven}, we know  
\[
\Q^{(1),\bm \theta_T}_{T(1-\rho),m}\paren{c(\varsigma^{(h)}_{T(1-\rho)})=c_h} =\frac{\E_m[Z^{(c_h)}_{T(1-\rho)}e^{-\bm \theta_T \cdot \bo Z_{T(1-\rho)}}]}{\E_m\left[N_{T(1-\rho)}e^{-\bm \theta_T\cdot \bo Z_{T(1-\rho)}}\right]}.
\]
Thus, using \eqref{eqnDiscountedPowerOfSizeOfGaltonWatson} and  \eqref{eqnDiscountedPowerOfSizeOfGaltonWatson2-1}, the latter implies
\begin{align*}
\Q^{(1),\bm \theta_T}_{T(1-\rho),m}\paren{c(\varsigma^{(h)}_{T(1-\rho)})=c_h} & \sim \frac{ \frac{\xi_m \eta_{c_h}}{\paren{1+(1-\rho)\bm \eta\cdot \bm \theta }^{2}}}{ \frac{\xi_m(\bm \eta\cdot\vec{1})}{\paren{1+(1-\rho)\bm \eta\cdot \bm \theta }^{2}}}=\frac{\eta_{c_h}}{\bm \eta\cdot\vec{1}}.
\end{align*}
For the second limit,  we recall from  identity \eqref{eqnonumber} that \[
\Q^{(k),\bm \theta_T}_{T,r}\paren{c(\varsigma^{(h)}_{T})=c_h,\forall\ h\in [k]} =\frac{\E_{r}[\bo Z^{\floor{\bo D_{\bo c}}}_{T}e^{-\bm \theta_T \cdot \bo Z_{T}}]}{\E_{r}\left[N^{\floor{k}}_{T}e^{-\bm \theta_T\cdot \bo Z_{T}}\right]}.
\]
Thus the result follows from the asymptotics in \eqref{eqnDiscountedPowerOfSizeOfGaltonWatson} and \eqref{eqnDiscountedPowerOfSizeOfGaltonWatson2}.
\end{proof}

We are now ready to establish the limit of the joint law stated in Proposition \ref{lemmaLimitOfJointLaw}, together with the distribution of the marks having specific colours when sampled.
This result shows, that asymptotically,  the colours of the sampled spines are independent and identically distributed, with common law $(\eta_m/(\bm \eta\cdot \vec{1});m\in [d])$.

\begin{coro}\label{coroJointLawTogetherWithSamplingGivenColors}
Assume that  \eqref{hyprifa} holds. Fix $k\geq 2$, and $\bo c\in [d]^k$, 
then
\begin{equation}\label{eqnJointLawTogetherWithSamplingGivenColors}
	\begin{split}
	&\lim_{T\to \infty}\Q^{(k),\bm \theta_T }_{T,r}\Big(\Delta_T(k-1),c(\varsigma^{(h)}_{T})=c_h,\forall\ h\in [k]\Big) =\lim_{T\to \infty}\Q^{(k),\bm \theta_T }_{T,r}\Big(\Delta_T(k-1)\Big)\frac{\bm \eta^{\bo D_{\bo c}}}{(\bm \eta\cdot \vec{1})^k},
\end{split}
\end{equation}where the limit in the right-hand side is  given in Proposition \ref{lemmaLimitOfJointLaw}.
\end{coro}
\begin{proof}
The result follows directly from Lemma \ref{coroJointLawAndGivenSampleAreIndependent0}  and Lemma \ref{lemmaLimitUnderQColorOfSingleSpineIsGiven}. 
\end{proof}

Recall the notation in Lemma \ref{coroJointLawAndGivenSampleAreIndependent0}.
In the following result, we show that under the measure $\Q^{(\bo c),\bm \theta_T }_{T,r}$, the limiting joint probability of $\Delta_T(n)$ coincides with that  under $\Q^{(k),\bm \theta_T }_{T,r}$.

\begin{propo}\label{teoJointLawUnderQSSameLawAsUnderQk} Under assumption \eqref{hyprifa}, it holds 
\[
\lim_{T\to \infty}\Q^{(k),\bm \theta_T}_{\bo c, T,r}\big(\Delta_T(k-1)\big)=\lim_{T\to \infty}\Q^{(k),\bm \theta_T }_{T,r}\big(\Delta_T(k-1)\big),
\]where the right-hand side is  given in Proposition \ref{lemmaLimitOfJointLaw}.
\end{propo}
\begin{proof}
The result follows applying directly Lemma \ref{lemmaLimitUnderQColorOfSingleSpineIsGiven} and Corollary \ref{coroJointLawTogetherWithSamplingGivenColors} to  Proposition \ref{teoJointLawUnderQkbutnewver} with $n=k-1$.\end{proof}

\subsection{Joint law of spine splitting events under \texorpdfstring{ $\lim_{T\to \infty}\Q^{(k),\bm \theta_T}_{\bo w, T,r}$}{TEXT}}
Our aim here is to obtain the analogue of Proposition \ref{lemmaLimitOfJointLaw} (or \eqref{eqnLimitOfJointLaw}), 
for the regime of sampling according to type dependent weights, that is, to evaluate
$\lim_{T\to \infty}\Q^{(k),\bm \theta_T}_{\bo w, T,r}\big(\Delta_T(k-1)\big).$

The following result generalises our sampling procedure to general weights.
\begin{propo}\label{propLimitQWEqualsLimitQk}
Under assumption \eqref{hyprifa}, it holds 
\begin{equation}
\lim_{T\to \infty}\Q^{(k),\bm \theta_T}_{\bo w, T,r}\Big(\Delta_T(k-1)\Big)=\lim_{T\to \infty}\Q^{(k),\bm \theta_T }_{T,r}\Big(\Delta_T(k-1)\Big),
\end{equation}
where the right-hand side is  given in Proposition \ref{lemmaLimitOfJointLaw}.
\end{propo}
\begin{proof} From identity \eqref{eqnQWIsAMixtureOfQS}, we have
\[
\begin{split}
 \lim_{T\to \infty}\Q^{(k),\bm \theta_T}_{\bo w, T,r}\Big(\Delta_T(k-1)\Big)
& =\lim_{T\to \infty}\sum_{ \bo c\in \bo [d]^k} q^{(\bo w),\bo c,\bm \theta_T}_{T,r}\Q^{(k),\bm \theta_T}_{\bo c, T,r}(\Delta_T(k-1))\\
& =\sum_{ \bo c\in \bo [d]^k} \lim_{T\to \infty}q^{(\bo w),\bo c,\bm \theta_T}_{T,r}\lim_{T\to \infty}\Q^{(k),\bm \theta_T}_{\bo c, T,r}(\Delta_T(k-1)).
 \end{split}
\]  
Let us first study the limiting behaviour of  $q^{(\bo w),\bo c,\bm \theta_T}_{T,r}$. Recall its definition  from \eqref{defqbowBoc} and observe that its limit,  as $T\to\infty$, exists. Indeed, from  Lemma \ref{lemmaLimitUnderQColorOfSingleSpineIsGiven}, we obtain
\begin{align*}
 \lim_{T\to \infty} \Q^{(k),\bm \theta }_{T,r}\left[\bo w^{\bo D_{\bm \varsigma_T}}\right]
%q^{(\bo w),\bo c,\bm \theta_T}_{T,r}=
& = \lim_{T\to \infty} \sum_{ \bo c'\in \bo [d]^k} \bo w^{\bo{D}_{\bo c'}}
\Q^{(k),\bm \theta }_{T,r}\paren{c(\bm \varsigma_T)=\bo c'}\\
& = \sum_{ \bo c'\in \bo [d]^k}  \bo{w}^{\bo{D}_{\bo c'}}\frac{\bm \eta^{\bo D_{\bo c'}}}{(\bm \eta\cdot \vec{1})^{k}}.
\end{align*}
%\comentario{Here I erased the extra combinatorial term}
In other words, from  \eqref{defqbowBoc} and the previous limit, we have
\[
\lim_{T\to \infty}q^{(\bo w),\bo c,\bm \theta_T}_{T,r}=\frac{ \bo w^{\bo{D}_{\bo c}}\bm \eta^{\bo D_{\bo c}}}{\sum_{ \bo c'\in \bo [d]^k} \bo w^{\bo{D}_{\bo c'}}\bm \eta^{\bo D_{\bo c'}}}.
\]
%\comentario{Here I erased the extra combinatorial term}
Therefore, from our discussion above and the limit in Proposition \ref{teoJointLawUnderQSSameLawAsUnderQk}, we deduce
\[
\begin{split}
 \lim_{T\to \infty}\Q^{(k),\bm \theta_T}_{\bo w, T,r}\Big(\Delta_T(k-1)\Big)
& =\lim_{T\to \infty}\Q^{( k),\bm{\theta}_T}_{T,r}(\Delta_T(k-1))\sum_{ \bo c\in \bo [d]^k} \frac{ \bo w^{\bo{D}_{\bo c}}\bm \eta^{\bo D_{\bo c}}}{\sum_{ \bo c'\in \bo [d]^k} \bo w^{\bo{D}_{\bo c'}}\bm \eta^{\bo D_{\bo c'}}}\\
& =\lim_{T\to \infty}\Q^{( k),\bm{\theta}_T}_{T,r}(\Delta_T(k-1)),
\end{split}
\]
%\comentario{Here I erased the extra combinatorial term}
as required.
\end{proof}

\subsection{Proof of Theorem \ref{mainresult2}}

For ease of exposition, we prove Theorem \ref{mainresult2} through a sequence of intermediate results. We begin by establishing  identity \eqref{dk2withcolors}, which corresponds to the limit 
 of   $\p^{(k)}_{unif,T,r}(\Delta_T(k-1))$.

\begin{propo}\label{teo7}	Under assumption \eqref{hyprifa}, we have 
\[
	\begin{split}
	\lim_{T\to \infty}&\p^{(k)}_{unif,T,r}\paren{\Delta_T(k-1)} =   \frac{2^{k-1}}{(k-1)!}\left(\prod_{h=1}^{k-1}\frac{\zeta_{i_h}}{\zeta}\frac{p_{i_h}(\bm \ell_{h})w(\bm \ell_{h})}{\mathbb{E}_{i_h}\left[w(\bo L)\right]}\frac{\bm \ell_{h}^{\floor{\bo g_h}}\bm \xi^{\bo g_h}}{w(\bm \ell_{h})}{\rm d}\rho_h\right)\\
	&\hspace{5.5cm}\times\int_0^\infty  \frac{y^{k-1}}{(1+y)^2} \prod_{h=1}^{k-1}\frac{1}{(1+(1-\rho_h)y)^2} {\rm d}y.\\
	\end{split}
	\]
\end{propo}
\begin{proof}
From Theorem \ref{propofsplitting1} and applying the change of variables $\phi=2\omega/(\zeta T)$, we get
\begin{equation}\label{unifsamlimitvsint}
	\begin{split}
		\mathbb{P}^{(k)}_{unif,T,r}\left( \Delta_T(k-1)\right)&  = 
\frac{1}{(k-1)!}\frac{2}{\zeta }\\
&\times\int_0^\infty\left(e^\frac{2\omega}{\zeta T}-1\right)^{k-1}\Q^{(k),\frac{2\omega}{\zeta T}\vec{1}}_{T,r}(\Delta_T(k-1))\frac{\E_{r}\left[N^{\floor{k}}_Te^{-\frac{2\omega}{\zeta T}\vec{1}\cdot \bo Z_T}\right]}{\mathbb{P}_{r}\paren{\ N_T\geq k}}\frac{{\rm d}\omega}{T}.
\end{split}
\end{equation}
Now from Proposition \ref{propSophie} and \eqref{eqnAsymptoticOfProbabilityOfNonExtinction} we observe that for $T$ large enough, we have
\begin{equation}\label{conssophie}
\mathbb{P}_{r}\paren{ N_T\geq k}=\mathbb{P}_{r}\paren{ N_T\geq k, \bo Z_T\neq \bo 0}=\mathbb{P}_{r}\paren{ N_T\geq k| \bo Z_T\neq \bo 0}\mathbb{P}_{r}\paren{ \bo Z_T\neq \bo 0}\sim \frac{2\xi_r}{\zeta T},
\end{equation}
where the conditional probability goes to 1 since
\[
\left(\frac{N_T}{T}\ \Big| \bo{Z}_T\ne \bo{0}\right) \stackrel{(d)}{\to} \frac{\zeta(\bm  \eta\cdot \vec{1})}{2}\gamma.
\]
For the moment, let us assume that the limit and the integral on the right-hand side of \eqref{unifsamlimitvsint} can be interchanged; we will justify this step below. Thus using \eqref{conssophie},   \eqref{eqnDiscountedPowerOfSizeOfGaltonWatson} (with $\rho=0$) and
\begin{equation}\label{exponentialasymp}
\lim_{T\to\infty}T^{k-1}\left(e^\frac{2\omega}{\zeta T}-1\right)^{k-1}=\left(\frac{2\omega}{\zeta}\right)^{k-1},
\end{equation}
we observe
\[
	\begin{split}
		\lim_{T\to\infty}\mathbb{P}^{(k)}_{unif,T,r}\left( \Delta_T(k-1)\right)&= \frac{2^{k-1}(\bm \eta \cdot \vec{1})^k}{(k-1)!}\prod_{h=1}^{k-1}\frac{\zeta_{i_h}}{\zeta}\frac{p_{i_h}(\bm \ell_{h})w(\bm \ell_{h})}{\mathbb{E}_{i_h}\left[w(\bo L)\right]}\frac{\bm \ell_{h}^{\floor{\bo g_h}}\bm \xi^{\bo g_h}}{w(\bm \ell_{h})}\\
&\hspace{.5cm}\times\left(\int_0^\infty\omega^{k-1}(1+\omega\bm \eta \cdot \vec{1})^{-2}\frac{1 }{(1+\omega(1-\rho_h)\bm \eta \cdot \vec{1})^{2}}{\rm d}\omega\right){\rm d}\rho_h.
	\end{split}
\]
To obtain our result, we simply perform the change of variables 
$y=x\bm \eta\cdot \vec{1}$.

To conclude the proof, we justify that the limit and the integral on the right-hand side of \eqref{unifsamlimitvsint} can be interchanged by invoking the Generalized Lebesgue Dominated Convergence Theorem. Indeed, we denote by $f(T,\omega)$ the integrand in the right-hand side of \eqref{unifsamlimitvsint} and  
observe 
\begin{equation}\label{eqnInequalityfAndgUniformSampling}
f(T,\omega)\leq 
 T^{k-1}\left(e^\frac{2\omega}{\zeta T}-1\right)^{k-1}\frac{\E_{r}\left[N^{\floor{k}}_Te^{-\frac{2\omega}{\zeta T}\vec{1}\cdot \bo Z_T}\right]}{T^k\mathbb{P}_{r}\paren{\ N_T\geq k}}=:g(T,\omega).
\end{equation}
From \eqref{unifsamlimitvsint} with $A=\Omega$, it is clear
\[
\int_0^\infty g(T,\omega){\rm d}\omega=(k-1)!\frac{\zeta}{2}\p^{(k)}_{unif,T,r}\left( \Omega\right)=(k-1)!\frac{\zeta}{2}, \qquad \textrm{for } T> 0.
%1=\p^{(k)}_{unif,T,r}\left( \Omega\right)
%=\frac{1}{(k-1)!}\paren{\frac{2}{\zeta}}^k\int_0^\infty g(T,\omega){\rm d}\omega,
\]
Next, we show that, as $T\to \infty$, the limits of $f(T,\cdot)$ and $g(T, \cdot)$ exist pointwise. By the asymptotic in \eqref{eqnLimitOfJointLaw}, it is enough to establish this convergence  for $g(T, \cdot)$. By \eqref{eqnDiscountedPowerOfSizeOfGaltonWatson} (with $\rho=0$),  \eqref{conssophie} and \eqref{exponentialasymp},
as $T$ goes to infinity, we have
\[
g(T,\omega)\to \omega^{k-1}\paren{\frac{\zeta}{2}}^{1-k}\frac{\paren{\frac{\zeta}{2}}^{k-1} \frac{k!\xi_r(\bm \eta\cdot \vec{1})^k}{\paren{1+\omega\bm \eta\cdot \vec{1}}^{k+1}}}{\frac{2\xi_{r}}{\zeta }}
=\frac{\zeta}{2} k!(\bm \eta\cdot \vec{1})^k\frac{\omega^{k-1}}{\paren{1+\omega\bm \eta\cdot \vec{1}}^{k+1}}.
\]Finally, we prove that 
\[
\lim_{T\to \infty}\int_0^\infty g(T,\omega){\rm d}\omega=\int_0^\infty \lim_{T\to \infty}g(T,\omega){\rm d}\omega,
\]which by the previous discussions, it is equivalent to prove that
\[
(\bm \eta\cdot \vec{1})^k\int_0^\infty \frac{\omega^{k-1}}{\paren{1+\omega\bm \eta\cdot \vec{1}}^{k+1}}{\rm d}\omega=\frac{1}{k}.
\]The latter follows by  performing  the following change of variables,  $\omega\bm \eta\cdot \vec{1}=t/(1-t)$, i.e.
\begin{equation}\label{usefulinte}
\begin{split}
(\bm \eta\cdot \vec{1})^k\int_0^\infty \frac{\omega^{k-1}}{\paren{1+\omega\bm \eta\cdot \vec{1}}^{k+1}}{\rm d}\omega
=\int_0^1 \frac{\paren{\frac{t}{1-t}}^{k-1}}{\paren{1-t}^{-k-1}}\frac{1}{(1-t)^2}{\rm d}t
=\int_0^1 t^{k-1}{\rm d}t=\frac{1}{k}.
\end{split}
\end{equation}
In other words, by the Generalised Lebesgue Dominated Convergence Theorem, we conclude 
\[
\lim_{T\to \infty}\int_0^\infty f(T,\omega){\rm d}\omega=\int_0^\infty \lim_{T\to \infty}f(T,\omega){\rm d}\omega,
\]
as desired. \end{proof}

We now proceed to deduce the first identity in \eqref{limitsigualdad}.
\begin{propo}\label{prop8}
Under assumption \eqref{hyprifa}, we have 
\begin{equation}\label{limitimpS}
	\lim_{T\to \infty}\p^{(k)}_{\bo c,T,r}(\Delta_T(k-1))=\lim_{T\to \infty}\p^{(k)}_{unif,T,r}\paren{\Delta_T(k-1)}. 
	\end{equation}
\end{propo}
\begin{proof} We first recall from \eqref{eqpunifcIntro}, the explicit form of $\p^{(k)}_{\bo c,T,r}(\Delta_T(k-1))$. We also recall that $\{\bo Z_{t}\geq \bo D_{\bo c}\}=\{Z^{(m)}_{t}\geq D_{\bo c,m}, \textrm{ for all } m\in [d]\}$
 and using the asymptotic in \eqref{eqnAsymptoticOfProbabilityOfNonExtinction}, we  note that, for $T$ large enough,
\begin{equation}\label{limitDc}
\mathbb{P}_{r}\paren{ \bo Z_T\geq \bo D_{\bo c}}\sim \mathbb{P}_{r}\paren{\bo Z_T\neq \bo 0}\sim \frac{2\xi_{r} }{\zeta }\frac{1}{T}.
\end{equation}
For now, we shall assume that it is permissible to interchange the limit and the integral on the right-hand side of \eqref{eqpunifcIntro}; the validity of this step will be justified later.
Without loss of generality, we also assume $S_{\bo c}=[d]$, since the general case follows a similar proof.

Putting all the pieces together in \eqref{eqpunifcIntro}  for  the set of interest (with $n=k-1$), and applying the change of variables $\bm \phi=2\bm \omega/(\zeta T)$, with Jacobian $|\textrm{det}\ J|=2^d /(\zeta T)^d$,  we obtain, 
\[
\begin{split}
&\lim_{T\to\infty}\p^{k}_{\bo c,T,r}(\Delta_T(k-1))=\lim_{T\to\infty}\frac{1}{\mathbb{P}_{r}\paren{\ \bo Z_T\geq \bo D_{\bo c}}}\left(\prod_{m=1}^d\frac{1}{(D_{\bo c,m}-1)!}\right)\\
&\hspace{4cm}\times \int_{\re^d_+} \big(e^{\bm\phi}-\vec 1\big)^{\bo D_{\bo c}-\vec{1}}\Q^{(k),\bm \phi}_{\bo c,T,r}(\Delta_T(k-1))\E_{r}\left[\bo Z^{\floor{\bo D_{\bo c}}}_Te^{-\bm \phi\cdot \bo Z_T}\right]{\rm d}\bm \phi\\
&= \prod_{m=1}^d\frac{1}{(D_{\bo c,m}-1)!}\frac{2^{d-1}\zeta^{1-d} }{\xi_{r}}\\
&\hspace{2cm}\times\int_{\re^d_+}\lim_{T\to\infty}T^{1-d} \left(e^{\frac{2\bm \omega}{\zeta T}}-\vec 1\right)^{\bo D_{\bo c}-\vec{1}}%\paren{\prod_{m=1}^{d}\frac{2^{D_{\bo c,m}-1}}{(\zeta T)^{D_{\bo c,m}-1}}}
\Q^{(k),\frac{2}{\zeta T}\bm \omega}_{\bo c, T,r}\left(\Delta_T(k-1)\right)\E_{r}\left[\bo Z^{\floor{\bo D_{\bo c}}}_Te^{-\frac{2}{\zeta T}\bm \omega\cdot \bo Z_T}\right]{\rm d}\bm \omega\\
%&\hspace{1cm}\sim \prod_{m=1}^d\frac{1}{(D_{\bo c,m}-1)!}\frac{\zeta T}{2\xi_{r}}\frac{2^d}{(\zeta T)^d}\\
%&\hspace{2.5cm}\times \int_{\re^d_+}\paren{\prod_{m=1}^{d}\frac{2^{D_{\bo c,m}-1}}{(\zeta T)^{D_{\bo c,m}-1}}}\bm \omega^{\bo D_{\bo c}-\vec{1}}\Q^{(\bo c),\frac{2}{\varsigma T}\bm \omega}_{T,r}\left(\Delta_T(k-1)\right)\E_{r}\left[\bo Z^{\floor{\bo D_{\bo c}}}_Te^{-\frac{2}{\varsigma T}\bm \omega\cdot \bo Z_T}\right]{\rm d}\bm \omega\\
&= \prod_{m=1}^d\frac{1}{(D_{\bo c,m}-1)!}\frac{2^{k-1}}{\xi_{r}\zeta^{k-1}}\\
&\hspace{2cm}\times \int_{\re^d_+}\bm \omega^{\bo D_{\bo c}-\vec{1}}\lim_{T\to\infty}\frac{1}{T^{k-1}}\Q^{(k),\frac{2}{\zeta T}\bm \omega}_{\bo c, T,r}\left(\Delta_T(k-1)\right)\E_{r}\left[\bo Z^{\floor{\bo D_{\bo c}}}_Te^{-\frac{2}{\zeta T}\bm \omega\cdot \bo Z_T}\right]{\rm d}\bm \omega,
\end{split}
\]
where in the last identity, we had use
\[
\lim_{T\to\infty}T^{k-d}\left(e^{\frac{2\bm \omega}{\zeta T}}-\vec 1\right)^{\bo D_{\bo c}-\vec{1}}=\prod_{m=1}^d\left(\frac{2 \omega_m}{\zeta}\right)^{D_{\bo c,m}-1}.
\]
On the other hand, using Propositions \ref{lemmaLimitOfJointLaw} and \ref{teoJointLawUnderQSSameLawAsUnderQk},  and the asymptotic in \eqref{eqnDiscountedPowerOfSizeOfGaltonWatson2}, one obtains
\[
\begin{split}
\lim_{T\to\infty}\frac{1}{T^{k-1}}&\Q^{(k),\frac{2}{\zeta T}\bm \omega}_{\bo c,T,r}\left(\Delta_T(k-1)\right)\E_{r}\left[\bo Z^{\floor{\bo D_{\bo c}}}_Te^{-\frac{2}{\zeta T}\bm \omega\cdot \bo Z_T}\right]\\
& = \xi_{r}\bm \eta^{\bo D_{\bo c}} 
\left(\prod_{h=1}^{k-1}p_{i_h}(\bm \ell_{h})\bm \ell_{h}^{\floor{\bo g_h}}\bm \xi^{\bo g_h}\eta_{i_h} \alpha_{i_{h}}\right) \frac{1}{(1+\bm \eta\cdot \bm \omega)^2}\prod_{h=1}^{k-1}\frac{1}{(1+(1-\rho_h)\bm \eta \cdot \bm \omega)^{2}} {\rm d}\rho_h.
\end{split}
\]Thus, we have
\begin{equation}\label{eqnPSUnifAsAnIntegral}
\begin{split}
\lim_{T\to\infty}&\p^{(k)}_{\bo c,T,r} (\Delta_T(k-1))= \frac{2^{k-1}}{\zeta^{k-1}}\eta^{\bo D_{\bo c}} \left(\prod_{m=1}^d\frac{1}{(D_{\bo c,m}-1)!}\right)\bm 
\prod_{h=1}^{k-1}p_{i_h}(\bm \ell_{h})\bm \ell_{h}^{\floor{\bo g_h}}\bm \xi^{\bo g_h}\eta_{i_h} \alpha_{i_{h}}\\
&\hspace{4cm}\times \int_{\re^d_+}{\rm d}\bm \omega\, \bm \omega^{\bo D_{\bo c}-\vec{1}}\frac{1}{(1+\bm \eta\cdot \bm \omega)^2}\prod_{h=1}^{k-1}\frac{1}{(1+(1-\rho_h)\bm \eta \cdot \bm \omega)^{2}} \ {\rm d}\rho_h.
\end{split}
\end{equation}
The previous identity can be simplified. Indeed, we  use Liouville's extension of Dirichlet integral  (see for instance the identities in  \cite{MR2360010}, 4.635) to rewrite the previous integral as follows
\[
\begin{split}
\int_{\re^d_+}{\rm d}\bm \omega\, \bm \omega^{\bo D_{\bo c}-\vec{1}}\frac{1}{(1+\bm \eta\cdot \bm \omega)^2}\prod_{h=1}^{k-1}\frac{1}{(1+(1-\rho_h)\bm \eta \cdot \bm \omega)^{2}}&=\frac{\left(\prod_{m=1}^{d}(D_{\bo c,m}-1)!\right)}{(k-1)!}\frac{1}{\bm \eta^{\bo D_{\bo c}}}\\
&\hspace{-2cm}\times\int_0^\infty\omega^{k-1}(1+\omega)^{-2}\prod_{h=1}^{k-1}\frac{1 }{(1+\omega(1-\rho_h))^{2}}{\rm d}\omega.
\end{split}\]
By combining the preceding identity with \eqref{eqnPSUnifAsAnIntegral} we deduce the desired identity.

In order to conclude the proof of the previous result, we require to justify the interchange of the limit with the integral on the right-hand side of \eqref{eqpunifcIntro}. We proceed similarly as in the proof of Proposition \ref{teo7}. That is to say, we invoke the Generalised Lebesgue Dominated Convergence Theorem. Let 
\[
\begin{split}
f(T,\omega)&:= T^{k-d}\left(e^{\frac{2\bm \omega}{\zeta T}}-\vec 1\right)^{\bo D_{\bo c}-\vec{1}}\Q^{(\bo c),\frac{2}{\zeta T}\bm \omega}_{T,r}(\Delta_T(k-1))\frac{\E_{r}\left[\bo Z^{\floor{\bo D_{\bo c}}}_Te^{-\frac{2}{\zeta T}\bm \omega\cdot \bo Z_T}\right]}{T^k\mathbb{P}_{r}\paren{\ \bo Z_T\geq \bo D_{\bo c}}}\\
&\le T^{k-d}\left(e^{\frac{2\bm \omega}{\zeta T}}-\vec 1\right)^{\bo D_{\bo c}-\vec{1}}\frac{\E_{r}\left[\bo Z^{\floor{\bo D_{\bo c}}}_Te^{-\frac{2}{\zeta T}\bm \omega\cdot \bo Z_T}\right]}{T^k\mathbb{P}_{r}\paren{\ \bo Z_T\geq \bo D_{\bo c}}}=:g(T,\omega).
\end{split}
\] 
From Propositions \ref{lemmaLimitOfJointLaw} and \ref{teoJointLawUnderQSSameLawAsUnderQk}, it follows that the limit of $f(T,\omega)$ exists whenever the limit of  $g(T,\omega)$ does. Thus we use the asymptotic behaviours in \eqref{eqnDiscountedPowerOfSizeOfGaltonWatson2} and \eqref{limitDc}, together with  $\sum_{m\in[d]}D_{\bo c,m}=k$, to deduce
\begin{equation}\label{glimitDc}
\lim_{T\to\infty}g(T,\bm \omega)=\lim_{T\to\infty} T^{k-d}\left(e^{\frac{2\bm \omega}{\zeta T}}-\vec 1\right)^{\bo D_{\bo c}-\vec{1}}\frac{\frac{\big(\frac{\zeta T}{2}\big)^{k-1}k!\xi_r\bm \eta^{\bo D_{\bo c}}}{(1+\bm \eta\cdot \bm \omega)^{k+1}}}{T^k\frac{2\xi_r}{\zeta T}}=k!\left(\frac{\zeta}{2}\right)^{d}\bm \eta^{\bo D_{\bo c}}\frac{\bm \omega^{\bo D_{\bo c}-\vec{1}}}{(1+\bm \eta\cdot \bm \omega)^{k+1}}.
\end{equation}
This proves that the limit exists. Recall from the proof of Theorem \ref{propofsplitting2Intro} that identity \eqref{eqpunifcIntro} holds for any $A\in \mathcal{F}^{(k)}_T$. Thus when $A=\Omega$, after applying the change of variables $\bm \phi=2\bm \omega/(\zeta T)$ we obtain
\[
1=\left(\frac{2}{\zeta}\right)^d\left(\prod_{m=1}^d\frac{1}{(D_{\bo c,m}-1)!}\right)\int_{\re^d_+}g(T,\bm \omega){\rm d}\bm \omega.
\]
On the other hand, from the limit in \eqref{glimitDc} and Liouville's extension of Dirichlet integral, we see 
\[
\int_{\re^d_+}\lim_{T\to \infty}g(T,\bm \omega){\rm d}\bm \omega=k!\left(\frac{\zeta}{2}\right)^{d}\bm \eta^{\bo D_{\bo c}}\int_{\re^d_+}\frac{\bm \omega^{\bo D_{\bo c}-\vec{1}}}{(1+\bm \eta\cdot \bm \omega)^{k+1}}{\rm d}\bm \omega=\left(\frac{\zeta}{2}\right)^d\left(\prod_{m=1}^d(D_{\bo c,m}-1)!\right).
\]
In other words, we have proved that 
\[
\lim_{T\to \infty}\int_{\re^d_+}g(T,\bm \omega){\rm d}\bm \omega=\int_{\re^d_+}\lim_{T\to \infty}g(T,\bm \omega){\rm d}\bm \omega,
\]
which from the Generalised Lebesgue Dominated Convergence Theorem implies 
\[
\lim_{T\to\infty}\int_{\re^d_+}f(T,\bm \omega){\rm d}\bm \omega=\int_{\re^d_+}\lim_{T\to\infty}f(T,\bm \omega){\rm d}\bm \omega,
\]
as required.

\end{proof}

We now turn to the proof of the remaining identity in \eqref{limitsigualdad}. To this end, we first establish some auxiliary results. The following asymptotic relation will play a key role in the argument.
\begin{lemma}\label{lemmaAsymptoticGForQW}
Assume that hypothesis \eqref{hyprifa} holds. Let  $\phi\in \mathbb{R}_+$ and $\rho\in [0,1)$. Then 
\[
\begin{split}
\lim_{T\to\infty}\frac{1}{T^{k-1}}\E_{m}&\left[e^{-\frac{2\phi}{\zeta T}\bo w\cdot \bo Z_{(1-\rho)T}}\sum_{\bo v\in \mathcal{N}^{(k)}_T}\bo w^{\bo D_{\bo v}}\right] =\paren{(1-\rho)\frac{\zeta}{2}}^{k-1}\frac{k!\xi_m\big(\bm \eta\cdot \bo w\big)^k}{\paren{1+(1-\rho)\phi\bm \eta\cdot \bo w }^{k+1}}.
\end{split}
\]
\end{lemma}
%\comentario{Here I erased the extra combinatorial term}
\begin{proof}
First, we decompose 
\begin{align*}
& \E_m\left[e^{-\frac{2\phi}{\zeta T}\bo w\cdot \bo Z_{(1-\rho)T}}\sum_{\bo v\in \mathcal{N}^{(k)}_{(1-\rho)T}}\bo w^{\bo D_{\bo v}}\right]= \sum_{\substack{\bo D\in \{0,1,\ldots,k\}^d:\\ \sum_{m=1}^dD^{(m)}=k}}{k \choose \bo D}\E_m\left[e^{-\frac{2\phi}{\zeta T}\bo w\cdot \bo Z_{(1-\rho)T}}\bo Z^{\floor{\bo D}}_T\bo w^{\bo D}\right].
\end{align*}
%\comentario{Here I erased the extra combinatorial term}
Following the proof of Lemma \ref{eqnAsymptoticZDsExp}, we can easily obtain the asymptotic order of each expectation in the display above, that is to say
\[
\begin{split}
 \E_m&\left[e^{-\frac{2\phi}{\zeta T}\bo w\cdot \bo Z_{(1-\rho)T}}\sum_{\bo v\in \mathcal{N}^{(k)}_{(1-\rho)T}}\bo w^{\bo D_{\bo v}}\right]\sim \paren{(1-\rho)T\frac{\zeta}{2}}^{k-1}\\
 &\hspace{3cm}\times\frac{k!\xi_m}{\paren{1+(1-\rho)\phi\bm \eta\cdot \bo w }^{k+1}}\sum_{\substack{\bo D\in \{0,1,\ldots,k\}^d:\\ \sum_{m=1}^dD^{(m)}=k}}{k \choose \bo D} \bm \eta^{\bo D}\bo w^{\bo D}\\
&\hspace{3cm}= \paren{(1-\rho)T\frac{\zeta}{2}}^{k-1}\frac{k!\xi_m}{\paren{1+(1-\rho)\phi\bm \eta\cdot \bo w }^{k+1}}\big(\bm \eta\cdot \bo w\big)^k,
\end{split}
\]
%\comentario{Here I erased the extra combinatorial term}
where the last identity follows from the multinomial theorem. The asymptotic clearly implies our result. 
\end{proof}
The following result plays also a crucial  rol in establishing the remaining identity in \eqref{limitsigualdad}. 
\begin{lemma}\label{limitunderqw}
Assume that hyothesis \eqref{hyprifa} holds. Let  $\phi\in \mathbb{R}_+$ and define $\phi_T:=\frac{2\phi}{\zeta T}$. Then,  we have
\begin{equation}\label{eqnQwFractionAndWithoutraction}
\lim_{T\to\infty}\Q^{(k),\phi_T\bo w}_{\bo w, T,r}\left[\frac{\big(\bo Z_T\cdot \bo w\big)^k}{\sum_{\bo v\in \mathcal{N}^{(k)}_T}\bo w^{\bo D_{\bo v}}}\mathbf{1}_{\Delta_T(k-1)}\right]=\lim_{T\to\infty} \Q^{(k),\phi_T \bo w}_{T,r}\Big(\Delta_T(k-1)\Big).
\end{equation}
%\comentario{Here I erased the extra combinatorial term}
\end{lemma}
\begin{proof}
Consider $n\geq k$ arbitrary natural numbers. 
Then
\[
\frac{n^k}{n^{\floor{k}}}=\frac{n^k}{n^k\paren{1-\frac{1}{n}}\cdots \paren{1-\frac{k-1}{n}}}=\prod_{j=1}^{k-1}\paren{1-\frac{j}{n}}^{-1}.
\]Using $(1-x)^{-1}=\sum_{\ell=0}^\infty x^\ell$ for $|x|<1$, collecting the terms with  constant coefficient and those with coefficient $n^{-1}$, we get
\begin{equation}\label{eqnAsymptoticsFornkAndnFloork}
\frac{n^k}{n^{\floor{k}}}=1+\sum_{j=1}^{k-1}\frac{j}{n}+O\paren{\sum_{j=1}^{k-1}\frac{j^2}{n^2}}=1+\frac{k(k-1)}{2}\frac{1}{n}+O\paren{\frac{1}{n^2}}.
\end{equation}Note that the error term can be bounded from above by a constant times $n^{-2}$. 
Let us define the fraction appearing on the left-hand side of \eqref{eqnQwFractionAndWithoutraction} as
\[
R_T:=\frac{\big(\bo Z_T\cdot \bo w\big)^k}{\sum_{\bo v\in \mathcal{N}^{(k)}_T}\bo w^{\bo D_{\bo v}}}=\frac{\displaystyle\sum_{|\bo D|=k}{k \choose \bo D}\bo Z^{\bo D}_T\bo w^{\bo D}}{\displaystyle\sum_{|\bo D|=k}{k \choose \bo D}\bo Z^{\floor{\bo D}}_T\bo w^{\bo D}},
\]
%\comentario{Here I erased the extra combinatorial term}
where the last equality follows from  \eqref{eqnMultinomialDenominatorQS0} and \eqref{eqnApproximationToMultinomialDenominatorQS}, and writing the sum over $\{\bo D\in \{0,1,\ldots,k\}^d: \sum_{m=1}^dD_{m}=k\}$ simply as $\{|\bo D|=k\}$.  Let $A\in \mathcal{F}_T^{(k)}$, thus
from the identity  
\[
R_T\mathbf{1}_{A}=(R_T-1)\mathbf{1}_{A}+\mathbf{1}_{A},
\] we obtain 
\begin{equation}\label{eqnInequalityQwAndRTMinusOne}
\begin{split}
\left|\Q^{(k),\phi_T\bo w}_{\bo w, T,r}\left[R_T\mathbf{1}_{A}\right]-\Q^{(k),\phi_T\bo w}_{\bo w, T,r}\paren{A}\right|&\leq \Q^{(k),\phi_T\bo w}_{\bo w, T,r}\left[R_T-1\right],
\end{split}
\end{equation}
since $R_T\ge 1$. 
 Using that $\bo Z^{\bo D}_t\geq \bo Z^{\floor{\bo D}}_t$ for all $\bo D\in\{0,1,\ldots,k\}^d$ and \eqref{eqnAsymptoticsFornkAndnFloork}, we deduce
\[
\begin{split}
\bo Z^{\bo D}_t&=\bo Z^{\floor{\bo D}}_t\prod_{m\in [d]}\frac{\left(Z^{(m)}_T\right)^{D_m}}{\left(Z^{(m)}_T\right)^{\floor{D_m}}}\\
&=\bo Z^{\floor{\bo D}}_t\prod_{m\in [d]}\paren{1+\frac{D_m(D_m-1)}{2}\frac{1}{Z^{(m)}_T}+O\paren{\left(Z^{(m)}_T\right)^{-2}}}\\
&=\bo Z^{\floor{\bo D}}_t\paren{1+\sum_{m\in [d]}\frac{D_m(D_m-1)}{2}\frac{1}{Z^{(m)}_T}+O\paren{\max_{m\in [d]}\left(Z^{(m)}_T\right)^{-2}}}.
\end{split}
\]This implies
\[
R_T-1=\frac{\sum_{m\in [d]}\frac{1}{2Z^{(m)}_T}\sum_{|\bo D|=k}{k \choose \bo D}\bo Z^{\floor{\bo D}}_T\bo w^{\bo D}D_m(D_m-1)}{\sum_{|\bo D|=k}{k \choose \bo D}\bo Z^{\floor{\bo D}}_T\bo w^{\bo D}}+O\paren{\max_{m\in [d]}\left(Z^{(m)}_T\right)^{-2}}.
\]
Let us denote by $\mathcal{R}_T$ the first term on the right-hand side above. 
Substituting the previous into \eqref{eqnInequalityQwAndRTMinusOne} and using the definition of $\Q^{(k),\phi_T\bo w}_{\bo w, T,r}$ in  \eqref{defMeasureQkTrGivenAllInfoS3}, we obtain
\[
\begin{split}
\Q^{(k),\phi_T\bo w}_{\bo w, T,r}\left[R_T-1\right]& =\frac{\E^{(k)}_r\left[\mathcal{R}_T e^{-\phi_T\bo w\cdot \bo Z_{T}}g_{k,T}{\bo w}^{{\bo D}_{\bm \varsigma_T}}\right]}{\E_r\left[e^{-\phi_T\bo w\cdot \bo Z_{T}}\sum_{\bo{v}\in \mathcal{N}_{T}^{(k)}}\bo w^{\bo D_{\bo v}}\right]}\\
&\hspace{1cm}+\frac{\E^{(k)}_r\left[O\paren{\max_{m\in [d]}\left(Z^{(m)}_T\right)^{-2}}e^{-\phi_T\bo w\cdot \bo Z_{T}}g_{k,T}{\bo w}^{{\bo D}_{\bm \varsigma_T}}\right]}{\E_r\left[e^{-\phi_T\bo w\cdot \bo Z_{T}}\sum_{\bo{v}\in \mathcal{N}_{T}^{(k)}}\bo w^{\bo D_{\bo v}}\right]}.
\end{split}
\]
%\comentario{Here I erased the extra combinatorial term}
Let us first rewrite the numerator of the first term on the right-hand side of the above identity. Specifically, we condition by $\mathcal{F}_T$ and then  use \eqref{tomate3} and \eqref{eqnMultinomialDenominatorQS0} to obtain 
\[
\begin{split}
\E^{(k)}_r&\left[\mathcal{R}_Te^{-\phi_T\bo w\cdot \bo Z_{T}}g_{k,T}{\bo w}^{{\bo D}_{\bm \varsigma_T}}\right]\\
&\hspace{3cm}=\E^{(k)}_r\left[e^{-\phi_T\bo w\cdot \bo Z_{T}}\sum_{m\in [d]}\frac{1}{2Z^{(m)}_T}\sum_{|\bo D|=k}{k \choose \bo D}\bo Z^{\floor{\bo D}}_T\bo w^{\bo D}D_m(D_m-1)\right]\\
&\hspace{3cm} = \sum_{m\in [d]}\frac{1}{2}\sum_{|\bo D|=k}{k \choose \bo D}D_m(D_m-1)\bo w^{\bo D}\E_r\left[\frac{1}{Z^{(m)}_T}\bo Z^{\floor{\bo D}}_t e^{-\phi_T\bo w\cdot \bo Z_{T}}\right].
\end{split}
\]
%\comentario{Here I erased the extra combinatorial term}
Proceeding  similarly as in the proof of Lemma \ref{eqnAsymptoticZDsExp} (basically applying Yaglom's limit from Proposition \ref{propSophie}), we can deduce
\[
\E_r\left[\frac{1}{Z^{(m)}_T}\bo Z^{\floor{\bo D}}_{T}e^{-\phi_T\bo w\cdot \bo Z_{T}}\right] \sim \paren{T\frac{\zeta}{2}}^{k-2} \frac{(k-1)!\xi_r\bm \eta^{\bo D}}{\eta_m\paren{1+\bm \eta\cdot \phi\bo w }^{k}}.
\]The latter, together with Lemma \ref{lemmaAsymptoticGForQW} imply
\[
\frac{\E_r\left[\mathcal{R}_Te^{-\phi_T\bo w\cdot \bo Z_{t}}g_{k,t}{\bo w}^{{\bo D}_{\bm \varsigma_t}}\right]}{\E_r\left[e^{-\phi_T\bo w\cdot \bo Z_{t}}\sum_{\bo{v}\in \mathcal{N}_{t}^{(k)}}\bo w^{\bo D_{\bo v}}\right]}=O\paren{\frac{1}{T}}.
\]
%\comentario{Here I erased the extra combinatorial term}
For the remainder term, we proceed in a similar way and obtain\[
\begin{split}
&\frac{\E^{(k)}_r\left[O\paren{\max_{m\in[d]}\left(Z^{(m)}_T\right)^{-2}}e^{-\phi_T\bo w\cdot \bo Z_{T}}g_{k,T}{\bo w}^{{\bo D}_{\bm \varsigma_T}}\right]}{\E_r\left[e^{-\phi_T\bo w\cdot \bo Z_{T}}\sum_{\bo{v}\in \mathcal{N}_{T}^{(k)}}\bo w^{\bo D_{\bo v}}\right]} \\
&\hspace{4cm}=\frac{\E_r\left[O\paren{\max_{m\in [d]}\left(Z^{(m)}_T\right)^{-2}}e^{-\phi_T\bo w\cdot \bo Z_{T}}\sum_{|\bo D|=k}{k \choose \bo D}\bo Z^{\floor{\bo D}}_T\bo w^{\bo D}
\right]}{\E_r\left[e^{-\phi_T\bo w\cdot \bo Z_{T}}\sum_{|\bo D|=k}{k \choose \bo D}\bo Z^{\floor{\bo D}}_T\bo w^{\bo D}
\right]}\\
&\hspace{6cm} =O\paren{T^{-2}}.
\end{split}
\]
%\comentario{Here I erased the extra combinatorial term}
In other words, we have proven that 
\[
\lim_{T\to\infty}\Q^{(k),\phi_T\bo w}_{\bo w, T,r}\left[\frac{\big(\bo Z_T\cdot \bo w\big)^k}{\sum_{\bo v\in \mathcal{N}^{(k)}_T}\bo w^{\bo D_{\bo v}}}\mathbf{1}_{A}\right]=\lim_{T\to\infty} \Q^{(k),\phi_T\bo w}_{\bo w, T,r}\Big(A\Big).
\]
%\comentario{Here I erased the extra combinatorial term}
The result now follows by taking $A=\Delta_T(k-1)$ and  Proposition \ref{propLimitQWEqualsLimitQk}.
\end{proof}
Having gathered all the required tools, we are now ready to deduce the remaining identity in \eqref{limitsigualdad}. 

\begin{propo}\label{prop9}
Under assumption \eqref{hyprifa}, we have
\[
	\lim_{T\to \infty}\p^{(k)}_{\bo w,T,r}\paren{\Delta_T(k-1)}=\lim_{T\to \infty}\p^{(k)}_{unif,T,r}\paren{\Delta_T(k-1)}
	\]
\end{propo}
\begin{proof} From \eqref{eqnfsplitting3Intro} and after performing    the change of variables $\phi=2x/(\zeta T)$, we obtain
\begin{equation}\label{eqnJointLawUnderPUnifWNewRescaled}
\begin{split}
\p^{(k)}_{\bo w,T,r}&\left( \Delta_T(k-1)\right)= \frac{1}{(k-1)!}\left(\frac{2}{\zeta }\right)^{k}\int_0^\infty x^{k-1}\Q^{(\bo w),\frac{2x}{\zeta T}\bo w}_{T,r}\left[\frac{\big(\bo Z_T\cdot \bo w\big)^k}{\sum_{\bo v\in \mathcal{N}^{(k)}_T}\bo w^{\bo D_{\bo v}}}\mathbf{1}_{\Delta_T(k-1)}\right]\\
& \hspace{6.5cm}\times\frac{\E_{r}\left[e^{-\frac{2x}{\zeta T}\bo w\cdot \bo Z_T}\sum_{\bo v\in \mathcal{N}^{(k)}_T}\bo w^{\bo D_{\bo v}}\right]}{T^k\mathbb{P}_{r}\paren{\ N_T\geq k}}{\rm d}x.\\
\end{split}
\end{equation}
%\comentario{Here I erased the extra combinatorial term}
For the moment, we  assume that it is permissible to interchange the limit with the integral on the right-hand side of \eqref{eqnJointLawUnderPUnifWNewRescaled};  this step will be justified later. It then follows from Lemmas \ref{lemmaAsymptoticGForQW} and \ref{limitunderqw}, together with 
the asymptotic relation in \eqref{conssophie}, that, upon interchanging the limit with the integral, we obtain
\begin{align*}
&\lim_{T\to \infty}\p^{(k)}_{\bo w,T,r}\left( \Delta_T(k-1)\right) \\
&\hspace{1cm}=  \frac{1}{(k-1)!}\left(\frac{2}{\zeta }\right)^{k}\int_0^\infty x^{k-1}\lim_{T\to\infty}\Q^{(k),\frac{2x}{\zeta T}\bo w}_{T,r}\paren{\Delta_T(k-1)}\frac{1}{T\mathbb{P}_{r}\paren{\ N_T\geq k}}\\
& \hspace{7cm}\times\paren{\frac{\zeta}{2}}^{k-1}\frac{k!\xi_{r}}{\paren{1+x\bm \eta\cdot \bo w }^{k+1}}\big(\bm \eta\cdot \bo w\big)^k{\rm d}x\\
&\hspace{1cm}=  k\big(\bm \eta\cdot \bo w\big)^k\int_0^\infty x^{k-1}\lim_{T\to\infty}\Q^{(k),\frac{2x}{\zeta T}\bo w}_{T,r}\paren{\Delta_T(k-1)}\frac{1}{\paren{1+x\bm \eta\cdot \bo w }^{k+1}}{\rm d}x.
\end{align*}Finally, from Proposition \ref{lemmaLimitOfJointLaw} and the change of variables $y=x\bm \eta\cdot \bo w$, we deduce
\begin{align*}
\lim_{T\to \infty}&\p^{(k)}_{\bo w,T,r}\left( \Delta_T(k-1)\right) \\
&=\frac{2^{k-1}}{(k-1)!}\left(\prod_{h=1}^{k-1}\frac{\zeta_{i_h}}{\zeta}\frac{p_{i_h}(\bm \ell_{h})w(\bm \ell_{h})}{\mathbb{E}_{i_h}\left[w(\bo L)\right]}\frac{\bm \ell_{h}^{\floor{\bo g_h}}\bm \xi^{\bo g_h}}{w(\bm \ell_{h})}{\rm d}\rho_h\right) 
\int_0^\infty \frac{y^{k-1}}{(1+y)^2}\prod_{h=1}^{k-1}\frac{1}{(1+(1-\rho_h)y)^{2}}{\rm d}y\\
&\hspace{1cm} = \lim_{T\to \infty}\p^{(k)}_{unif,T,r}\paren{\Delta_T(k-1)},
\end{align*}where the last equality follows from \eqref{dk2withcolors}.

To complete the proof, it remains to justify the validity of interchanging the limit and the integral.  Denote by $f(T,x)$ the integrand on the right-hand side of \eqref{eqnJointLawUnderPUnifWNewRescaled} and define
\[
g(T,x):=x^{k-1}\Q^{(k),\phi_T\bo w}_{\bo w,T,r}\left[\frac{\big(\bo Z_T\cdot \bo w\big)^k}{\sum_{\bo v\in \mathcal{N}^{(k)}_T}\bo w^{\bo D_{\bo v}}}\right]\frac{\E_{r}\left[e^{-\frac{2x}{\zeta T}\bo w\cdot \bo Z_T}\sum_{\bo v\in \mathcal{N}^{(k)}_T}\bo w^{\bo D_{\bo v}}\right]}{T^k\mathbb{P}_{r}\paren{\ N_T\geq k}}.
\] 
%\comentario{Here I erased the extra combinatorial term}
Next, we apply the inequality $n^k\le k^kn^{\floor{k}}$, valid for $n> k$
to deduce
\[
\bo Z^{\bo D}_T=\prod_{m=1}^d\big(Z^{(m)}_T\big)^{D_m}\leq \prod_{m=1}^dD_m^{D_m}\big(Z^{(m)}_T\big)^{\floor{D_m}}\leq k^{dk}\bo Z_T^{\floor{\bo D}}. 
\]This implies, using the identity in \eqref{eqnApproximationToMultinomialDenominatorQS}, that
\[
\big(\bo Z_T\cdot \bo w\big)^k=\sum_{\substack{\bo D\in \{0,1,\ldots,k\}^d:\\ \sum_{m=1}^dD_{m}=k}}{k \choose \bo D}\bo Z^{\bo D}_T\bo w^{\bo D}\leq k^{dk}\sum_{\substack{\bo D\in \{0,1,\ldots,k\}^d:\\ \sum_{m=1}^dD_{m}=k}}{k \choose \bo D}\bo Z^{\floor{\bo D}}_T\bo w^{\bo D}.
\]
Thus from \eqref{eqnMultinomialDenominatorQS0}, we may deduce that  $f(T,x)\leq k^{dk}g(T,x)$. 

To apply the Generalised Lebesgue Dominated Convergence Theorem, we first observe  that the  functions $f$ and $g$ differ only by the  indicator  function appearing  inside the probability measure  $\Q^{(\bo w),\frac{2x}{\zeta T}\bo w}_{T,r}$. In view of Corollary \ref{teoLimitQWEqualsLimitQk},  the asymptotic relation in  \eqref{eqnApproximationToMultinomialDenominatorQS} and Theorem \ref{lemmaLimitOfJointLaw}, it is therefore sufficient to show that the limit of $g$ exists and 
\[
\lim_{T\to \infty}\int_{\re_+}g(T,x){\rm d}x=\int_{\re_+}\lim_{T\to \infty}g(T,x){\rm d}x.
\]
By Lemmas \ref{lemmaAsymptoticGForQW} and \ref{limitunderqw}, together with the asymptotic relation in \eqref{conssophie}, we obtain
\[
\lim_{T\to \infty}g(T,x)=x^{k-1}\lim_{T\to\infty}\frac{\frac{\big(\frac{\zeta T}{2}\big)^{k-1}k!\xi_r(\bm \eta\cdot \bo w)^{k}}{(1+x\bm \eta\cdot \bo w)^{k+1}}}{T^k\frac{2\xi_r}{\zeta T}}=\Big(\frac{\zeta }{2}\Big)^{k}k!(\bm \eta\cdot \bo w)^{k}\frac{x^{k-1}}{(1+x\bm \eta\cdot \bo w)^{k+1}}.
\]
Thus using the identity in \eqref{usefulinte}, we see
\[
\int_{0}^\infty \lim_{T\to \infty}g(T,x){\rm d}x=\left(\frac{\zeta }{2}\right)^{k}k!\int_{0}^\infty (\bm\eta\cdot \bo w)^{k}\frac{x^{k-1}}{(1+x\bm \eta\cdot \bo w)^{k+1}}{\rm d}x=\left(\frac{\zeta }{2}\right)^{k}(k-1)!.
\]
On the other hand, from the proof of Theorem \ref{propofsplitting3Intro}, we know that identity \eqref{eqnfsplitting3Intro}  holds for any $A\in \mathcal{F}_T^{(k)}$. Thus taking  $A=\Omega$,  we have 
\[
1=\p^{(k)}_{\bo w,T,r}\left( \Omega\right)=\frac{1}{(k-1)!}\paren{\frac{2}{\zeta}}^k\int_0^\infty g(T,\omega){\rm d}\omega,
\]which  allow us to conclude the proof. 
\end{proof}
\begin{proof}[Proof of Theorem \ref{mainresult2}] The fact that only binary splittings occur in the limit (see \eqref{onlybin}) follows directly from Theorem\ref{teoJointLawSplittingTimeChildrenPartitionScaledBeforeTheLimit}. Alternatively, this can also be deduced from \eqref{dk2withcolors} by marginalising over types, splitting times, and offspring.

The limit in \eqref{dk2withcolors} is precisely Proposition \eqref{teo7}. The identities in \eqref{limitsigualdad} follow from Propositions \ref{prop8} and \ref{prop9}.
\end{proof}

\subsection{Proof of Theorem \ref{propoColorOfTheSpineIsMCPUnifLimit}}
The proof of Theorem \ref{propoColorOfTheSpineIsMCPUnifLimit} follows directly from the next two Lemmas. The first   establishes the  limiting joint distribution of the colours between consecutive  spine  splitting events under $\lim_{T\to \infty}\mathbb{P}^{(k)}_{unif,T,r}$ which corresponds to \eqref{inhoMarkovsplit}. 
To this end,  we combine Lemma \ref{lemmaColorOfTheSpineIsMCQLimit} with the asymptotic identity given in   equation \eqref{unifsamlimitvsint}. 

\begin{lemma}\label{lemmaColorOfTheSpineIsMCPUnifLimit}Fix $n\in \na$, $0<\rho_1<\cdots <\rho_n<1$,$ r\in [d]$ and $(i_h)_{h\in [n]}\in [d]^{n}$. For any $m\in [d]$, we let  $D_m:=\#\{h\in [n]:i_h=m \}$  and define $\bo D=(D_1,\ldots, D_d)$. 
Then, we have
\begin{align*}
&\lim_{T\to \infty} \mathbb{P}^{(k)}_{unif,T,r}\left(c(\varsigma_{\rho_1T}^{(1)})=i_1,\ldots, c(\varsigma_{\rho_nT}^{(1)})=i_n,\tau_1>\rho_nT\right)=\bm \eta^{\bo D}\bm \xi^{\bo D}\mathbb{E}\left[\left(\frac{1-\rho_n}{1-\rho_n W }\right)^{k-1}\right],
\end{align*}
where $W$ is a Beta random variable with parameters $(k,1)$.  
\end{lemma}
\begin{proof}
Let $\Gamma_T:=\{c(\varsigma_{\rho_1T}^{(1)})=i_1,\ldots, c(\varsigma_{\rho_nT}^{(1)})=i_n,\tau_1>\rho_nT \}$ denotes the event of interest. This event, combined with arguments similar to those employed in the proof of Proposition \eqref{teo7}, implies
\[
	\begin{split}
		\lim_{T\to \infty}\mathbb{P}^{(k)}_{unif,T,r}&\left(\Gamma_T\right)  = \frac{1}{(k-1)!}\frac{2}{\zeta }\\&\times\lim_{T\to\infty}\int_0^\infty\left(e^\frac{2\omega}{\zeta T}-1\right)^{k-1}\Q^{(k),\frac{2\omega}{\zeta T}\vec{1}}_{T,r}(\Gamma_T)\frac{\E_{r}\left[N^{\floor{k}}_Te^{-\frac{2\omega}{\zeta T}\vec{1}\cdot \bo Z_T}\right]}{\mathbb{P}_{r}\paren{\ N_T\geq k}}\frac{{\rm d}\omega}{T}\\
&= k\bm \eta^{\bo D}\bm \xi^{\bo D}(\bm \eta\cdot \vec{1})^{k}\int_{0}^{\infty}\omega^{k-1}\paren{1-\frac{\rho_n}{1+\omega(1-\rho_n)\bm \eta \cdot \vec{1} }}^{k-1}\frac{1}{(1+\omega\bm \eta \cdot \vec{1})^{k+1}}{\rm d}\omega\\
%& =k\bm \eta^{\bo D}\bm \xi^{\bo D}\int_{0}^{\infty}y^{k-1}\paren{1-\frac{\rho_n}{1+(1-\rho_n)y }}^{k-1}\frac{1}{(1+y)^{k+1}}dy\\
& =k\bm \eta^{\bo D}\bm \xi^{\bo D}(1-\rho_n)^{k-1}\int_{0}^{\infty}\frac{y^{k-1}}{(1+y)^2(1+(1-\rho_n)y)^{k-1}}{\rm d}y,
	\end{split}
\]
where in the second identity we have used  Lemma \ref{lemmaColorOfTheSpineIsMCQLimit} and the last identity follows from the change of variables $y=\omega\bm \eta \cdot \vec{1}$. Finally, we perform the change of variables $y=z/(1-z)$, obtaining 
\[
\begin{split}
\lim_{T\to \infty}\mathbb{P}^{(k)}_{unif,T,r}\left(\Gamma_T\right) &=\bm \eta^{\bo D}\bm \xi^{\bo D}\int_{0}^{1}\paren{\frac{1-\rho_n}{1-\rho_nz }}^{k-1}kz^{k-1}{\rm d}z =\bm \eta^{\bo D}\bm \xi^{\bo D}\mathbb{E}\left[\left(\frac{1-\rho_n}{1-\rho_n W }\right)^{k-1}\right],
	\end{split}
\]as claimed.
\end{proof}

In  \cite{AHP-p1}, Lemma 4 determines the distribution of the first time at which the spines split. Since this time corresponds to the most recent common ancestor of the sample, we now analyse its limiting behaviour, which is closely related to the previous result. The following result shows \eqref{firstsplitinglimit1}.
\begin{lemma}
For any $\rho\in (0,1)$ we have
\[
\begin{split}
\lim_{T\to \infty}\mathbb{P}^{(k)}_{unif,T,r}\left(\tau_1\geq \rho T\right) & =\mathbb{E}\left[\left(\frac{1-\rho}{1-\rho W }\right)^{k-1}\right].
\end{split}
\] 
\end{lemma}
\begin{proof}
To establish this result, one may sum over all $(i_1,\ldots, i_n)\in [d]^{n}$ in Lemma \ref{lemmaColorOfTheSpineIsMCPUnifLimit}, and use the identity $\bm{\eta}\cdot\bm{\xi}=1$. 
Alternatively, one can follow the same steps as in the proof of Proposition \ref{teo7}, applied to the event $\{c(\varsigma_T^{(1)})=i,\tau_1\geq \rho T\}$ for any $i\in [d]$, and use Corollary  \ref{corolarioantesconst} to compute the probability of this event.  
Summing over all $i$ gives the desired result. 
\end{proof}
To conclude the proof of Theorem \ref{propoColorOfTheSpineIsMCPUnifLimit}, we observe that the third identity in the statement follows directly from \eqref{inhoMarkovsplit} and \eqref{firstsplitinglimit1}. The last identity follows directly from Theorem \ref{mainresult2}.

\begin{funding}
 The first author was partially supported by the Deutsche Forschungsgemeinschaft (through grant DFG-SPP-2265). 
The second author acknowledges the support of the New Zealand Aotearoa Royal Society Te Ap\={a}rangi
Marsden Fund (22-UOA-052).
The third author was supported by the grant CF-2023-I-2566 from SECIHTI, Mexico.
\end{funding}

\end{document}